\documentclass[]{article}
\usepackage{setspace}
\usepackage{pdfpages}
\usepackage{float}
\usepackage[english]{babel}
\usepackage[utf8]{inputenc}
\usepackage{babel,csquotes,xpatch}
\usepackage[margin=1in]{geometry}
\usepackage{fancyhdr}
\usepackage{amsmath}
\usepackage{amssymb}
\usepackage{centernot}
\usepackage{mathtools}
\usepackage{amssymb}
\usepackage{tikz-cd}
\usepackage{tkz-euclide}
\usepackage[toc,page]{appendix}
\usetikzlibrary{decorations.markings}
\usepackage{esint}
\usepackage{cases}
\usepackage{url}
\usepackage{listings}
\usetikzlibrary{intersections}
\usetikzlibrary{calc}
\usepackage{amsthm}
\usepackage{dsfont}
\usepackage{makecell}
\usepackage{color} %
\definecolor{mygreen}{RGB}{28,172,0} %
\definecolor{mylilas}{RGB}{170,55,241}
\usepackage{upquote} %
\usepackage{eurosym} %
\usepackage{grffile}
\usepackage{subcaption}%
\usepackage{hyperref}
\hypersetup{
    colorlinks=true,
    citecolor=blue,
    linkcolor=blue,
    filecolor=magenta,      
    urlcolor=blue,
    pdftitle={Overleaf Example},
}
\usepackage{longtable} %
\usepackage{booktabs}  %
\usepackage[inline]{enumitem} %
\usepackage[normalem]{ulem} %
\usepackage{mathrsfs}
\usepackage[linesnumbered, ruled, lined, shortend]{algorithm2e}
\usepackage[style=trad-abbrv, backend=bibtex, doi = true, url = false, isbn = false, giveninits=true]{biblatex}
\AtEveryBibitem{\clearfield{issn}}

\usepackage{titlesec}

\titleformat*{\section}{\large\bfseries}
\titleformat*{\subsection}{\bfseries}

\title{Rigorous enclosure of Lyapunov exponents of stochastic flows}
\author{Maxime Breden\thanks{CMAP, CNRS, \'Ecole polytechnique, Institut Polytechnique de
Paris, 91120 Palaiseau, France. \href{mailto:maxime.breden@polytechnique.edu}{maxime.breden@polytechnique.edu}} $\quad$ Hugo Chu\thanks{Department of Mathematics, Imperial College London, London SW7 2AZ, United Kingdom. \href{mailto:hugo.chu17@imperial.ac.uk}{hugo.chu17@imperial.ac.uk}, \href{mailto:jsw.lamb@imperial.ac.uk}{jsw.lamb@imperial.ac.uk}, \href{mailto:m.rasmussen@imperial.ac.uk}{m.rasmussen@imperial.ac.uk}} $\quad$ Jeroen S.W. Lamb\footnotemark[2]$\;$\thanks{International Research Center for Neurointelligence, The University of Tokyo, Tokyo,113-0033, Japan}$\;$\thanks{Centre for Applied Mathematics and Bioinformatics, Department of Mathematics and
Natural Sciences, Gulf University for Science and Technology, Halwally 32093, Kuwait} $\quad$ Martin Rasmussen\footnotemark[2]}

\theoremstyle{definition}
\newtheorem{thm}{Theorem}
\newtheorem{defn}[thm]{Definition}

\newtheorem{cor}[thm]{Corollary}
\newtheorem{rmk}[thm]{Remark}
\newtheorem{lem}[thm]{Lemma}

\newtheorem{prop}[thm]{Proposition}

\newtheorem{manualtheoreminner}{\bf Claim}

\newtheorem{manualtheoreminnerb}{Step}
\newenvironment{step}[1]{%
  \renewcommand\themanualtheoreminnerb{#1}%
  \manualtheoreminnerb
}{\endmanualtheoreminner}

\newcounter{dummy}
\makeatletter
\newcommand\myitem[1][]{\item[#1]\refstepcounter{dummy}\def\@currentlabel{#1}}
\makeatother

\renewcommand{\d}{\mathrm{d}}
\newcommand{\diff}[2]{\frac{\mathrm{d}#1}{\mathrm{d}#2}}
\newcommand{\pdiff}[2]{\frac{\partial#1}{\partial#2}}
\newcommand{\Pb}{\mathbb{P}}

\newcommand{\Ind}[1]{\mathds{1}_{#1}}
\newcommand{\R}{\mathbb{R}}

\newcommand{\T}{\mathbb{T}}

\newcommand{\cL}{\mathcal{L}}
\newcommand{\tr}{\mathrm{tr}}

\renewcommand{\d}{\mathrm{d}}

\newcommand{\N}{\mathbb{N}}

\newcommand{\vect}[2]{\left(\begin{matrix}
         #1 \\
         #2 \\
\end{matrix} \right)}

\allowdisplaybreaks[1]

\begin{document}

\maketitle
\begin{abstract}
We develop a powerful and general method to provide rigorous and accurate upper and lower bounds for Lyapunov exponents of stochastic flows. Our approach is based on computer-assisted tools, the adjoint method and established results on the ergodicity of diffusion processes. We do not require any structural assumptions on the stochastic system and work under mild hypoellipticity conditions and outside of perturbative regimes. Therefore, our method allows for the treatment of systems that were so far out of reach from existing mathematical tools. We demonstrate our method to exhibit the chaotic nature of four different systems. Finally, we show the robustness of our approach by combining it with continuation methods to produce bounds on Lyapunov exponents over large parameter regions.
\end{abstract}

\begin{center}
{\bf \small Keywords} \\ \vspace{.05cm}
{\small Adjoint method $\cdot$ Lyapunov exponents $\cdot$ Computer-assisted proofs\\Stochastic differential equations $\cdot$ Poisson problem}
\end{center}

\begin{center}
{\bf \small Mathematics Subject Classification (2020)}  \\ \vspace{.05cm}
{\small 37M25 $\cdot$ 37H15 $\cdot$ 65P30 $\cdot$ 60J22 $\cdot$ 65G20} 
\end{center}



\section{Introduction}



We consider stochastic flows $(\varphi_t)_{t\geq 0}$ generated by nonlinear stochastic differential equations of the form
\begin{equation}
    \d\varphi_t(\omega, x) = X_0(\varphi_t(\omega, x)) \d t + \sum_{i=1}^\ell X_i(\varphi_t(\omega, x)) \circ \d B^i_t(\omega) \label{eqn:gen-sde}, \qquad \varphi_0(\omega, x) = x,
\end{equation}
where $X_0, X_1, \ldots, X_\ell$ denote analytic and complete vector fields on an analytic Riemannian manifold $(\mathcal{M},\langle\cdot,\cdot \rangle)$, the $B^i$'s are independent standard Brownian motions over a filtered probability space $(\Omega, \mathcal{F}, (\mathcal{F}_t)_{t\geq 0}, \mathbb{P})$ and $\circ$ denotes Stratonovich integration. A central object in the study of the dynamics of the stochastic flow $(\varphi_t)_{t\geq 0}$ is the (top) Lyapunov exponent 
\begin{equation}
     \lambda(\omega, x, v) := \lim_{t\to \infty} \frac{1}{t}\log\|D\varphi_t(\omega, x)v\|\label{eqn:ftl}.
\end{equation}
Throughout this paper, we assume the ergodicity of the Markov process $(\varphi_t)_{t\geq 0}$, in the sense that $(\varphi_t)_{t\geq 0}$ possesses a unique stationary probability measure $\mu$ on $\mathcal{M}$.
Under the existence and uniqueness of such an ergodic measure (and mild integrability conditions), it is well-established by the celebrated Osdelets multiplicative ergodic theorem~\cite{Furstenberg1960ProductsMatrices,oseledec1968multiplicative} (see also~\cite[Chapters 3 \& 4]{Arnold1998RandomSystems}) that the Lyapunov exponent $\lambda$ is well-defined and constant for $(\mathbb{Pb}\times \mu)$-almost every $(\omega, x) \in \Omega\times \mathcal{M}$, Lebesgue-almost every $v\in \mathbf{P}(T_x\mathcal{M})$, the projective space of the tangent space $T_x \mathcal{M}$.

The Lyapunov exponent $\lambda$ is an essential tool to characterise stochastic dynamics. It is an important ingredient in the proof of \emph{synchronisation when $\lambda$ is negative}~\cite{Baxendale1991StatisticalDiffeomorphisms, Flandoli2017SynchronizationNoise} and in the characterisation of various forms of \emph{chaotic dynamics when $\lambda$ is positive}~\cite{Lamb2025HorseshoesMaps, Ledrappier1988EntropyTransformations}. In particular, from~\eqref{eqn:ftl}, it can be seen that it is linked to sensitivity to initial conditions, which is a quintessential feature of chaos. Thus, obtaining rigorous bounds on the Lyapunov exponent is crucial to understanding the dynamics of $(\varphi_t)_{t\geq 0}$. This has been a long-standing fundamental problem, highlighted for instance by Kingman~\cite{Kingman1973SubadditiveTheory}: ``Pride of place among the unsolved problems of subadditive ergodic theory must go to [its] calculation''. This question has already been studied extensively in~\cite{Arnold1998RandomSystems, Blumenthal2022LyapunovMaps, Blumenthal2018LyapunovMaps, Blumenthal2017LyapunovMap,Blumenthal2022PositiveMaps} and in particular for stochastic flows~\cite{Baxendale2024LyapunovNoise,Baxendale2002LyapunovSystems,Bedrossian2022AEquations,Chemnitz2023PositiveNoise}. However, most existing results are restricted to the small noise limit $X_i\to 0$ (for $i\neq 0$) and 
to systems enjoying a particular structure, such as volume-preserving/incompressible/Hamiltonian stochastic flows~\cite{Arnold2001TheSystems,Baxendale2002LyapunovSystems}, for which establishing the positivity of the (top) Lyapunov exponent often follows from non-degeneracy conditions such as the Furstenberg criterion~\cite{Furstenberg1963NoncommutingProducts} (see~\cite{CotiZelati2024Three-dimensionalFlows} for a recent example), especially when the noise is small.




In this paper, we propose a computer-assisted method for \emph{rigorously enclosing} $\lambda$ under mild conditions, which substantially reduces the problem of rigorously computing Lyapunov exponents for stochastic flows, and is in principle applicable to a wide range of low-dimensional Markovian random dynamical systems.

While computer-assisted methods have played an increasingly important role in the theory of dynamical systems~\cite{Galias1998ComputerEquations,Kapela2021CAPD::DynSys:Systems,Lanford2017AConjectures,Mischaikow1995ChaosProof,Tucker2002AProblem,vandenBerg2015RigorousDynamics,Vytnova2025HausdorffGasket}, the problem of rigorously enclosing Lyapunov exponents has thus far remained very challenging: in the more manageable discrete-time setting, a substantial body of work resolves this problem, including~\cite{Pollicott2010MaximalProducts} for random products of matrices and~\cite{Chihara2022ExistenceMaps, Froyland2000RigorousProducts,Galatolo2020ExistenceProof} for iterated function systems (see also~\cite{Galatolo2014AnMeasures,Pollicott2023AccurateInterval,Wormell2019SpectralDynamics} in the deterministic setting). In the case of stochastic flows induced by stochastic differential equations of the type~\eqref{eqn:gen-sde}, the tackling of this problem with computer-assisted methods is more recent, with a first attempt in~\cite{Breden2023Computer-AssistedSystems} that only applied to specific SDEs conditioned to a bounded domain, and the successful treatment of a class of chaotic systems in the small noise limit in~\cite{Bedrossian2023LowerEquations}.


The approach proposed in this paper is much more general, but has the same starting point as several other studies, namely, the formulation of the top Lyapunov exponent $\lambda$ as an ergodic average via the limit~\eqref{eqn:ftl}, the so-called Furstenberg--Khasminskii formula~\cite{Carverhill1985ATheorem,Furstenberg1963NoncommutingProducts, Khasminskii1967NecessarySystems}
\begin{equation}\label{eqn:intro-FK-formula}
    \lambda = \int_{\mathbf{P}\mathcal{M}} Q\d\tilde{\mu},
\end{equation}
where the integrand $Q$ has an explicit formula in terms of the vector fields $X_0,X_1,\ldots, X_\ell$ and is the mean infinitesimal growth rate of the linear flow along the so-called \emph{projective process} $(\xi_t)_{t\geq 0} = (\varphi_t, s_t)_{t\geq 0}$ on $\mathbf{P}\mathcal{M}= \bigcup_{x\in \mathcal{M}}\mathbf{P}_x \mathcal{M}$, where

$$s_t(\omega, x, v) = \frac{D\varphi_t(\omega, x)v}{\|D\varphi_t(\omega, x)v\|} \in \mathbf{P}_{\varphi_t(\omega, x)}\mathcal{M};$$
and $\tilde{\mu}(\d \xi) = \tilde{w}(\xi)\d \xi$ is the ergodic probability measure of the Markov process $(\xi_t)_{t\geq 0}$ on $\mathbf{P}\mathcal{M}$, where $\d \xi$ denotes the (Riemannian) volume measure on $\mathbf{P}\mathcal{M}$ by abuse of notation.

Note that in general, the existence and uniqueness of such an invariant probability measure $\tilde{\mu}$ is non-trivial, but sufficient conditions for this to hold can be found in the literature, e.g. in~\cite{SanMartin1986AFlow}. It is a control-theoretic problem associated to the stochastic differential equation solved by $(\xi_t)_{t\geq 0}$
\begin{equation}
    \d \xi_t = \tilde{X}_0(\xi_t) \d t + \sum_{i=1}^\ell \tilde{X}_i(\xi_t) \circ \d B^i_t \label{eqn:proj-sde}, \qquad \xi_0 = (x, v),
\end{equation}
with vector fields $\tilde{X}_0, \tilde{X_1}, \ldots, \tilde{X}_\ell$ on $\mathbf{P}\mathcal{M}$.

In all but a few cases (e.g.~\cite{Baxendale1986AsymptoticDiffeomorphisms, Engel2019BifurcationCycle}), obtaining bounds or even only the sign for $\lambda$ via the Furstenberg--Khashminskii formula~\eqref{eqn:intro-FK-formula} has so far proven extremely difficult. A common strategy (in both the deterministic and random settings~\cite{Breden2023Computer-AssistedSystems,Froyland2000RigorousProducts,Galatolo2020ExistenceProof, Galatolo2014AnMeasures,Wormell2019SpectralDynamics}) is to try to enclose $\tilde{w}$ rigorously and then integrate $Q$ against $\tilde{w}$ to obtain bounds on $\lambda$ via~\eqref{eqn:intro-FK-formula}. In the context of stochastic flows, we could try to make use of computer-assisted proof methods~\cite{Gomez-Serrano2019Computer-assistedSurvey,Nakao2019NumericalEquations,Rump2010VerificationArithmetic,vandenBerg2015RigorousDynamics} to solve the stationary Fokker--Planck equation 
\begin{equation}
\label{eq:stationaryFK}
    \tilde{\mathcal{L}}^*\tilde{w} = 0,
\end{equation}
where $\tilde{\mathcal{L}}^*$ denotes the adjoint in $L^2(\mathbf{P}\mathcal{M})$ of
$$\tilde{\mathcal{L}} = \tilde{X}_0 +\frac{1}{2}\sum_{i=1}^\ell\tilde{X}_i^2,$$
which is the infinitesimal generator of the process $(\xi_t)_{t\geq 0}$ written in H\"ormander form. Such computer-assisted proof, if successful, would yield a precise and quantitative description of $\tilde{\omega}$, which would then allow to accurately enclose the integral~\eqref{eqn:intro-FK-formula} (see, e.g.,~\cite{Breden2023Computer-AssistedSystems}).  However, the rigorous general treatment of general partial differential equations (PDEs) of the form~\eqref{eq:stationaryFK} by computer-assisted means is for now far out of reach due to the typically \emph{non-elliptic} nature of $\tilde{\mathcal{L}}^*$ (and $\tilde{\mathcal{L}}$). The gap to bridge is even greater when considering problems on an unbounded state space $\mathcal{M}$~\cite{Breden2023Computer-AssistedSystems}. This difficulty stems from the fact that both these aspects usually prevent the derivation of explicit a priori bounds to quantify errors on the computation of $\tilde{w}$. While there exist techniques~\cite{Arnold1984ASystems,Glynn1996AEquation,Pardoux2005On3} to prove the existence of a spectral gap for $\tilde{\mathcal{L}}^{*}$ (or $\tilde{\mathcal{L}}$), these are typically non-constructive, and thus do not provide an explicit bound on this spectral gap. Even when explicit bounds are available, for instance in some elliptic cases, they can be very intricate and lead to unusable values in practice, typically depending exponentially on $\|\tilde{X}_0\|$~\cite{bogachev_fokkerplanckkolmogorov_2015, Bogachev2018TheDiffusions}. We will therefore depart from the usual strategy of estimating $\tilde{w}$ and develop a more flexible approach.


Nonetheless, in certain special cases, some estimates for \eqref{eqn:intro-FK-formula} can be achieved analytically in an asymptotic parameter regime e.g.~as $X_i\to 0, \, i\neq 0$~\cite{Bedrossian2022AEquations,Bedrossian2023LowerEquations, Chemnitz2023PositiveNoise}. In such a regime, one can, for instance, make use of the so-called \emph{adjoint method} proposed by Arnold, Papanicolaou and Wihstutz~\cite{Arnold1986AsymptoticApplications} (see~\cite{Baxendale2024LyapunovNoise, Baxendale2024Almost-SureSystem} for some recent examples).

In this paper, we combine the adjoint method with computer-assisted techniques
to obtain rigorous and tight bounds on Lyapunov exponents for general systems: we do not assume any asymptotic or perturbative regime, assuming only mild hypoellipticity conditions. Our method, which we now introduce in Section~\ref{sec:method}, is applicable to unbounded domains and does not necessitate any specific structure on~\eqref{eqn:gen-sde}. We showcase this in Section~\ref{sec:results}, which contains four examples of stochastic flows without specific structure or asymptotically small parameters, for which we establish the positivity of the (top) Lyapunov exponent. 
The first three examples feature the phenomenon of noise-induced chaos, i.e.~a transition from negative to positive Lyapunov exponent with increasing noise level (although the Lyapunov exponent may become negative again if the noise is too large), and the last two answer open conjectures.

\subsection{Our method to rigorously enclose Lyapunov exponents}
\label{sec:method}


As already mentioned, working directly with the Furstenberg--Khashminskii formula~\eqref{eqn:intro-FK-formula} is often very difficult. The idea of the adjoint method introduced in~\cite{Arnold1986AsymptoticApplications} is to instead consider, for a well chosen $u$,
\begin{equation}\label{eqn:intro-adj}
    \int_{\mathbf{P}\mathcal{M}} \left(Q-\tilde{\cL}u\right)\d\tilde{\mu}.
\end{equation}
Indeed, since $\tilde{\mathcal{L}}^*\tilde{w} = 0$, we formally get
\begin{equation}
    \int_{\mathbf{P}\mathcal{M}} \tilde{\cL}u \d\tilde{\mu} = \int_{\mathbf{P}\mathcal{M}} \tilde{\mathcal{L}}u(\xi)\tilde{w}(\xi) \d \xi = \int_{\mathbf{P}\mathcal{M}} u(\xi)(\tilde{\mathcal{L}}^*\tilde{w}(\xi)) \d \xi = 0,
\end{equation}
and this calculation can be made rigorous under suitable integrability and growth assumptions on $u$, see for instance Lemma~\ref{lem:Bax}, which is based on~\cite[Proposition B.1]{Baxendale2024LyapunovNoise}. Therefore, the (top) Lyapunov exponent $\lambda$ is also given by~\eqref{eqn:intro-adj}, that is
\begin{equation}\label{eqn:intro-adj-2}
    \lambda =  \int_{\mathbf{P}\mathcal{M}} \left(Q-\tilde{\cL}u\right)\d\tilde{\mu},
\end{equation}
and one can try to pick a function $u$ for which the integrand $(Q - \tilde{\cL}u)$ is more approachable so that the right-hand side of~\eqref{eqn:intro-adj-2} is easier to work with than the original Furstenberg--Khashminskii formula~\eqref{eqn:intro-FK-formula}. 
Ideally, one would like to find a function $u$ and a constant $\Lambda$ such that $Q-\tilde{\cL}u=\Lambda$, i.e., such that $u$ and $\Lambda$ solve the \emph{Poisson equation}
\begin{equation}\label{eqn:intro-gen-Poisson}
    \tilde{\mathcal{L}}u = Q - \Lambda,
\end{equation} 
because then we simply get $\lambda = \Lambda$ from~\eqref{eqn:intro-adj-2}. However, solving the Poisson equation~\eqref{eqn:intro-gen-Poisson} is not necessarily easier than solving the stationary Fokker--Planck equation~\eqref{eq:stationaryFK}, and finding an explicit solution is in general not possible. 

For some specific problems, in an asymptotic regime, ad hoc choices of $u$ that make~\eqref{eqn:intro-adj-2} more manageable than~\eqref{eqn:intro-FK-formula} can sometimes be found, and we again refer to~\cite{Baxendale2024LyapunovNoise} for a recent example where this idea is used several times. In contrast, in this work we propose a generic and robust way of using the adjoint method on a broad class of systems, which leads to quantitative and accurate estimates on the Lyapunov exponent. 

The main idea is that, even if the ``optimal'' $u$ (the one solving~\eqref{eqn:intro-gen-Poisson}) cannot be found explicitly, for many problems we can still find \emph{numerically} an approximate solution $\bar{u}$, and then use this approximate solution in~\eqref{eqn:intro-adj-2} in order to derive estimates on $\lambda$, which are going to be sharp if $\bar{u}$ is an accurate approximate solution. Suppose for the moment that we have an approximation $\bar{\lambda}$ of $\lambda$ (e.g.,~via a Monte-Carlo method~\cite{Grorud1996ApproximationEquations, Kloeden1992NumericalEquations, Talay1999TheEquations}), and that we have obtained an approximate solution $\bar{u}$ to the Poisson problem~\eqref{eqn:intro-gen-Poisson}, by solving numerically the equation
$$\tilde{\mathcal{L}}\bar{u} \approx Q - \bar{\lambda}.$$
How to find suitable $\bar\lambda$ and $\bar{u}$ in practice will be discussed in Section~\ref{sec:cap}. Setting
$$\bar{Q} \overset{\mathrm{def}}{=}\tilde{\mathcal{L}}\bar{u} + \bar{\lambda},$$
 we expect to have $\bar{Q} \approx Q$. Furthermore, using $\bar{u}$ in~\eqref{eqn:intro-adj-2} we get
 \begin{align*}
     \lambda = \int_{\mathbf{P}\mathcal{M}} \left(Q-\tilde{\cL}\bar{u}\right)\d\tilde{\mu} = \int_{\mathbf{P}\mathcal{M}} \left(Q-\bar{Q}\right)\d\tilde{\mu} + \bar{\lambda},
 \end{align*}
 hence
 \begin{equation}\label{eqn:lambda-Q}
    |\lambda -\bar{\lambda}| = \left|\int_{\mathbf{P}\mathcal{M}}(Q-\bar{Q})\d \tilde{\mu} \right|.
\end{equation}
That is, we control the difference between the approximate Lyapunov exponent $\bar\lambda$ (which we know explicitly) and the exact one $\lambda$ (unknown), via the difference $Q-\bar{Q}$ (where both terms are known).
Thus, the problem of enclosing $\lambda$ has been reduced to the easier problem of bounding the integral in~\eqref{eqn:lambda-Q}: while $\tilde{\mu}$ is still unknown, various inequalities can be applied to show that this integral is small. The derivation of such a bound mostly depends on the nature of $Q$ (and $\bar{Q}$) and $\mathcal{M}$ and on a priori estimates on $\tilde{\mu}$. In this paper, our strategy is mainly based on directly applying $L^{\infty}$ and $L^{1}$ bounds globally. For instance, on a compact state space $\mathcal{M}$, the functions $Q$ and $\bar{Q}$ are typically $L^{\infty}$ and
\begin{equation}\label{eqn:W=1-est}
|\lambda -\bar{\lambda}| = \left|\int_{\mathbf{P}\mathcal{M}}(Q-\bar{Q})\d \tilde{\mu} \right|\leq \|Q-\bar{Q}\|_{\infty} \int_{\mathbf{P}\mathcal{M}}\d \tilde{\mu} = \|Q-\bar{Q}\|_{\infty}.
\end{equation}

Note that $\bar{Q}$ is known explicitly in terms of $\bar{u}$ and $\bar\lambda$, therefore there is in principle no difficulty in getting an explicit upper bound for $\|Q-\bar{Q}\|_{\infty}$. In practice, this can be done efficiently using rigorous numerics and interval arithmetic~\cite{Tuc11}, see Proposition~\ref{prop:method} for more details. We emphasize that $\bar{u}$ and $\bar\lambda$ can and should be obtained using traditional numerics with floating-point arithmetic. The only steps that have to be carried out rigorously are the calculation of $\bar{Q}$, i.e., the application of $\tilde{\mathcal{L}}$ to $\bar{u}$, and then the estimation of $\|Q-\bar{Q}\|_{\infty}$. In contrast, even in specific situations in which one could have rigorously solved the stationary Fokker--Planck equation~\eqref{eq:stationaryFK} or the Poisson equation~\eqref{eqn:intro-gen-Poisson} using a computer-assisted proof, this would be a computer-assisted proof for solving a PDE, for which the amount of rigorous numerics required would be much higher compared to the computer-assisted approach proposed here (see Remark~\ref{rmk:mat-vec} for more details). We also point out that our method does not only apply to the computation of top Lyapunov exponents but to any ergodic average of the form~\eqref{eqn:intro-FK-formula} where $Q$ is given and explicit.

Estimate~\eqref{eqn:W=1-est} is in some sense rather crude, as it does not seem to use any information on $\tilde{\mu}$, but this is in fact a strength of our approach. Indeed, obtaining quantitative information on $\tilde{\mu}$ can often be very hard, and being able to rigorously enclose $\lambda$ without knowing anything about $\tilde{\mu}$ is part of what makes our method broadly applicable. Of course, if some information on $\tilde{\mu}$ is available then it may be leveraged to obtain sharper estimates (see for instance Section~\ref{sec:hopf} and~\ref{sec:Duffing}), but directly using the simple estimate~\eqref{eqn:W=1-est} can already lead to very precise results such as the one in Theorem~\ref{thm:intro-cellular} below.

The above estimate~\eqref{eqn:W=1-est} can be generalised to treat systems on a non-compact state space $\mathcal{M}$, where $Q$ or $\bar{Q}$ may not be in $L^{\infty}$, and one instead makes use of Foster--Lyapunov inequalities~\cite{Canizo2023Harris-typeSemigroups, Hairer2021ConvergenceProcesses, Meyn1993StabilityProcesses} to recover the necessary estimates (see Section~\ref{sec:cap_framework}). Examples with both bounded and unbounded state spaces are provided below.

\subsection{Examples of results obtained with our method}
\label{sec:results}

We illustrate the power of our method by providing sharp quantitative bounds on the Lyapunov exponent $\lambda$ for several SDE examples, well beyond the previous state-of-the-art. More details on each of these systems, together with the proofs of the theorems, are given in Sections~\ref{sec:cellular} to~\ref{sec:Duffing}.

\begin{thm}\label{thm:intro-cellular}
    Consider the cellular flow with sinks $(\varphi_t)_{t \geq 0}$ on $\mathcal{M} = \mathbb{T}^2$ generated by the stochastic differential equation
    \begin{equation}\label{eqn:intro-cellular}
    \begin{cases}
        \d x_t &= (\cos(x_t)/2-\cos y_t)\sin x_t \d t+\sigma\d B^1_t,\\
        \d y_t &= (\cos(y_t)/2+\cos x_t)\sin y_t \d t+\sigma\d B^2_t.
    \end{cases}
    \end{equation}
    Then, for $\sigma = \sqrt{2}$, the following bounds hold for the (top) Lyapunov exponent $\lambda$ defined by~\eqref{eqn:ftl}:
    \begin{equation}
    \lambda = 0.0558453099857 \pm 10^{-13}>0.\label{eqn:intro-lambda-cellular}
    \end{equation}
\end{thm}
\noindent Here and everywhere else in the paper, expressions like $\lambda = x \pm r$ mean that $\lambda \in [x-r,x+r]$.
\begin{figure}[ht]
\captionsetup[subfigure]{justification=centering}
\begin{subfigure}{0.5\textwidth}
\includegraphics[height=0.8\linewidth]{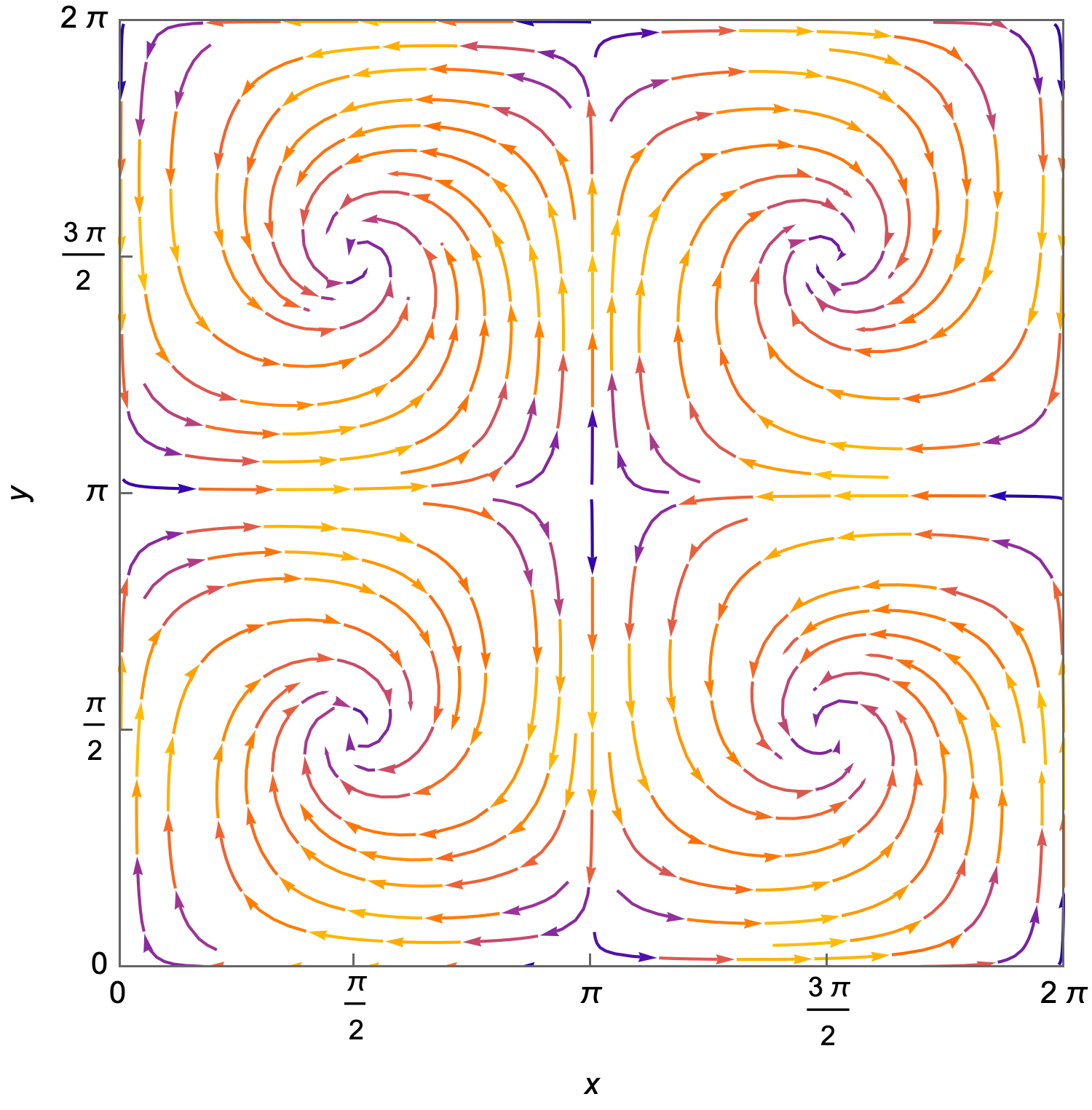} 
\caption{Phase portrait of the vector field~\eqref{eqn:intro-cellular}\\ with $\sigma = 0$.}
\label{fig:cellular-a}
\end{subfigure}
\begin{subfigure}{0.5\textwidth}
\includegraphics[height=0.8\linewidth]{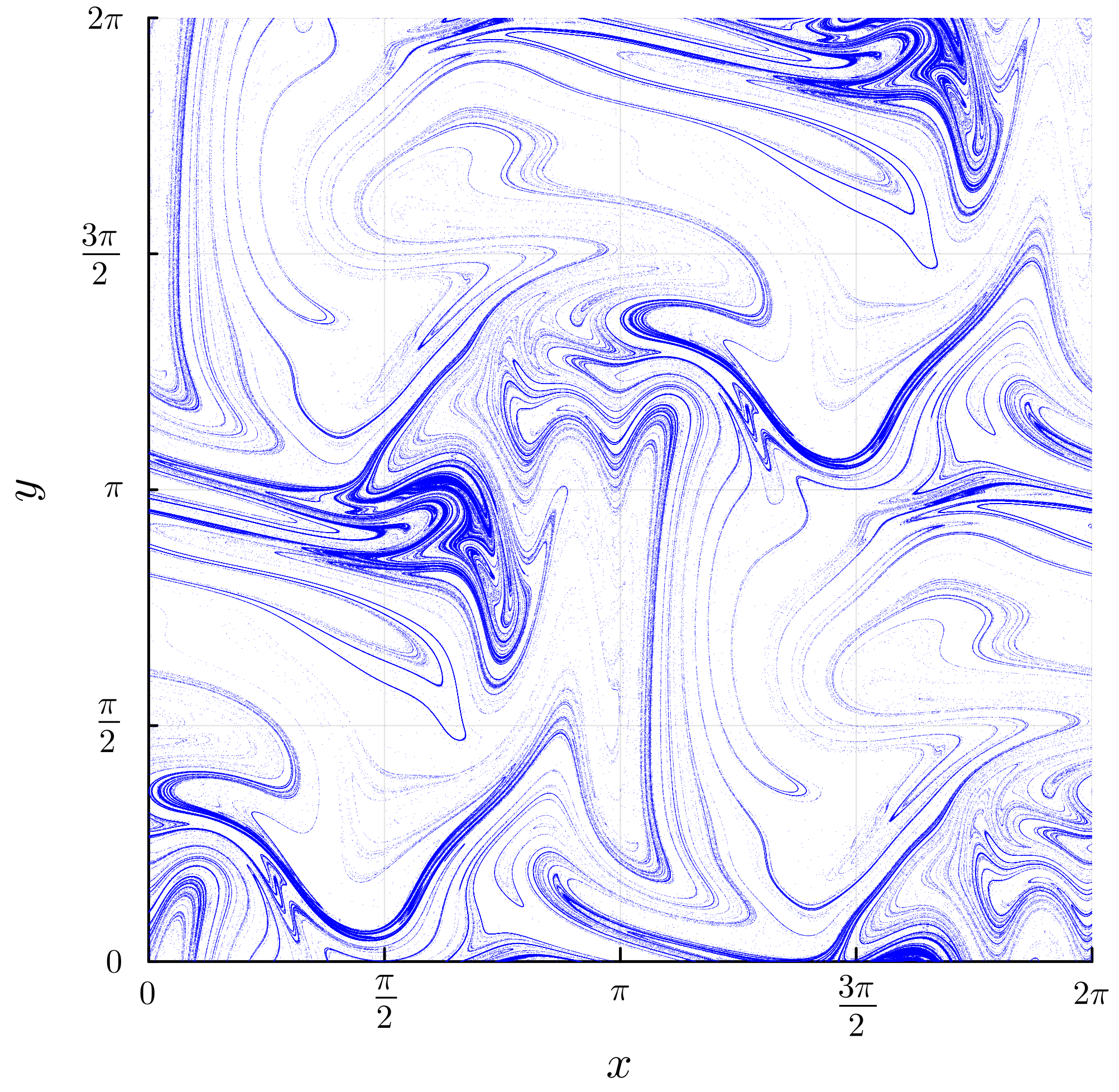} 
\caption{Snapshot of a chaotic random attractor\\ of~\eqref{eqn:intro-cellular} with $\sigma = \sqrt{2}$.}
\label{fig:cellular-b}
\end{subfigure}
\caption{Deterministic and random dynamics of~\eqref{eqn:intro-cellular}.}
\end{figure}
The positivity of the Lyapunov exponent has strong implications on the dynamics of $(\varphi_t)_{t\geq 0}$.
\begin{cor}\label{cor:chaos}
    For $\sigma = \sqrt{2}$, system~\eqref{eqn:intro-cellular} generates a random strange attractor and a random weak horseshoe.
\end{cor}
\begin{proof}
Since the stochastic flow $(\varphi_t)_{t\geq 0}$ generated by~\eqref{eqn:intro-cellular} induces a family of independently composed random diffeomorphisms of the compact space $\mathcal{M} = \mathbb{T}^2$, the positivity of the Lyapunov exponent implies the existence of a random strange attractor~\cite{Ledrappier1988EntropyTransformations}. Furthermore, applying the results of~\cite{Huang2017EntropySystems} also shows the existence of a random weak horseshoe. We refer to these two references for precise definitions of these objects.
\end{proof}
Corollary~\ref{cor:chaos} is a typical example showcasing the complementarity of computer-assisted proofs with more traditional mathematical tools. Indeed, computer-assisted proofs enable the usage of powerful theoretical results whose assumptions
(in that case, the assumption that the system has a positive Lyapunov exponent) are
often very challenging to check in practice otherwise.

While Theorem~\ref{thm:intro-cellular} is about an SDE on a compact state space, our approach can also be used to enclose Lyapunov exponents of SDEs posed on unbounded state spaces, as illustrated by the next examples. 

\begin{thm}\label{thm:intro-pendulum}
Consider the stochastic flow $(\varphi_t)_{t\geq 0}$ on $\mathcal{M} = \mathbb{T} \times \mathbb{R}$ generated by the randomly forced pendulum equation
\begin{equation}\label{eqn:intro-pendulum}
    \begin{cases}
        \d x_t &= y_t\d t,\\
        \d y_t &= -(\kappa\sin x_t +\gamma y_t)\d t +\sigma \d B_t.
    \end{cases}
\end{equation}
    Then, with gravitational constant $\kappa = 2/3$, friction coefficient $\gamma = 1/4$ and noise strength $\sigma = 4$, the Lyapunov exponent $\lambda$ is positive with bounds
    \begin{equation*}
    \lambda = 0.0271763 \pm 4.29\times 10^{-3}>0.
    \end{equation*}
\end{thm}
\begin{figure}[h]
\captionsetup[subfigure]{justification=centering}
\begin{subfigure}{0.5\textwidth}
\includegraphics[height=0.8\linewidth]{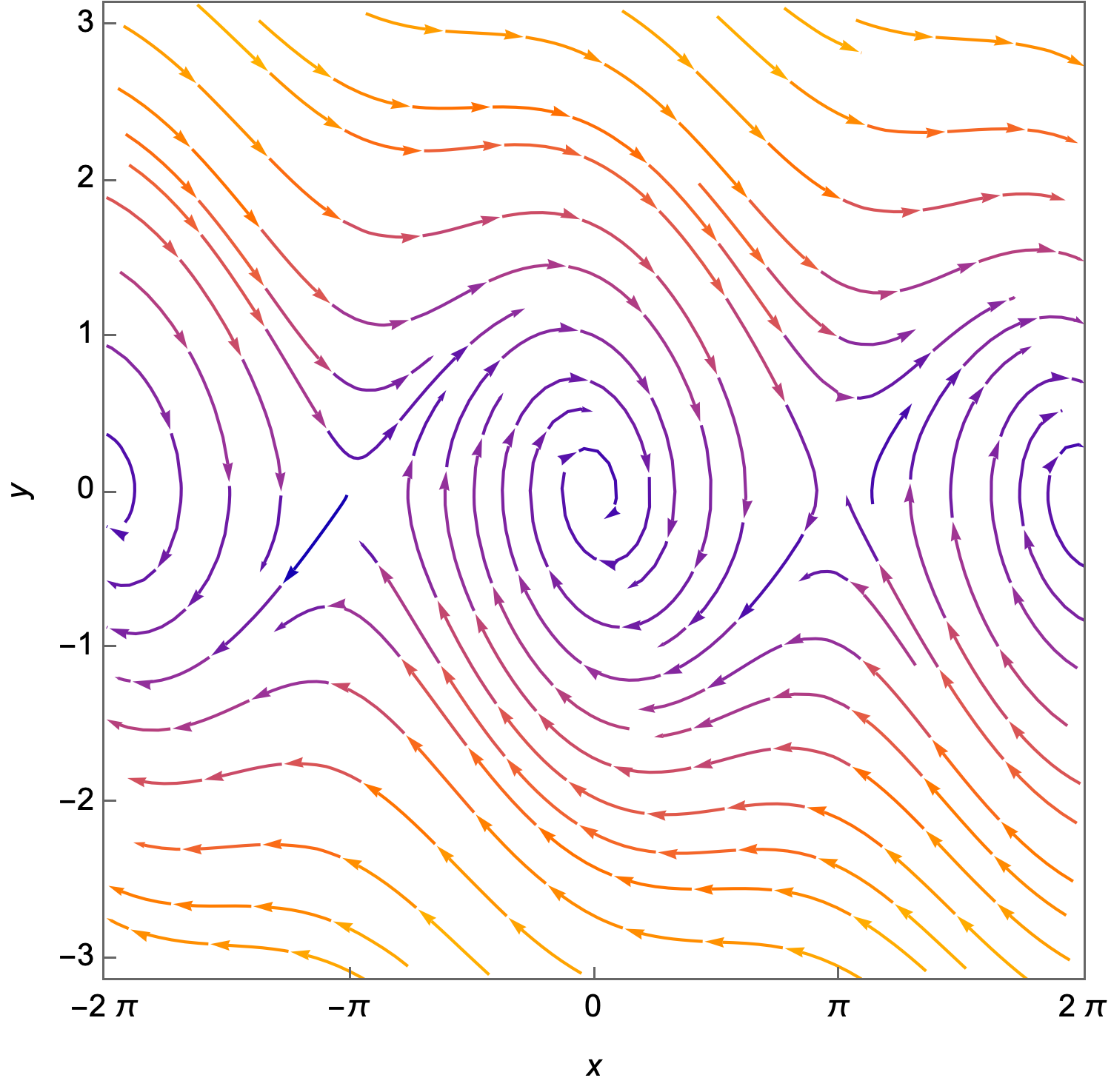} 
\caption{Phase portrait of the vector field~\eqref{eqn:intro-pendulum}\\ with $\sigma = 0$.}
\label{fig:pendulum-a}
\end{subfigure}
\begin{subfigure}{0.5\textwidth}
\includegraphics[height=0.8\linewidth]{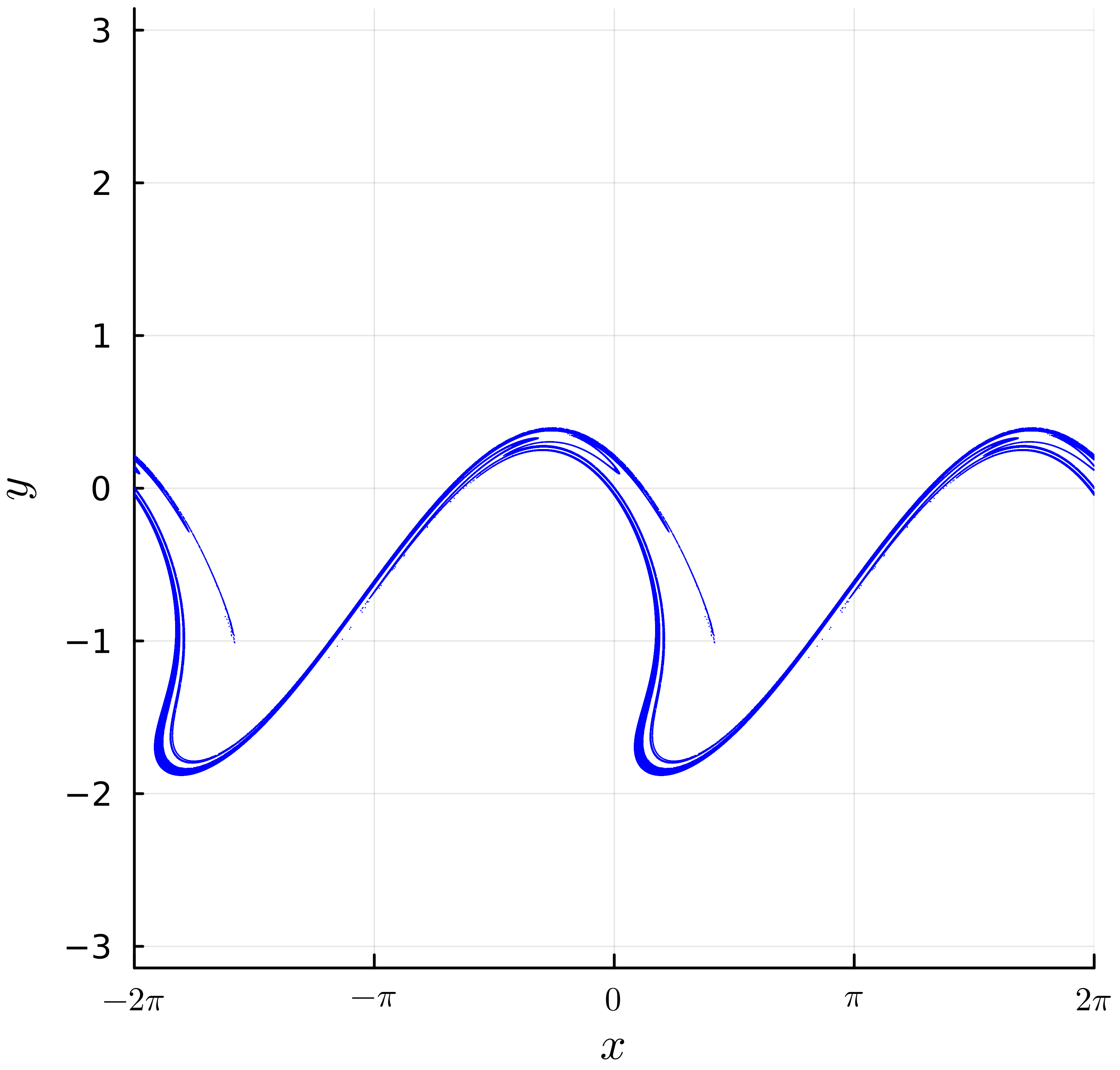} 
\caption{Snapshot of a chaotic random attractor\\ of~\eqref{eqn:intro-pendulum} with $\sigma = 4$.}
\label{fig:pendulum-b}
\end{subfigure}
\caption{Deterministic and random dynamics of~\eqref{eqn:intro-pendulum} for $\kappa = 2/3$ and $\gamma = 1/4$.}
\end{figure}

For system~\eqref{eqn:intro-pendulum}, proving that there is a positive Lyapunov exponent seems far out of reach without the use of our method. In particular, there are no clear mechanisms explaining the chaotic dynamics shown in Figure~\ref{fig:pendulum-b}. Moreover, the operators $\tilde{\mathcal{L}}$ and $\tilde{\mathcal{L}}^*$ are very much non-elliptic, since there is no noise acting on the $x$ variable, and also on the projective variable $s$ at all points of $\mathbf{P}\mathcal{M}$ (see Section~\ref{sec:pendulum} for the explicit formula). Due to this lack of ellipticity, one cannot use usual computer-assisted proofs to rigorously solve the Fokker--Planck equation~\eqref{eq:stationaryFK} or the Poisson equation~\eqref{eqn:intro-gen-Poisson}. The strategy based on the adjoint method proposed in this paper bypasses these obstructions, since it only relies on an approximate solution of the Poisson equation. Nonetheless, the lack of ellipticity still impacts the output of our method, as it makes it more challenging in practice to find an accurate approximate solution $\bar{u}$, which partially explains why the enclosure of $\lambda$ in Theorem~\ref{thm:intro-pendulum} is not as sharp as the one in Theorem~\ref{thm:intro-cellular}. 

Next, we show that our method is robust and can be combined with continuation methods recently introduced in~\cite{Breden2023AExpansions} (see also~\cite{Arioli2021UniquenessConditions}) to enclose Lyapunov exponents for a large parameter range.

\begin{thm}\label{thm:intro-Hopf}
    Consider the stochastic flow $(\varphi_t)_{t\geq 0}$ on $\mathcal{M} = \mathbb{R}^2$ generated by the Hopf normal form with additive noise
    \begin{equation}\label{eqn:intro-Hopf}
        \d\vect{x_t}{y_t} = \left[\vect{\alpha & -\beta}{\beta &\alpha}\vect{x_t}{y_t} - \vect{a & b}{-b & a}\vect{x_t}{y_t}(x_t^2+y_t^2)\right]\d t + \sigma \d \vect{B^1_t}{B^2_t}, \qquad a>0.
    \end{equation}
    Let $\beta \in \mathbb{R}$ and fix $a=\alpha = 4$, $\sigma = \sqrt{2}$. Let $\lambda_b$ denote the Lyapunov exponent of this system for a shear parameter $b$, and $b \mapsto \bar{\lambda}_b$ be the function represented in Figure~\ref{fig:hopf-a} (and whose precise description can be found at~\cite{Huggzz/Enclosure-of-Lyapunov-exponents}). Then, for \emph{all} $b\in [0,30]$
    $$|\lambda_b - \bar{\lambda}_b|\leq 3.41\times 10^{-4}.$$
\end{thm}

\begin{cor}\label{cor:intro-Hopf}
    Consider $\lambda_b$ as in Theorem~\ref{thm:intro-Hopf}, then there exists $b^*\in (21.5323, 21.5381)$ such that $\lambda_{b^*} = 0$. Furthermore, $\lambda_b<0$ for $b\in[0,21.5322]$ and $\lambda_b>0$ for $b\in[21.5381,30]$.
\end{cor}

\begin{figure}[h]
\captionsetup[subfigure]{justification=centering}
\begin{subfigure}{0.5\textwidth}
\includegraphics[height=0.65\linewidth]{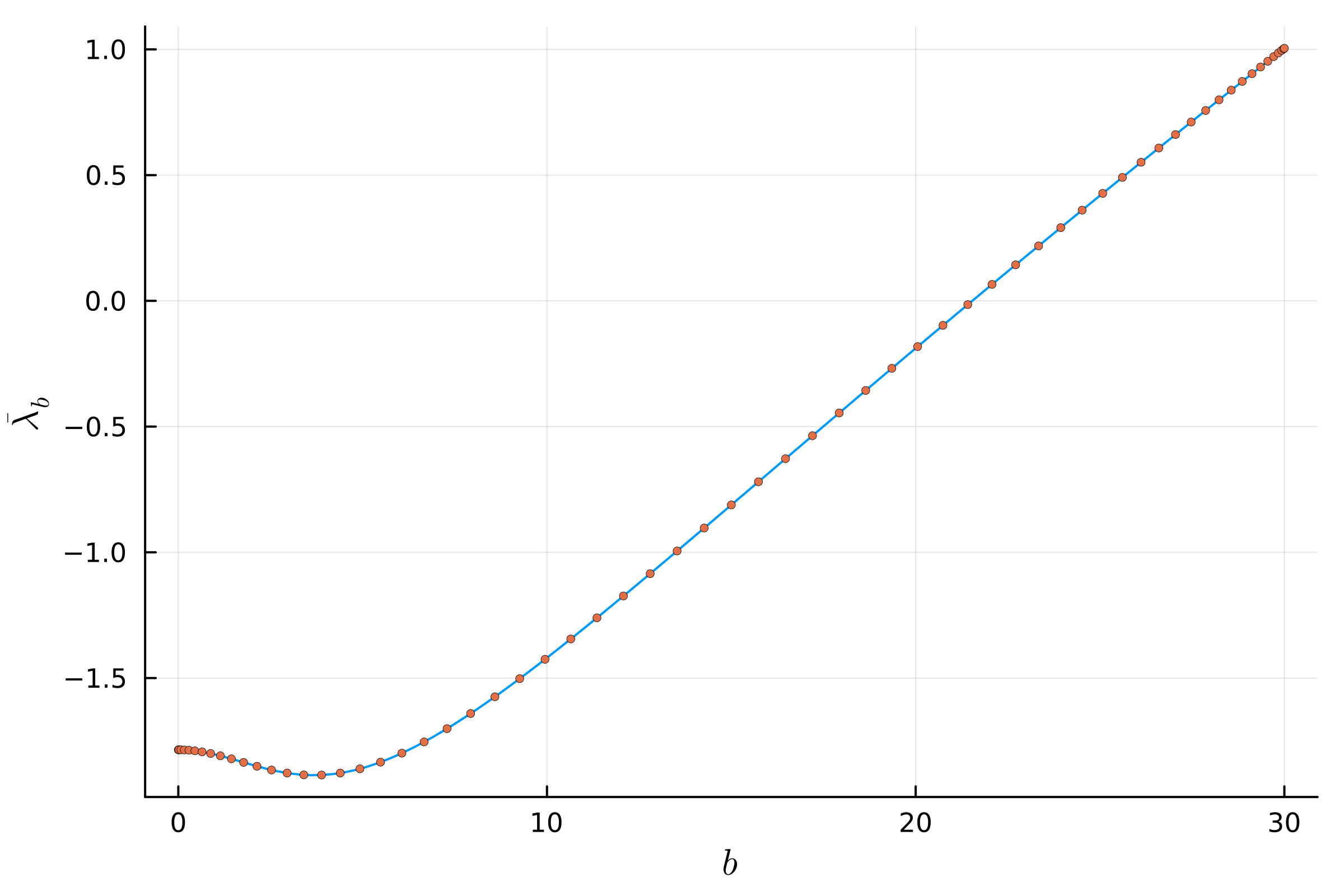} 
\caption{Graph of the function $b\mapsto \bar{\lambda}_b$; the orange\\ points are sampled at Chebyshev nodes.}
\label{fig:hopf-a}
\end{subfigure}
\begin{subfigure}{0.5\textwidth}
\includegraphics[height=0.65\linewidth]{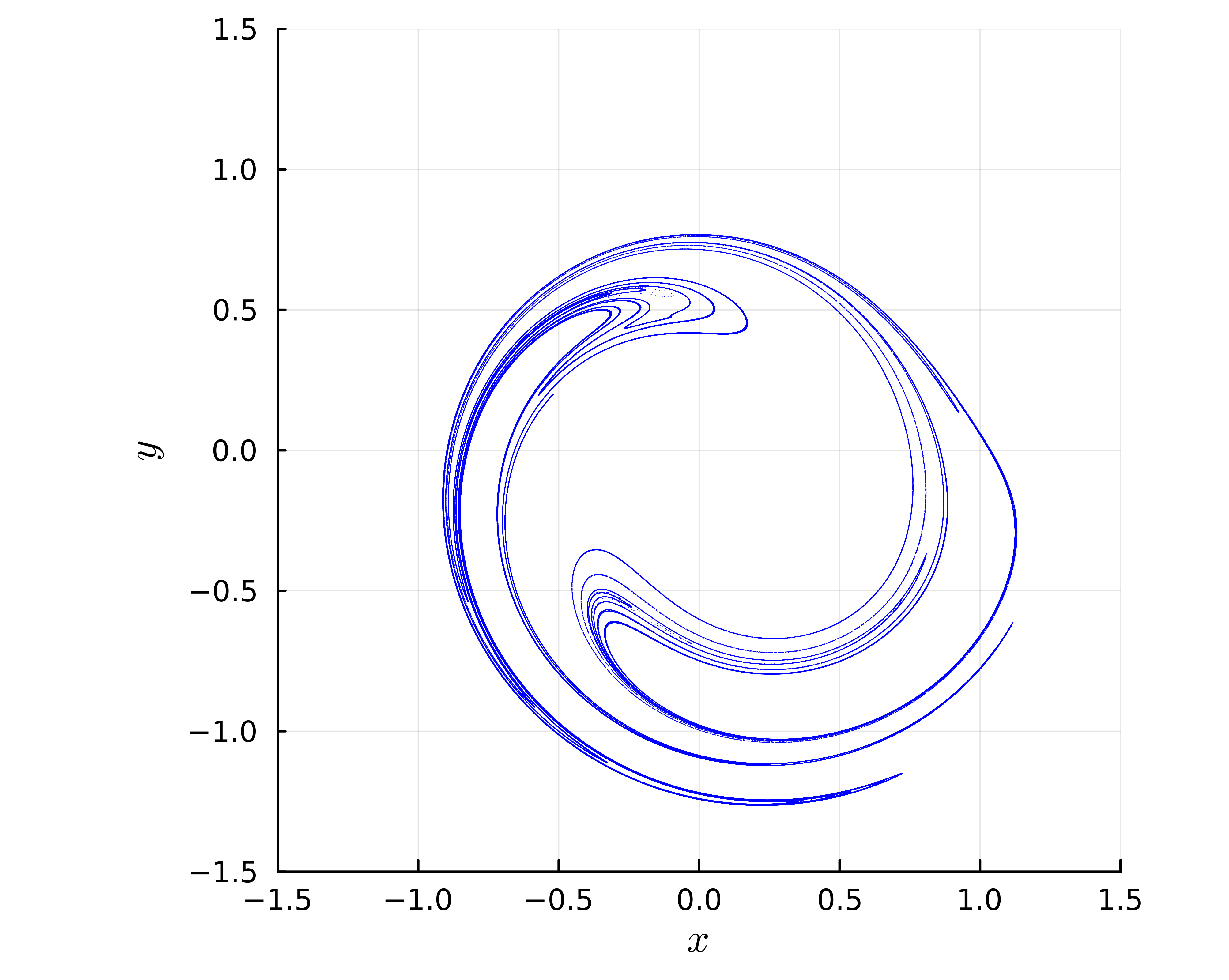} 
\caption{Snapshot of a chaotic random\\ attractor of~\eqref{eqn:intro-Hopf} with $b = 21.5381$.}
\label{fig:hopf-b}
\end{subfigure}

\caption{Quantitative and qualitative random dynamics induced by~\eqref{eqn:intro-Hopf} with $a = \alpha = 4$ and $\sigma = \sqrt{2}$.}
\end{figure}


The study of the Lyapunov exponent for system~\eqref{eqn:intro-Hopf} was initiated in~\cite{DeVille2011StabilitySystem}, where it was numerically conjectured that the Lyapunov exponent $\lambda$ was positive for $b$ large enough. After several attempts pointing to this result~\cite{Breden2023Computer-AssistedSystems, Doan2018HopfNoise,Engel2019BifurcationCycle}, the conjecture was recently proved in the asymptotic regime $b\to\infty$ in~\cite{Baxendale2024LyapunovNoise, Chemnitz2023PositiveNoise}. Theorem~\ref{thm:intro-Hopf} complements these results by covering the non-asymptotic regime. In particular, we obtain a precise estimate of the shear parameter value $b^*$ at which the transition from negative to positive Lyapunov exponent occurs.



As a final example, we consider the stochastic differential equation~\eqref{eqn:gen-sde} on $\mathcal{M}= \mathbb{R}^2\backslash\{0\}$ with the vector fields
\begin{equation}\label{eqn:bax_vfs}
\begin{aligned}X_0(x, y) = \left(\frac{\beta}{2}-\frac{b}{8}(x^2+y^2)\right)\vect{x}{y}-\frac{3a}{8\nu}(x^2+y^2)\vect{y}{-x}&,
\\X_1(x,y) = \frac{\sigma}{2\sqrt{2}\nu}\vect{x}{-y},\qquad X_2(x,y) = \frac{\sigma}{2\sqrt{2}\nu}\vect{y}{x},\qquad X_3(x,y) =& \frac{\sigma}{2\nu}\vect{y}{-x},
\end{aligned}
\end{equation}
with $-\sigma^2/4\nu^2<\beta\leq 0$, $a\geq 0$ and $b, \nu,\sigma>0$. This system appears as a result of a stochastic averaging procedure for the stochastic Duffing--van der Pol equation on $\mathcal{M}= \mathbb{R}^2\backslash\{0\}$
\begin{equation}\label{eqn:Duffing}
    \begin{cases}
        \d x_t &= -\nu y_t\d t\\
        \d y_t &=\left(\nu x_t+\varepsilon^2\beta y_t+\dfrac{a}{\nu} x^3_t-bx^2_ty_t\right)\d t -\dfrac{\varepsilon\sigma}{\nu}x_t\d B_t
    \end{cases}
\end{equation}
as $\varepsilon\to 0$. Baxendale~\cite{Baxendale2004StochasticEquation} conjectures the following for the Lyapunov exponent $\lambda_{\beta}$ associated to the stochastic differential equation~\eqref{eqn:gen-sde} with vector fields~\eqref{eqn:bax_vfs}: if $\sigma/\nu = 1$ and $b\nu/a = 1.24$, there exists $\beta^*\approx -0.17$ such that $\lambda_{\beta}<0$ for $-\sigma^2/4\nu^2 = -1/4<\beta<\beta^*$ and $\lambda_{\beta}>0$ for $\beta^*<\beta<0$.  We provide an answer to this conjecture by proving that a transition from negative to positive Lyapunov exponent indeed occurs as $\beta$ increases.
\begin{figure}[h]
\captionsetup[subfigure]{justification=centering}
\begin{subfigure}{0.5\textwidth}
\includegraphics[height=0.65\linewidth]{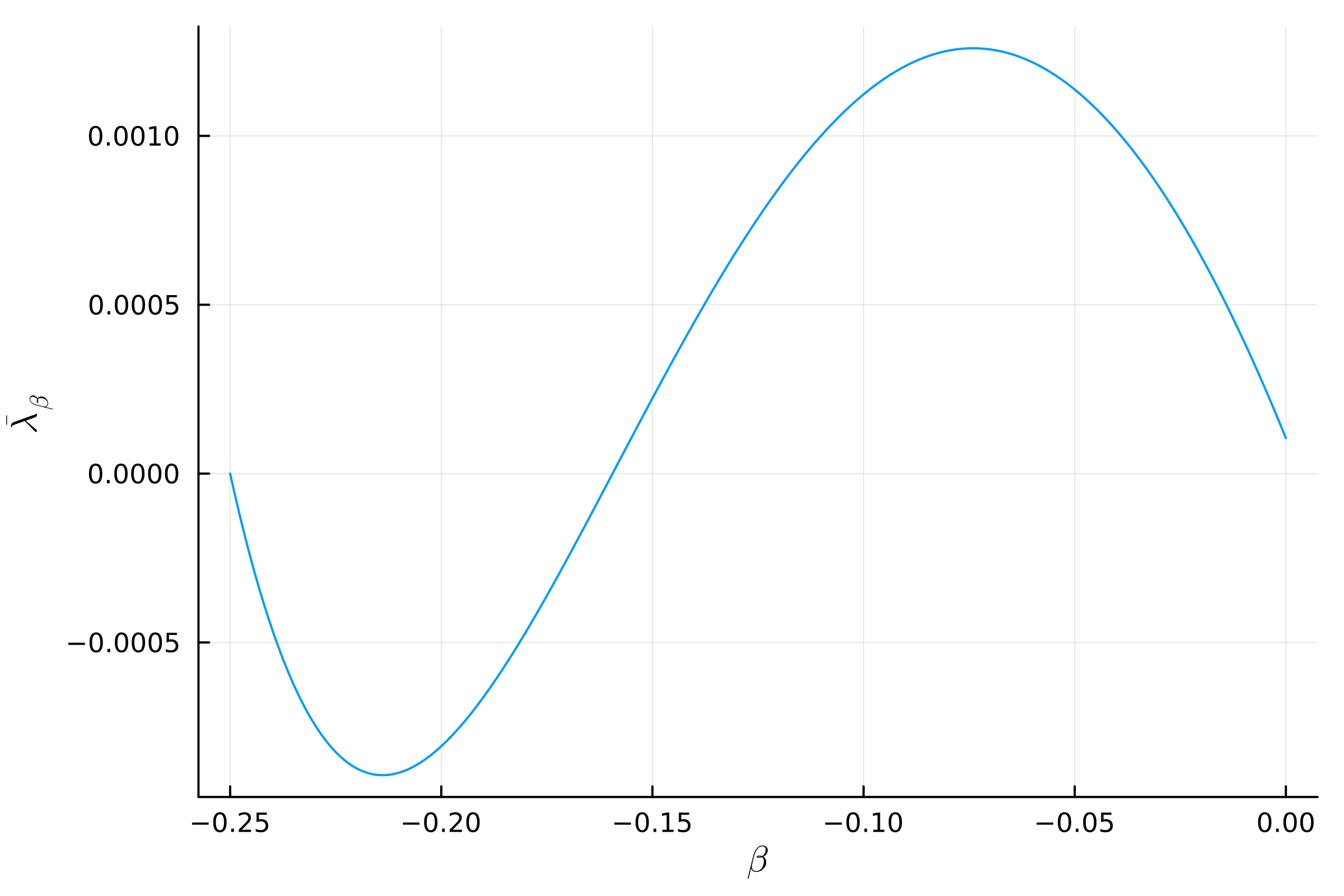} 
\caption{Graph of the function $\beta\mapsto \bar{\lambda}_{\beta}$ from Theorem~\ref{thm:intro-Duffing}.\\ \vspace{\baselineskip}}
\label{fig:Duff-a}
\end{subfigure}
\begin{subfigure}{0.5\textwidth}
\includegraphics[height=0.65\linewidth]{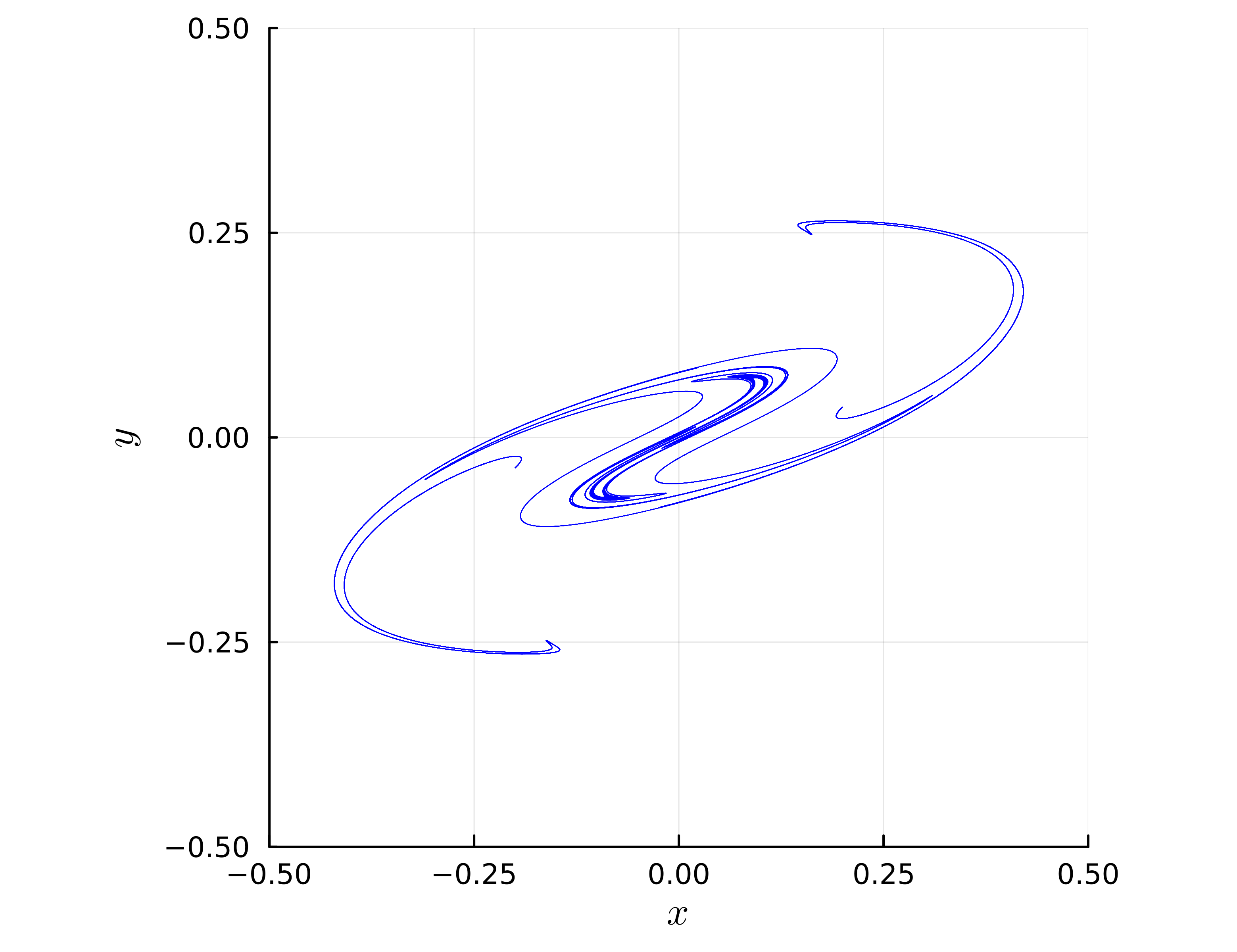} 
\caption{Snapshot of a chaotic random\\ attractor of~\eqref{eqn:bax_vfs} for $\beta = -0.1$.}
\label{fig:Duff-b}
\end{subfigure}

\caption{Quantitative and qualitative random dynamics induced by~\eqref{eqn:bax_vfs}, for $\sigma/\nu = a/\nu =  1$ and $b = 1.24$.}
\end{figure}

\begin{thm}\label{thm:intro-Duffing}
    Consider the stochastic differential equation~\eqref{eqn:gen-sde} with vector fields~\eqref{eqn:bax_vfs}. Assume that $\sigma/\nu = 1$ and $b\nu/a = 1.24$, and let $\beta \mapsto \bar{\lambda}_{\beta}$ be the function represented in Figure~\ref{fig:Duff-a} (and whose precise description can be found at~\cite{Huggzz/Enclosure-of-Lyapunov-exponents}). Then, for \emph{all} $\beta\in (-1/4,0]$,
    $$|\lambda_{\beta} - \bar{\lambda}_{\beta}|\leq 2.2\times10^{-6}.$$
\end{thm}
In Theorem~\ref{thm:intro-Duffing}, we give a uniform error bound in $\beta$ for simplicity, but we in fact get tighter bounds for some values of $\beta$ (see Table~\ref{tab:bounds}). In particular, by refining the estimates close to the values of $\beta$ for which $\lambda_\beta$ vanishes, we obtain the following result.
\begin{thm}\label{thm:Duffing-average-sign}
    Assume that $\sigma/\nu = 1$ and $b\nu/a = 1.24$, then there exists $\beta^*\in (-0.159449,-0.159437)$ such that $\lambda_{\beta^*} = 0$. Furthermore, $\lambda_{\beta}<0$ for $\beta\in(-0.249997,-0.159449]$ and $\lambda_{\beta}>0$ for $\beta\in[-0.159437,0]$.
\end{thm}
Beyond the fact that Theorems~\ref{thm:intro-Duffing} and~\ref{thm:Duffing-average-sign} answer the conjecture from~\cite{Baxendale2004StochasticEquation}, this example is interesting because it features a singularity which needs to be handled at the origin. Another aspect of note is that, when $\beta = -1/4 =-\sigma^2/4\nu^2 $, the system loses its ergodicity. Nonetheless, our method allows to treat the system for all $\beta\in (-1/4,0]$, without any kind of singularity issue as $\beta$ approaches $-1/4$. These two points will be addressed in more detail in Section~\ref{sec:Duffing}. 

Note that the above conjecture from~\cite{Baxendale2004StochasticEquation}, which was obtained using simulations~\cite{Talay1999TheEquations}, turned out to be slightly imprecise regarding the value $\beta^*$ at which the transition from negative to positive Lyapunov exponent occurs. This illustrates the fact that one cannot reliably determine the sign of the Lyapunov exponent $\lambda_{\beta}$ using purely numerical simulations when the magnitude of $\lambda_{\beta}$ is small. In contrast, our method is very precise numerically and also provides guaranteed error bounds, which allows us to rigorously conclude regarding the sign of $\lambda_\beta$, even for $\beta$ relatively close to the transition value $\beta^*$.

Finally, using the averaging results obtained in~\cite{Baxendale2004StochasticEquation}, we can say the following about the Duffing--van der Pol system~\eqref{eqn:Duffing}.
\begin{cor}\label{cor:Duffing-sign}
    Consider the stochastic Duffing--van der Pol system~\eqref{eqn:Duffing}, with  $\sigma/\nu = 1$ and $b\nu/a = 1.24$. Then for $\varepsilon$ small enough, the associated Lyapunov exponent $\hat{\lambda}_{\beta, \varepsilon}$ is well-defined and $\hat{\lambda}_{\beta, \varepsilon} \sim \varepsilon^2\lambda_{\beta}$. Moreover, 
    \begin{itemize}
        \item if $\beta\in[-0.159437,0]$, there exists $\varepsilon_0>0$ such that for all $\varepsilon<\varepsilon_0$, $\hat{\lambda}_{\beta, \varepsilon}$ is positive;
        \item if $\beta\in(-0.249997,-0.159449]$, there exists $\varepsilon_0>0$ such that for all $\varepsilon<\varepsilon_0$, $\hat{\lambda}_{\beta, \varepsilon}$ is negative.
    \end{itemize}
\end{cor}
\begin{proof}
    The facts that $\hat{\lambda}_{\beta, \varepsilon}$ is well-defined and that $\hat{\lambda}_{\beta, \varepsilon} \sim \varepsilon^2\lambda_{\beta}$ as $\varepsilon\to 0$ follow from~\cite[Theorem 1.1.(iii)]{Baxendale2004StochasticEquation}. Corollary~\ref{cor:Duffing-sign} is then a direct consequence of Theorem~\ref{thm:Duffing-average-sign}.
\end{proof}
 


The remainder of the paper is organised as follows. The principle of the adjoint method is recalled in Section~\ref{sec:cap}, where we also introduce a general framework allowing us to turn this idea into a computer-assisted proof, together with a first simple example. We then recall how the Furstenberg--Khasminskii formula can be derived and how the ergodicity of the projective process can be shown in Section~\ref{sec:generalitiesLyap}. In Sections~\ref{sec:cellular},~\ref{sec:pendulum},~\ref{sec:hopf} and~\ref{sec:Duffing}, we provide more background on each of the systems studied in Theorems~\ref{thm:intro-cellular},~\ref{thm:intro-pendulum},~\ref{thm:intro-Hopf} and \ref{thm:intro-Duffing} respectively, together with the proofs of each Theorem. Finally, Section~\ref{sec:outlook} discusses potential further applications of our computer-assisted adjoint method. The computer-assisted parts of the proofs can be reproduced using the code deposited in~\cite{Huggzz/Enclosure-of-Lyapunov-exponents}.

\section{The adjoint method using an approximate solution}
\label{sec:cap}

\subsection{General framework}
\label{sec:cap_framework}
In this section, we lay out the main method used in this paper to enclose ergodic averages. Given a diffusion process $(x_t)_{t\geq 0}$ on a $\sigma$-compact manifold $\mathcal{X}$ (i.e.~the countable union of compact submanifolds) induced by a stochastic differential equation (SDE) with complete vector fields
\begin{equation}\label{eq:SDE_method}
\d x_t  = V_0(x_t)\d t + \sum_{i=1}^\ell V_i(x_t)\circ\d B^i_t,
\end{equation}
assumed to have a unique stationary measure $\mu(\d x)  = w(x) \d x $, and a $\mathcal{C}^2$ $\mu$-integrable function $q: \mathcal{X}\to \mathbb{R}$, our aim is to rigorously enclose the integral
\begin{equation}\label{eq:integral_method}
\mathcal{I} = \int_{\mathcal{X}}q \d \mu.
\end{equation}
This includes the problem of enclosing Lyapunov exponents discussed in the introduction, for which the diffusion process to be considered is the projective process given by~\eqref{eqn:proj-sde} (see also Section~\ref{sec:FK}), but the method presented here applies to any ergodic average.

In order to study $\mathcal{I}$, we consider the operator
$$\mathcal{L} = V_0 +\frac{1}{2}\sum_{i=1}^\ell V_i^2 : \mathcal{C}^{2}(\mathcal{X}) \to \mathcal{C}^{0}(\mathcal{X}),$$
which is the (infinitesimal) generator of the solution process $(x_t)_{t\geq 0}$ satisfying Eq.~\eqref{eq:SDE_method}. We aim to find $\bar{u} \in \mathcal{C}^{2}(\mathcal{X})$ and a constant $\bar{\mathcal{I}}$ such that
\begin{equation}\label{eqn:approx-problem}
    \mathcal{L}\bar{u} \approx q  - \bar{\mathcal{I}}.
\end{equation}
One may approach this problem by first approximating $\mathcal{I}$ with a numerical scheme such as a Monte-Carlo method to produce a candidate $\bar{\mathcal{I}}\approx \mathcal{I}$, and then solve for $\bar{u}$. However, we show below that such a step is not necessary. Indeed, because $\int \mathcal{L}\bar{u} \d \mu=0$, for any $(\bar{u},\bar{\mathcal{I}})$ satisfying~\eqref{eqn:approx-problem} we automatically get $\bar{\mathcal{I}} \approx \mathcal{I}$. This allows us to first solve for $\bar{u}$ only, and then recover $\bar{\mathcal{I}}$, which often leads to a much more precise approximation of $\mathcal{I}$. In this paper, we work under the following setting:


\begin{enumerate}[label=\textnormal{(\arabic*)}]
    \myitem[(S1)]\label{(S1)} We have a Hilbert (orthonormal) basis $\{f_n\}_{n\in \mathbb{N}}\subset \mathcal{C}^{2}(\mathcal{X})$ for a judicious Hilbert space $\left(\mathcal{H},(\cdot, \cdot)\right)$ such that $f_0 = 1$.
    \myitem[(S2)]\label{(S2)} For all $N\in \mathbb{N}$, there exists $M\geq N$ such that $\mathcal{L}$ maps $P_N(\mathcal{H}) = \mathrm{Span}\{f_n\}_{n=0}^N$ to $P_M(\mathcal{H}) = \mathrm{Span}\{f_m\}_{m =0}^M$, where $P_n$ denotes the projection in $\mathcal{H}$ onto $\mathrm{Span}\{f_i\}_{i=0}^n$.
    \myitem[(S3)]\label{(S3)} There is a $\mathcal{C}^{2}$ positive function $W : \mathcal{X} \rightarrow \R_+^*:=(0,\infty)$ such that
    $$\sup_{m\in \mathbb{N}} \left\|\frac{f_m}{W}\right\|_{\infty} <\infty,$$
    and $\mu(W) < \infty$. Furthermore, $\mu(W)$ and these suprema can be bounded explicitly.
    \myitem[(S4)]\label{(S4)} There exists a $\mathcal{C}^{2}$ function $\tilde{W} : \mathcal{X} \rightarrow \R_+^*$ such that $\mathcal{L}\tilde{W} \leq c \tilde{W}$ for some $c<\infty$ and that for all $n\in \mathbb{N}$ 
    $$\frac{|f_n(x)|}{\tilde{W}(x)} \longrightarrow 0 \qquad \mbox{as $|f_n(x)|\to \infty$}.$$
    \myitem[(S5)]\label{(S5)} There is $n_0\in \mathbb{N}$ such that $q\in \mathrm{Span}\{f_n\}_{n=0}^{n_0}$.
\end{enumerate}

In practice, we take $W = \tilde{W}$, except for the Duffing--van der Pol example in Section~\ref{sec:Duffing}; and if $\mathcal{X}$ is compact, we will simply take $W = \tilde{W} = 1$.

In the heuristic description of our approach in Section~\ref{sec:method}, we heavily used the fact that elements in the range of $\mathcal{L}$ have $\mu$-mean equal to zero. The following lemma, which is a direct rewriting of~\cite[Proposition B.1]{Baxendale2024LyapunovNoise} adapted to our framework, gives sufficient conditions allowing the use of this fact rigorously.
\begin{lem}\label{lem:Bax}
    Under assumptions~\ref{(S1)}--\ref{(S4)}, for all $n\in \N$,
    \begin{equation}\label{eqn:adjoint-prop}
    \int_{\mathcal{X}}(\mathcal{L}f_n)\d\mu = 0.
    \end{equation}
    Thus, by linearity, for any $u\in P_N(\mathcal{H})$,
    $$\int_{\mathcal{X}}(\mathcal{L}u)\d\mu = 0.$$
\end{lem}
\begin{proof}
    Fix $n\in \N$, it is clear that under assumptions~\ref{(S2)} and~\ref{(S3)}, $f_n$ and $\mathcal{L}f_n$ are $\mu$-integrable. These two integrability conditions together with assumption~\ref{(S4)} and the fact that $f_n$ is $\mathcal{C}^2$ imply~\eqref{eqn:adjoint-prop} by~\cite[Proposition B.1]{Baxendale2024LyapunovNoise}.
\end{proof}


The problem~\eqref{eqn:approx-problem} is then expressed and solved numerically with respect to the Hilbert basis $\{f_n\}_{n\in \N}$, i.e.~we represent the action of $\mathcal{L} :P_N(\mathcal{H})\to P_M(\mathcal{H})$ by its Galerkin representation, the finite-dimensional matrix
$$\left(f_m, \mathcal{L} f_n\right)_{m\leq M, n\leq N},$$
and all the subsequent operations are equivalent to finite-dimensional linear algebra: we find $\bar{u}$, a \emph{numerical} solution of the least-square problem
\begin{equation}
\label{eq:leastsquare}
    \underset{u \in P_N(\mathcal{H})}{\operatorname{argmin}}\|P_0^{\perp}(\mathcal{L}u - q)\|^2_{\mathcal{H}},
\end{equation}
or of the linear problem
$$(P_NP_0^{\perp} \mathcal{L}P_NP_0^{\perp}) u = P_0^{\perp}q.$$

In our experience, the first approach gives slightly better results. Note that this is a purely numerical guess; there is in general no theoretical guarantee or indication for this procedure to converge quickly as $N\to \infty$. For the main problems considered in this work, the usual compactness estimates do not hold and convergence is typically not of spectral nature~\cite{Trefethen2000SpectralMATLAB}. The argument presented below will only provide a tight enclosure of $\lambda$ if the approximate solution $\bar u$ is accurate enough, but we stress that we do not need any a priori estimate regarding the quality of $\bar u$. 

\begin{rmk}
    For very high-dimensional sparse systems, it will in general prove better (for instance in terms of memory requirements) to make use of iterative methods to solve truncated or least-square versions of this problem~\cite{Paige1982LSQR:Squares, Trefethen1997ChapterMethods}.
\end{rmk}

We now assume that an approximate solution $\bar{u}\in P_N(\mathcal{H})$ of~\eqref{eqn:approx-problem}  has been computed for some $N$. In practice, we do so via~\eqref{eq:leastsquare}, but we again emphasize that the way such a $\bar{u}$ is actually obtained is irrelevant for the estimates to come. In particular, none of the above computations need to have been performed rigorously. We are now ready give a fully justified and more detailed version of method presented in Section~\ref{sec:method}.

\begin{prop}\label{prop:method}
Assume that the SDE~\eqref{eq:SDE_method} having complete vector fields and posed on a $\sigma$-compact manifold $\mathcal{X}$ has a unique stationary measure $\mu$. Let $q: \mathcal{X}\to \mathbb{R}$ a $\mathcal{C}^2$ $\mu$-integrable function, and $\mathcal{I}$ be the integral~\eqref{eq:integral_method}. Assume~\ref{(S1)}--\ref{(S5)}, and let $\bar{u}\in P_N(\mathcal{H})$ for some $N\in\mathbb{N}$. Let
\begin{equation}\label{eqn:barI-def}
    \bar{\mathcal{I}} \overset{\mathrm{def}}{=} P_0 (q -\mathcal{L}\bar{u}),
\end{equation}
and 
\begin{equation}\label{eqn:barq-def}
    \bar{q} \overset{\mathrm{def}}{=} \mathcal{L} \bar{u}+\bar{\mathcal{I}} .
\end{equation} 
Then
\begin{equation}\label{eq:IminusbI}
    \mathcal{I} - \bar{\mathcal{I}}= \int_\mathcal{X}(q - \bar{q})\d \mu.
\end{equation}
Moreover, denoting by $\{\epsilon_m\}_{m = 1}^M$ the coefficients of $q - \bar{q}$ with respect to the basis $\{f_m\}_{m \geq 0}$, i.e.
\begin{equation} \label{eqn:rigorous-sum}
    q-\bar{q} = \sum_{m = 1}^M \epsilon_m f_m,
\end{equation}
we have the estimate
\begin{equation}\label{eq:computable_estimate}
    |\mathcal{I} - \bar{\mathcal{I}}| \leq \mu(W)\sum_{m = 1}^M |\epsilon_m|\left\|\frac{f_m}{W}\right\|_{\infty}.
\end{equation}
\end{prop}
\begin{proof}
    By Lemma~\ref{lem:Bax} we have $\int_{\mathcal{X}}(\mathcal{L}\bar{u})\d\mu = 0$, and integrating~\eqref{eqn:barq-def} immediately yields~\eqref{eq:IminusbI}. Then, thanks to assumptions~\ref{(S2)} and \ref{(S5)}, $q-\bar{q}$ belongs to $P_M(\mathcal{H})$ for some $M\in\mathbb{N}$, hence $q-\bar{q}$ can be written as $\sum_{m = 0}^M \epsilon_m f_m$. We note that, as claimed in~\eqref{eqn:rigorous-sum}, the first coefficient $\epsilon_0$ is in fact $0$, because~\eqref{eqn:barI-def} and~\eqref{eqn:barq-def} imply $q-\bar{q} = P_0^{\perp} (q -\mathcal{L}\bar{u})$. We then simply estimate
    \begin{equation*}
        |\mathcal{I} - \bar{\mathcal{I}}| = \left|\int_\mathcal{X}(q - \bar{q})\d \mu\right|\leq \mu(W)\left\|\frac{(q-\bar{q})}{W}\right\|_{\infty} \leq \mu(W)\sum_{m = 1}^M |\epsilon_m|\left\|\frac{f_m}{W}\right\|_{\infty}.\qedhere
    \end{equation*}
\end{proof}
Estimate~\eqref{eq:computable_estimate} gives us an enclosure of $\mathcal{I}$, which is expected to be tight if $\bar{u}$ is an accurate approximation of the solution $u$ of the Poisson equation $\mathcal{L}u = q-\mathcal{I}$. Moreover, the right hand side of~\eqref{eq:computable_estimate} can be explicitly and rigorously bounded. This only requires 
\begin{itemize}
    \item Rigorously applying $\mathcal{L}$ to the approximate solution $\bar{u}$ in order to obtain $\mathcal{\bar{I}}$ and $\bar{q}$ in~\eqref{eqn:barI-def}-\eqref{eqn:barq-def}, which then yields the coefficients $\epsilon_m$,
    \item Getting computable upper-bounds for $\mu(W)$ and for $\left\|\frac{f_m}{W}\right\|_{\infty}$, for finitely many $m$'s.
\end{itemize} 
For instance, if we use Fourier series and $W=1$ (as will be the case in some examples), we simply have $\left\|\frac{f_m}{W}\right\|_{\infty}=1$ for all $m$. When using more involved bases, we get explicit bounds on $\left\|\frac{f_m}{W}\right\|_{\infty}$ using interval arithmetic together with some elementary estimates. In this work, we use the Julia package \texttt{IntervalArithmetic.jl}~\cite{david_p_sanders_2024_10459547} to perform our rigorous calculations, and the required estimates are discussed case by case in the example sections.

\begin{rmk}\label{rmk:mat-vec}
    The evaluation of $\bar{q}$~\eqref{eqn:barq-def} which is the most expensive calculation in interval arithmetic is only a matrix-vector multiplication. This is in contrast with most computer-assisted proofs for ODEs or PDEs, which usually at least require a matrix-matrix multiplication, the inversion of a matrix, the resolution of a linear system or many matrix-vector multiplications in rigorous arithmetic. That allows for the treatment of problems with a much higher truncation dimension $N$ which are especially necessary for non-elliptic problems on an unbounded multi-dimensional domain $\mathcal{X}$. Furthermore (at least for the rigorous part of the proof), provided $\mathcal{L}$ has a sparse representation with respect to $\{f_n\}_{n\in \mathbb{N}}$, this approach is not so limited by memory resources (which typically restrict computer-assisted proofs), as we can take full advantage of the sparsity of $\mathcal{L}$.
\end{rmk}

\begin{rmk}
    The main reason for the need for a weight function $W$, is that for our class of problems which are analytic, the Hilbert basis $\{f_n\}_{n\in \mathbb{N}}$ and $q$ are typically of polynomial type and thus unbounded on an unbounded domain. In this case, $W(x)$ needs to grow exponentially or superexponentially as $|x|\to \infty$. Note that the existence of such weight or Lyapunov function is a standard ingredient for the proof of ergodicity of $(x_t)_{t\geq 0}$. When choosing $W$, there is of course a trade-off between the values of $\|f_m/W\|_{\infty}$ and $\mu(W)$.
\end{rmk}

\begin{rmk}
    Some of the assumptions of our setting can be loosened. For instance, instead of~\ref{(S5)}, it is sufficient for $q$ to be analytic and thus have geometrically decaying coefficients with respect to our Hilbert basis, but this would require an additional bound for the tail of this series. We could also choose to express the problem with respect to two different bases for arguments and images of $\mathcal{L}$. Moreover, working with a Hilbert basis for the domain of $\mathcal{L}$ is not strictly necessary, as will be showcased in Section~\ref{sec:adj_Duffing}. 
    More generally, we stress that the setting presented in this section is only one approach among others that one could take for the rigorous enclosure of ergodic averages via the adjoint method. The present approach is most likely the simplest one but may not be suitable in every situation (e.g.~if $w$ is not exponentially decaying on an unbounded domain $\mathcal{M}$).
\end{rmk}


\subsection{Example: The average Lyapunov exponent for the cellular flow with sinks}\label{subsec:volume}

Consider the cellular flow with sinks $(\varphi_t)_{t\geq 0}$ on the two-torus $\mathcal{M} = \mathbb{T}^2$ generated by the differential equation
$$\begin{cases}
    \d x_t &= (\cos(x_t)/2-\cos y_t)\sin x_t \d t+\sigma\d B^1_t\\
    \d y_t &= (\cos(y_t)/2+\cos x_t)\sin y_t \d t+\sigma\d B^2_t
\end{cases}$$
where $\mathbb{T}^2 = (\R / \mathbb{Z})^2 \simeq [0,2\pi)^2$ is endowed with the Riemannian structure inherited from $\mathbb{R}^2$ (throughout this paper, the Riemannian structure on $\mathbb{T} = \R / \mathbb{Z}$ will always be inherited from $\mathbb{R}$). Since the noise is additive and the state space $\mathcal{M} = \mathbb{T}^2$ is compact, the Markov process $(\varphi_t)_{t\geq 0}$ is clearly ergodic with unique stationary measure $\mu$ on $\mathbb{T}^2$. This system is a modification of the classical Hamiltonian cellular flow~\cite{Brue2024EnhancedFlows}, and thus the incompressibility of the flow is lost. For this variant, one actually expects volume contraction along trajectories (see Figure~\ref{fig:cellular-a}). This volume growth rate is given by the so-called \emph{average Lyapunov exponent}
$$\lambda_{\Sigma}(\omega, x) = \lim_{t\to\infty}\frac{1}{2t}\log\det{D\varphi_t(\omega, x)}.$$
Before turning our attention to the actual Lyapunov exponent $\lambda$ for this system in Section~\ref{sec:cellular}, we first study here the simpler case of the average Lyapunov exponent, which provides an easy example to illustrate our method. By Liouville's formula~\cite[Theorem 2.3.32]{Arnold1998RandomSystems},
$$\det(D\varphi_t(\omega, x)) = \exp\int_0^t\tr D(X_0(\varphi_\tau))\d\tau = \exp\int_0^t\frac{\cos 2x_\tau +\cos 2y_\tau}{2}\d\tau.$$
Then, by Birkhoff's ergodic theorem, $\lambda_{\Sigma}$ is constant $(\mathbb{P}\times \mu)$-almost surely and
$$\lambda_{\Sigma} = \lim_{t\to\infty}\frac{1}{2t}\int_0^t\frac{\cos 2x_\tau +\cos 2y_\tau}{2}\d\tau = \int_{\mathbb{T}^2}\frac{\cos 2x+\cos 2y}{4}\mu(\d x, \d y).$$
We thus aim to find numerical solutions $\bar{u}$ and $\bar{\lambda}_{\Sigma}$ to the Poisson problem
\begin{equation}
    \mathcal{L}u = \left(\frac{\cos x}{2}-\cos y\right)\partial_xu+\left(\frac{\cos y}{2}+\cos x\right)\partial_yu+\frac{\sigma^2}{2}\left(\partial_x^2+\partial_y^2\right)u=\frac{\cos 2x+\cos 2y}{4}-\lambda_{\Sigma} = q(x,y) -\lambda_{\Sigma}\label{eqn:volume_poisson}.
\end{equation}

We choose the basis $\{f_i\}_{i\in \mathbb{N}} = \{\cos mx, \sin mx\}_{m \in \N}\otimes \{\cos ny, \sin ny\}_{n \in \N}$, fix $\sigma = \sqrt{2}$ (with $W = \tilde{W} = 1$), and in the file~\texttt{cellular/proof\_volume} at~\cite{Huggzz/Enclosure-of-Lyapunov-exponents} we construct
$$\bar{u}(x,y) = \sum_{m = 0 }^{50}\sum_{n = 0}^{50} a_{mn}\cos mx\cos ny
+b_{mn}\cos mx\sin ny +c_{mn} \sin mx \cos ny +d_{mn} \sin mx \sin n y,$$
represented in Figure~\ref{fig:volume_ubar}.

\begin{figure}[h]
    \centering
    \includegraphics[width=0.5\textwidth]{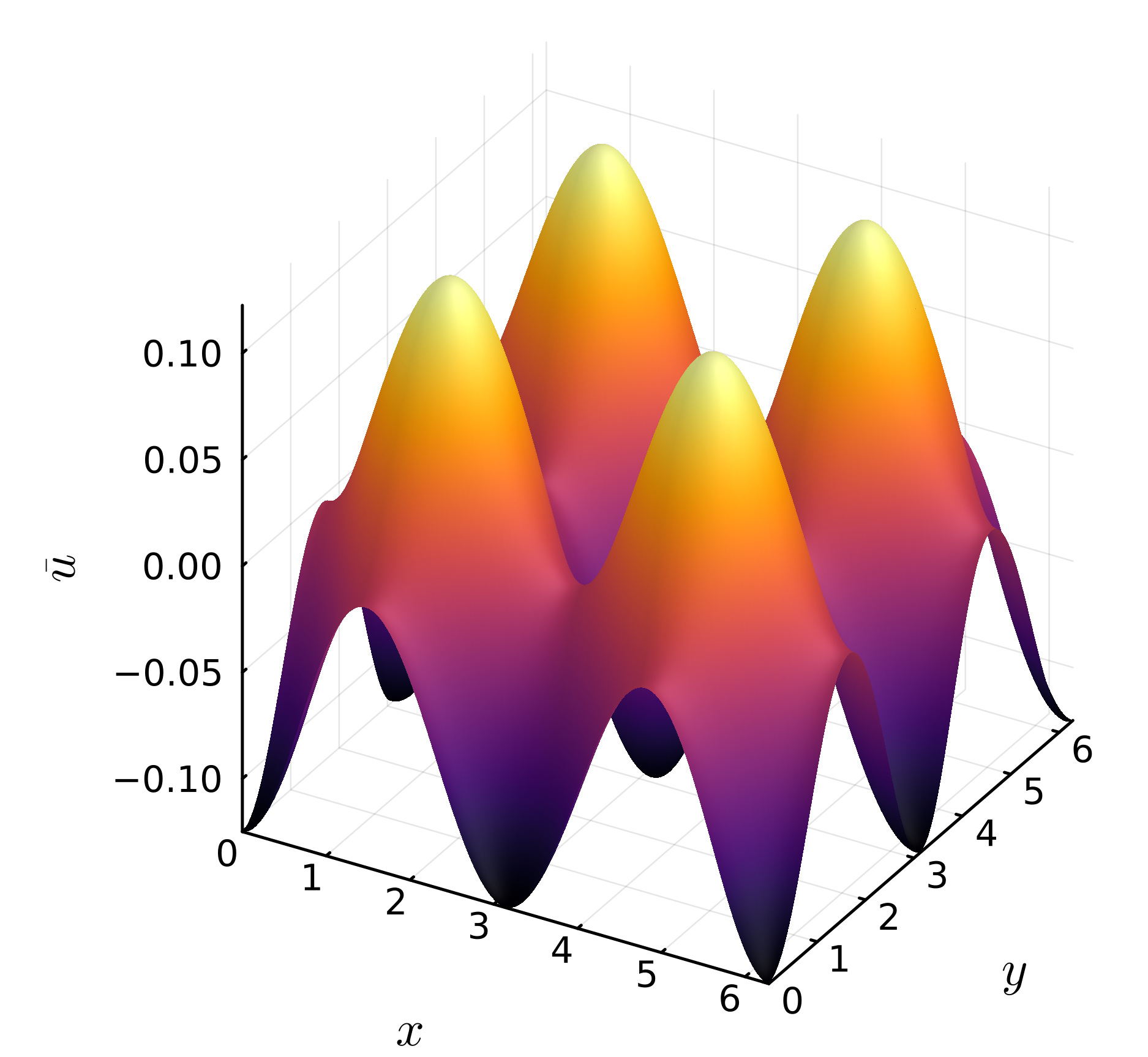}
    \caption{Plot of the numerical solution $\bar{u}$ to Eq.~\eqref{eqn:volume_poisson} with $\sigma = \sqrt{2}$.\label{fig:volume_ubar}}
\end{figure}

\begin{rmk}
    Of course, one should normally try to reduce the dimension of the numerical problem by exploiting the symmetries of the system. For instance, in this case, one can observe in Figure~\ref{fig:volume_ubar} that
    $$\bar{u}(x,y) = \sum_{m =1}^{25}\sum_{n = 1}^{25} a_{2m, 2n} \cos{2mx}\cos{2ny}.$$
    This is of importance when solving large linear systems.
\end{rmk}

We then choose

$$\bar{\lambda}_{\Sigma} : =  \frac{1}{(2\pi)^2}\int_{\mathbb{T}^2} (q -\mathcal{L}\bar{u})\, \d x \d y=-\frac{1}{(2\pi)^2}\int_{\mathbb{T}^2} (\mathcal{L}\bar{u})\, \d x \d y,$$
i.e., as in Eq.~\eqref{eqn:barq-def}, $\bar{\lambda}_{\Sigma}$ is the projection of $q - \mathcal{L}\bar{u}$ on constants and
$$\bar{q} := \mathcal{L}\bar{u} +\bar{\lambda}_{\Sigma}.$$
Since

$$\bar{\lambda}_{\Sigma} = \int_{\mathbb{T}^2} \bar{q}\d \mu,$$
we thus have that

$$|\bar{\lambda}_{\Sigma} - \lambda_{\Sigma}| = \left|\int_{\mathbb{T}^2}(q-\bar{q})\d \mu\right|\leq \|q - \bar{q}\|_{\infty}.$$
Thus, by bounding the right-hand side, we find that
\begin{equation}\label{eqn:volume-lambda-cellular}
    \lambda_{\Sigma} = -0.0308582892201142 \pm 5 \times 10^{-16}< 0,
\end{equation}
confirming that the stochastic flow $(\varphi_t)_{t\geq 0}$ is volume contracting. Note that $q$ is a trigonometric polynomial, and that we constructed $\bar u$ as a trigonometric polynomial. Therefore, $\bar q = \mathcal{L}\bar{u} +\bar{\lambda}_{\Sigma}$ is also a trigonometric polynomial, and explicitly bounding $\|q - \bar{q}\|_{\infty}$ is straightforward (for instance using the $\ell^1$-norm of the Fourier coefficients).

Of course, this particular problem could instead also be treated with more traditional computer-assisted methods by computing the stationary density $w$ via the stationary Fokker--Planck equation $\mathcal{L}^*w = 0$ which is a uniformly elliptic equation on the torus $\mathbb{T}^2$~\cite{Arioli2005TwoModel,Hungria2015RigorousApproach,Nakao2019NumericalEquations}. However, as already mentioned, our method requires significantly fewer computations than what would be needed for rigorously solving the stationary Fokker--Planck equation. But most importantly, our method does not make any assumption on the ellipticity of the operator $\mathcal{L}$ or on the compactness of the domain and can thus treat a much wider class of examples. This is for instance the case of the main problem of this paper: the computations of Lyapunov exponents.

\section{Lyapunov exponents via the projective process}
\label{sec:generalitiesLyap}

In this section, we first recall how the Furstenberg--Khasminskii formula can be derived, and then give sufficient conditions for proving the ergodicity of the projective process.

\subsection{The Furstenberg--Khasminskii formula}
\label{sec:FK}


For the sake of presentation, let us consider first the additive noise case, i.e.~$X_i  = \sigma_i e_i$ for all $i \in\{1, \ldots, d\}$ for some $\sigma_i \geq 0$ where $\mathcal{M} = \mathbb{T}^{d_1}\times\R^{d_2}$ (with $d = d_1+d_2$). We then have that

$$\d \varphi_t = X_0(\varphi_t) \d t +\sum_{i=1}^{d}\sigma_i e_i \d B^i_t.$$
Then the Furstenberg--Khasminskii formula can be derived as follows.
%
%
For $v \in \mathbf{P}_x \mathcal{M}$, we define
\begin{equation}\label{eqn:rho-s-def}
\varrho_t(\omega, x, v) = \|D\varphi_t(\omega, x) v\|, \qquad s_t(\omega, x, v) = \frac{D\varphi_t(\omega, x)v}{\|D\varphi_t(\omega, x)v\|} \in T_{\varphi_t(\omega, x)}\mathcal{M}.
\end{equation}
Then $(s_t)_{t\geq 0}$ solves the random differential equation (a nonautonomous ordinary differential equation)

\begin{equation}\label{eqn:s-eq-additive}
    \diff{s_t}{t} = DX_0(\varphi_t)s_t - Q(\varphi_t, s_t)s_t,
\end{equation}
where $Q(x,s) = \langle DX_0(x)s, s\rangle$ and similarly

$$\diff{\log\varrho_t}{t} = Q(\varphi_t, s_t)\quad \text{and} \quad \varrho_t = \exp\int_0^t Q(\varphi_\tau, s_\tau)\d\tau.$$
Therefore, 
if $(\xi_t)_{t\geq 0} = (\varphi_t, s_t)_{t\geq 0}$ has a unique invariant probability measure $\tilde{\mu}$ and $Q\in L^1(\tilde{\mu})$, by Birkhoff's ergodic theorem, we obtain the Furstenberg--Khasminskii formula

$$\lambda = \lim_{t\to\infty}\frac{1}{t}\log \varrho_t = \lim_{t\to\infty}\frac{1}{t}\int_0^t Q(\xi_{\tau})\d \tau = \int_{\mathbf{P}\mathcal{M}} Q \d \tilde{\mu}.$$

\noindent In the multiplicative noise case,
\begin{equation*}
    \d \varphi_t = X_0(\varphi_t) \d t + \sum_{i=1}^\ell X_i(\varphi_t) \circ \d B^i_t, \qquad \varphi_0 = x,
\end{equation*}
the evolution equation of $(\xi_t)_{t\geq0}$ has the form

\begin{equation*}
    \d \xi_t = \tilde{X}_0(\xi_t) \d t + \sum_{i=1}^\ell\tilde{X}_i(\xi_t) \circ \d B^i_t \qquad \xi_0 = (x, v),
\end{equation*}
where the vector fields $\tilde{X}_0,\tilde{X}_1, \ldots, \tilde{X}_\ell$ on $\mathbf{P}\mathcal{M}$ are given by~\cite{Carverhill1985ATheorem}, which also gives a formula for the function $Q$ in the multiplicative noise case,
such that

$$\d \log \varrho_t = Q(\varphi_t, s_t)\d t + \sum_{i = 1}^\ell \langle D X_i(\varphi_t)s_t, s_t\rangle \d B^i_t;$$
where the above equation is now written in It\^{o} form. Provided

\begin{equation}\label{eqn:annoying-part}
\frac{1}{t}\int_0^t \langle D X_i(\varphi_{\tau})s_{\tau}, s_{\tau}\rangle \d B^i_{\tau} \xlongrightarrow{t\to\infty} 0, \qquad i = 1, \ldots,\ell,
\end{equation}
almost surely or in expectation with respect to $\mathbb{P}\times \tilde{\mu}$, then as before, by Birkhoff's ergodic theorem

$$\lambda = \int_{\mathbf{P}\mathcal{M}} Q \d \tilde{\mu}.$$
The limit~\eqref{eqn:annoying-part} holds in expectation for instance if

\begin{equation}\label{eqn:mult-noise-int}
\int_{\mathcal{M}} \|DX_i\|^2 \d \mu < \infty, \qquad i = 1, \ldots,\ell,
\end{equation}
which is a standard integrability condition for the multiplicative ergodic theorem~\cite[Theorem 4.2.13]{Arnold1998RandomSystems}. Indeed, under such integrability condition, a simple application of the Cauchy--Schwarz inequality and the It\^{o} isometry yields that

\begin{align*}
\frac{1}{t}\int_{\mathbf{P}\mathcal{M}}\mathbb{E}\left[\left|\int_0^t \langle D X_i(\varphi_{\tau})s_{\tau}, s_{\tau}\rangle \d B^i_{\tau}\right|\right]\d\tilde{\mu}&\leq\frac{1}{\sqrt{t}}\left(\int_{\mathcal{M}}\left\|D X_i\right\|^2 \d\mu\right)^{1/2}\xrightarrow{t\to\infty} 0.
\end{align*}
%
%
This implies that the limit~\eqref{eqn:annoying-part} holds in expectation.

\begin{rmk}
   Lower Lyapunov exponents (see~\cite[Chapter 3]{Arnold1986LyapunovSystems} for a definition) can be obtained similarly~\cite{Baxendale1986TheDiffeomorphisms} by generalising the projective process to a Grassmannian process. There is in principle no obstruction for our method to be applied to compute these additional Lyapunov exponents but, in this paper, we mostly restrict ourselves to the top Lyapunov exponent for the sake of clarity. In particular, the bottom Lyapunov exponents can be obtained with a system similar and of the same dimension as the one studied in this work for the (top) Lyapunov exponent.\label{rmk:grassmann} 
\end{rmk}

\subsection{On the ergodicity of the projective process}\label{sec:proj-ergodicity}

Let us now address the question of the existence and uniqueness of a measure $\tilde{\mu}$ for the projective process $(\xi_t)_{t\geq 0}$. To this end, let us recall that the Lie bracket $[U,V]$ between two vector fields $U$ and $V$ is the vector field given by

$$[U,V](x) = DV(x)U(x) -DU(x) V(x).$$
We can now state the following nondegeneracy conditions.

\begin{defn}
    Consider a stochastic differential equation written in Stratonovich form
    \begin{equation}
        \d x_t = V_0(x_t) \d t + \sum_{i=1}^\ell V_i(x_t) \circ \d B^i_t \label{eqn:horm-sde}
    \end{equation}
    with vector fields $V_0, V_1,\ldots, V_\ell$ on a manifold $\mathcal{X}$. Then \eqref{eqn:horm-sde} is said to satisfy
    \begin{enumerate}[label=\textnormal{(\arabic*)}]
        \myitem[(WH)]\label{(WH)} the \emph{weak} (or elliptic) H\"{o}rmander condition if $L(x)$ spans $T_x\mathcal{X}$ for all $x \in \mathcal{X}$ where $L = LA(V_0, V_1, \ldots, V_\ell)$ denotes the Lie algebra generated by $V_0, V_1, \ldots, V_\ell$.
        \myitem[(PH)]\label{(PH)} the \emph{parabolic} H\"{o}rmander condition if $L_0(x)$ spans $T_x\mathcal{X}$ for all $x \in \mathcal{X}$ where $L_0$ denotes the ideal generated by $V_1, \ldots, V_\ell$ in $L$.
        \myitem[(SH)]\label{(SH)} the \emph{strong} H\"{o}rmander if $LA(V_1, \ldots, V_\ell)$ spans $T_x\mathcal{X}$ for all $x\in \mathcal{X}.$
    \end{enumerate}
\end{defn}

These conditions are typically checked by iterating the Lie Brackets
$$\left[\,\left[\ldots \left[V_{i_1}, V_{i_2}\right], \ldots\right], V_{i_n}\right](x)$$
until they span $T_x\mathcal{X}$ for all $x\in\mathcal{X}$, where the allowed indices $i_1, \ldots, i_n$ depend on the condition to be checked. These conditions may be verified by hand e.g.~in Section~\ref{sec:pendulum} or by computer assistance e.g.~in Section~\ref{sec:cellular}.

Note that clearly $\ref{(SH)} \implies \ref{(PH)}\implies \ref{(WH)}$, however only in some rare cases, does~\ref{(WH)} hold without \ref{(PH)}(see~\cite{Ichihara1974ACharacterization} and its supplements and corrections). While the parabolic H\"{o}rmander~\ref{(PH)} condition implies that the Markov process $(x_t)_{t\geq 0}$ has smooth transition probabilities and is strong Feller~\cite[Theorem 1.3]{Hairer2011OnTheorem}, the weak H\"{o}rmander condition~\ref{(WH)} implies the smoothness of stationary densities.

If a diffusion satisfies the strong H\"{o}rmander condition~\ref{(SH)}, it is usually more straightforward to prove the existence and uniqueness of a stationary density~\cite{Arnold1987OnDiffusions}. 
If a diffusion does not satisfy the strong H\"{o}rmander condition~\ref{(SH)}, it is said to be degenerate. If it only satisfies the parabolic H\"ormander condition~\ref{(PH)} (i.e.~is strong Feller) coupled with a recurrence/Lyapunov condition, then in general, one can only conclude the existence of a stationary measure~\cite[Section 6]{Canizo2023Harris-typeSemigroups}. In order to obtain uniqueness, one typically has to prove the controllability or topological irreducibility of the Markov process $(x_t)_{t\geq 0}$ via the Stroock--Varadhan theorem~\cite{Hairer2008ErgodicPDEs, Stroock1972OnPrinciple}. This last step is not needed when the diffusion satisfies the strong H\"ormander condition~\ref{(SH)}, as the controllability is automatic in that case. 

Unfortunately, in the case of the projective process $(\xi_t)_{t\geq 0} = (\varphi_t, s_t)_{t\geq 0}$, the strong H\"ormander condition~\ref{(SH)} is typically not satisfied. This is for instance clear in the additive noise case (see Eq.~\eqref{eqn:s-eq-additive}). One then normally has to check for the controllability of $(\xi_t)_{t\geq 0}$, in addition of the parabolic H\"ormander condition~\ref{(PH)} for $(\xi_t)_{t\geq 0}$. While controllability is not so hard to check on a case-by-case basis, it can sometimes be fiddly and cumbersome. Fortunately,
San Martin and Arnold~\cite{SanMartin1986AFlow} give weak and general conditions for the controllability and hence uniqueness of a stationary measure for $(\xi_t)_{t \geq 0}$ on $\mathbf{P}(\mathcal{M})$ to hold; we state their result below.

First recall that the evolution equations of $(\varphi_t)_{t\geq 0}$ and $(\xi_t)_{t\geq 0}$ are respectively given by

\begin{equation}
    \d \varphi_t = X_0(\varphi_t) \d t + \sum_{i=1}^\ell X_i(\varphi_t) \circ \d B^i_t(\omega) \label{eqn:gen-sde-2}, \qquad \varphi_0 = x,
\end{equation}
and

\begin{equation}
    \d \xi_t = \tilde{X}_0(\xi_t) \d t + \sum_{i=1}^\ell \tilde{X}_i(\xi_t) \circ \d B^i_t \label{eqn:proj-sde-2}, \qquad \xi_0 = (x, v),
\end{equation}
where $X_0, X_1, \ldots, X_\ell$  (and thus $\tilde{X}_0, \tilde{X}_1, \ldots, \tilde{X}_\ell$) are \emph{analytic} and complete vector fields. Then~\cite{SanMartin1986AFlow} asserts that if $(\varphi_t)_{t\geq 0}$ possesses a unique ergodic probability measure on $\mathcal{M}$, then it is sufficient for $(\xi_t)_{t\geq 0}$ to merely fulfil the weak H\"{o}rmander condition~\ref{(WH)} for it to possess a unique ergodic probability measure on $\mathbf{P}\mathcal{M}$.

\begin{prop}\label{prop-san-martin}
    Recall that $X_0, X_1, \ldots, X_\ell$ are \emph{analytic} and complete vector fields on $\mathcal{M}$. Assume that the Markov process $(\varphi_t)_{t\geq 0}$ possesses a unique ergodic probability measure $\mu(\d x) = w(x) \d x$ on $\mathcal{M}$ and that $(\varphi_t)_{t\geq 0}$ is controllable on $C = \mathrm{supp}\mu \subset \mathcal{M}$. Then if the evolution equation~\eqref{eqn:proj-sde-2} of $(\xi_t)_{t\geq 0}$ satisfies the weak H\"{o}rmander condition~\ref{(WH)}, then there exists a unique ergodic probability measure $\tilde{\mu}$ for $(\xi_t)_{t\geq 0}$ on $\mathbf{P}\mathcal{M}$. Furthermore, denoting $\tilde{\mu}(\d\xi) = \tilde{w}(\xi)\d\xi$, then $\tilde{w}$ is smooth, $\tilde{\mathcal{L}}^*\tilde{w} = 0$ and for all $x\in \mathcal{M}$
    $$w(x) = \int_{\mathbf{P}(T_x\mathcal{M})}\tilde{w}(x,s)\d s$$
\end{prop}
\begin{proof}
    Under these assumptions, we can combine Theorem 7 and Corollary 2 of~\cite{SanMartin1986AFlow} which give the uniqueness of an invariant control set $\tilde{C}$ for $(\xi_t)_{t\geq 0} = (\varphi_t, s_t)_{t\geq 0}$ on $\mathbf{P}C \subset \mathbf{P}\mathcal{M}$. Furthermore, there exists a unique smooth ergodic measure $\d\tilde{\mu} = \tilde{w}\d \xi$ for $(\xi_t)_{t\geq 0}$ on $\mathbf{P}C$ and thus on $\mathbf{P}\mathcal{M}$. This essentially follows from the Stroock--Varadhan theorem (e.g.~see~\cite{Hairer2008ErgodicPDEs}) which asserts that invariant distributions are supported on such control sets. Note that under the weak H\"{o}rmander condition~\ref{(WH)}, there can only exist one invariant measure on such control set~\cite{Arnold1987OnDiffusions}.
\end{proof}

\begin{cor}\label{cor:ergodicity}
    Assume that the Markov process $(\varphi_t)_{t\geq 0}$ possesses an invariant distribution $\mu$ on $\mathcal{M}$. Assume further that the evolution equation~\eqref{eqn:gen-sde-2} of $(\varphi_t)_{t\geq 0}$ satisfies the strong  H\"{o}rmander condition~\ref{(SH)} and the evolution equation~\eqref{eqn:proj-sde-2} of $(\xi_t)_{t\geq 0}$ satisfies the weak H\"{o}rmander condition~\ref{(WH)}. Then $\mu$ is the unique ergodic measure for $(\varphi_t)_{t\geq 0}$ and the conclusions of Proposition~\ref{prop-san-martin} hold.
\end{cor}

As in Remark~\ref{rmk:grassmann}, this result can be generalised to other fibre bundles~\cite{SanMartin1986InvariantBundles} such as Grassmannian bundles to obtain lower Lyapunov exponents.

\section{The cellular flow with sinks}\label{sec:cellular}

In this section, we first discuss the example presented in Theorem~\ref{thm:intro-cellular}, and then treat a similar system but with multiplicative noise.

Consider again the cellular flow with sinks

\begin{equation}\label{eqn:cellular-sde}
    \begin{cases}
        \d x_t &= (\cos(x_t)/2-\cos y_t)\sin x_t \d t+\sigma\d B^1_t\\
        \d y_t &= (\cos(y_t)/2+\cos x_t)\sin y_t \d t+\sigma\d B^2_t
    \end{cases}
\end{equation}
i.e.~the corresponding vector fields on $\mathcal{M} = \mathbb{T}^2$ are

$$X_0(x, y) = \vect{(\cos(x)/2-\cos y)\sin x}{(\cos(y)/2+\cos x)\sin y},\qquad X_1(x,y) = \vect{\sigma}{0},\qquad X_2(x,y) = \vect{0}{\sigma}.$$
Recall that it is a modification of the standard cellular flow~\cite{Brue2024EnhancedFlows}, such that the deterministic dynamics of the system ($\sigma=0$) are attracting for almost all initial conditions (see Figure~\ref{fig:cellular-a}). In the case of the standard cellular flow, which is Hamiltonian, it is more straightforward to show that for any $\sigma>0$ small enough, the Lyapunov exponent $\lambda$ is positive. This chaotic behaviour can be seen as the result of a shearing effect along the level curves of the corresponding Hamiltonian. This is clearly not the case in our system which displays a negative Lyapunov exponent $\lambda$ for $\sigma>0$ close to zero. We however show that the system can be destabilised if the noise $\sigma>0$ is increased enough (see Figure~\ref{fig:cellular-b}).

\subsection{The additive noise case}
Recall that we aim to compute the Lyapunov exponent $\lambda$ via the projective process $(\xi_t)_{t\geq 0} = (\varphi_t, s_t)_{t\geq 0}$ where

$$s_t = \frac{D\varphi_t v}{\|D\varphi_t v\|},\qquad v\in T_x \mathcal{M}.$$

We can compute the random differential equation for $s_t$ via formula~\eqref{eqn:s-eq-additive}. Now identifying $\mathbf{P}(T_x\mathcal{M}) \simeq\mathbf{P}(\mathbb{R})\simeq \mathbb{T}$ and denoting $s_t = (\cos(\theta_t/2), \sin (\theta_t/2))$, we find

\begin{equation}\label{eqn:cellular-h-def}
\diff{\theta_t}{t} = \frac{1}{2} \left(\sin\theta_t (4 \cos x_t \cos y_t-\cos 2 x_t+\cos 2 y_t)-4 \sin x_t \sin
   y_t\right)=:h(x_y, y_t,\theta_t),
\end{equation}
and
\begin{align*}
\lambda &= \int_{\mathbb{T}^3}Q(x,y,\theta) \tilde{\mu}(\d x, \d y, \d \theta)\\ Q(x,y,\theta) &= \frac{1}{4} (\cos 2x +\cos 2y-\cos\theta (4 \cos x \cos y+\cos 2
   y-\cos 2x)),
\end{align*}
where $\tilde{\mu}$ is the unique stationary distribution of the process $(\varphi_t, \theta_t)_{t\geq 0}$. This uniqueness follows (via Corollary~\ref{cor:ergodicity}) from the fact that the base process $(\varphi_t)_{t\geq 0}$ is clearly ergodic and controllable on the compact state space $\mathcal{M}$ as a diffusion with additive noise and the verification of the parabolic H\"ormander condition~\ref{(PH)} for the process $(\varphi_t, \theta_t)_{t\geq 0}$ in the file \texttt{cellular/hypoelltipticity.ipynb} at~\cite{Huggzz/Enclosure-of-Lyapunov-exponents}. 


\begin{proof}[Proof of Theorem~\ref{thm:intro-cellular}] Following the strategy laid out in Section~\ref{sec:cap}, for the parameter $\sigma = \sqrt{2}$, we compute a numerical solution

$$\bar{u} \in\mathrm{Span}\left(\{\cos kx, \sin kx\}_{k=0}^{25}\otimes \{\cos my, \sin my\}_{m=0}^{25}\otimes \{\cos n\theta, \sin n\theta\}_{n=0}^{550}\right)$$
(which can be found in the file \texttt{cellular/ubar}~\cite{Huggzz/Enclosure-of-Lyapunov-exponents}) to the Poisson problem

\begin{equation}\tilde{\mathcal{L}} u = Q - \lambda,\label{eqn:cellular-poisson}\end{equation}
where $h$ is as in~\eqref{eqn:cellular-h-def} and
\begin{equation}\label{eqn:cellular-L-def}
\tilde{\mathcal{L}} = \left(\frac{\cos x}{2}-\cos y\right)\partial_x+\left(\frac{\cos y}{2}+\cos x\right)\partial_{y}+h(x,y,\theta)\partial_{\theta}+\frac{\sigma^2}{2}(\partial_{xx}+\partial_{yy}).
\end{equation}
We define $\bar{\lambda}$ and $\bar{Q}$ as outlined in Section~\ref{sec:cap}, i.e.,

\begin{equation}\label{eqn:barlambda-def}
    \bar{\mathcal{\lambda}} := P_0 (Q -\mathcal{L}\bar{u}),
\end{equation}
and
\begin{equation}\label{eqn:barQ-def}
    \bar{Q} := \mathcal{L} \bar{u}+\bar{\lambda} \approx Q.
\end{equation}
Choosing $W = \tilde{W} = 1$, we find that

$$|\lambda - \bar{\lambda}|\leq \|Q- \bar{Q}\|_{\infty},$$
which can be evaluated directly by summing the absolute values of the Fourier coefficients of $Q - \bar{Q}$. The proof is performed in the file \texttt{cellular/proof} at~\cite{Huggzz/Enclosure-of-Lyapunov-exponents}
\end{proof}

\begin{rmk}
Since the operator $\mathcal{L}$ defined in~\eqref{eqn:cellular-L-def} fulfils the parabolic H\"ormander condition~\ref{(PH)}, the process $(\xi_t)_{t\geq 0}$ is strong Feller and $u$ is expected to be of class $\mathcal{C}^{\infty}$. Therefore, its Fourier coefficients should decay faster than algebraically, and this is why we can get a tight enclosure of $\lambda$ in Theorem~\ref{thm:intro-cellular}. However, many more basis functions are still needed in the non-elliptic direction of the problem (the $\theta$ variable).
\end{rmk}

\begin{rmk}
    One can actually state~\cite[Theorem 4.2.6]{Arnold1998RandomSystems} for this two-dimensional example that for $(\Pb\times \mu)$-almost every $(\omega, x)\in \Omega \times \mathcal{M}$, there exists $v^{*}\in\mathbf{P}\mathcal{M}$ such that for all $v \in \mathbf{P}\mathcal{M}\backslash \{v^{*}\}$
    $$\lim_{t\to\infty}\frac{1}{t}\log\|D\varphi_t(\omega, x)v\| = \lambda,$$
    and
    $$\lim_{t\to\infty}\frac{1}{t}\log\|D\varphi_t(\omega, x)v^{*}\| =: \lambda_2 = 2\lambda_{\Sigma} - \lambda,$$
    where $\lambda_{\Sigma}$ is as in Section~\ref{subsec:volume} and $\lambda_2$ is the second Lyapunov exponent.
     Thus combining~\eqref{eqn:intro-lambda-cellular} and~\eqref{eqn:volume-lambda-cellular}, we obtain that
     $$\lambda_2 = 2\lambda_{\Sigma} - \lambda = -0.11756188842594 \pm 10^{-13}<0.$$
\end{rmk}




\subsection{A modification with multiplicative noise}

Now consider a modification of the cellular flow with sinks~\eqref{eqn:cellular-sde} where we replace the additive noise with a multiplicative one in the spirit of Baxendale~\cite{Baxendale1986AsymptoticDiffeomorphisms}. We now consider the vector fields

\begin{equation}\label{eqn:baxendale_vfs}
\begin{aligned}X_0(x, y) = \vect{(\cos(x)/2-\cos y)\sin x}{(\cos(y)/2+\cos x)\sin y},\qquad X_1&(x,y) = \vect{\sigma\sin x}{0},
\\X_2(x,y) = \vect{\sigma\cos x}{0},\qquad X_3(x,y) = \vect{0}{\sigma\sin y},\qquad &X_4(x,y) = \vect{0}{\sigma\cos y}.
\end{aligned}
\end{equation}

It turns out that the random dynamical system induced by these vector fields has the same \emph{statistical} one-point motion as the previously considered system, i.e.~they are described by the same Markov semi-group. That can be directly seen from the generator which is in both cases

$$\mathcal{L} = \left(\frac{\cos x}{2} - \cos y\right)\sin x \pdiff{}{x} +\left(\frac{\cos y}{2}+\cos x\right)\sin y \pdiff{}{y } +\frac{\sigma^2}{2}\left(\pdiff{^2}{x^2}+\pdiff{^2}{y^2}\right).$$

We will however see that the \emph{dynamical} properties of the stochastic flow $(\varphi_t)_{t \geq 0}$ are very different under this multiplicative noise. Indeed, we will show that the system now displays a negative Lyapunov exponent. Following~\cite{Baxendale1986AsymptoticDiffeomorphisms} (with minor modifications due to our choice of variables), we find that
$$\d \log\varrho_t = Q(x_t, y_t, \theta_t) \d t+\frac{\sigma}{2}\Big[(\cos\theta_t+1)(\cos x_t \d B^1_t -\sin x_t\d B^2_t)+(\cos\theta_t-1)(\cos y_t\d B^3_t - \sin y_t \d B^4_t)\Big],$$
where
$$Q(x,y, \theta) = \frac{1}{4} (\cos 2x +\cos 2y-\cos\theta (4 \cos x \cos y+\cos 2y-\cos 2x)) - \frac{\sigma^2}{2}\cos^2\theta.$$
Similarly, we find that
$$\tilde{\mathcal{L}} = \left(\frac{\cos x}{2} - \cos y\right)\sin x \partial_x +\left(\frac{\cos y}{2}+\cos x\right)\sin y \partial_y+\left(h(x,y,\theta)+\frac{\sigma^2}{2}\sin2\theta\right)\partial_{\theta}+\frac{\sigma^2}{2}\left(\partial_{xx}+\partial_{yy}+2\sin^2\theta\partial_{\theta\theta}\right),$$
where $h$ is as in~\eqref{eqn:cellular-h-def}. These calculations essentially follow from the Stratonovich-to-It\^o correction~\cite[Chapter 3.2]{Pavliotis2014StochasticApplications}.
Now as discussed earlier since $\mathcal{M} = \mathbb{T}^2$, the integrability condition~\eqref{eqn:mult-noise-int} clearly holds and we have that
$$\lambda = \int_{\mathbb{T}^3}Q(x,y, \theta) \tilde{\mu}(\d x, \d y, \d\theta).$$
\begin{thm}
    Consider the stochastic flow $(\varphi_t)_{t\geq 0}$ generated by the vector fields~\eqref{eqn:baxendale_vfs}, with $\sigma = \sqrt{2}$, then for the corresponding Lyapunov exponent $\lambda$, there are the following bounds
    $$\lambda = -0.6124 \pm 1.6\times 10^{-3}<0.$$
\end{thm}

\begin{proof}
    The weak H\"ormander condition is verified by checking that $\tilde{X}_1, \tilde{X}_2, \tilde{X_3}, \tilde{X}_4$ span $\mathbb{R}^3$ for $\theta \neq \pi/2 \pm \pi/2$ and $\tilde{X}_0, \tilde{X}_1, \tilde{X}_2, \tilde{X_3}, \tilde{X}_4, [\tilde{X}_0, \tilde{X}_2], [\tilde{X}_0, \tilde{X}_4], [[\tilde{X}_0, \tilde{X}_2], \tilde{X}_4]$ span $\mathbb{R}^3$ for $\theta = \pi/2 \pm \pi/2$. Since the generator $\mathcal{L}$ of the Markov process $(\varphi_t)_{t\geq 0}$ is uniformly elliptic, by Proposition~\ref{prop-san-martin}, there indeed exists a unique stationary measure $\tilde{\mu}(\d \xi) = \tilde{w}(\xi)\d \xi$ for the projective process $(\xi_t)_{t\geq 0} = (\varphi_t, s_t)_{t\geq 0}$. The application of the adjoint method as described in Section~\ref{sec:cap} is available at~\cite{Huggzz/Enclosure-of-Lyapunov-exponents} in the file \texttt{cellular/baxendale\_proof}.
\end{proof}

\section{The randomly forced pendulum}\label{sec:pendulum}

In this section, we discuss the example presented in Theorem~\ref{thm:intro-pendulum}, namely the randomly forced pendulum (or Josephson junction~\cite{Pavliotis2014StochasticApplications})

$$``\;\ddot{x}_t = - \kappa \sin x_t -\gamma \dot{x}_t +\sigma \dot{B}_t \;"$$
with gravitational acceleration $\kappa>0$, friction coefficient $\gamma>0$ and noise strength $\sigma>0$. It translates to the proper stochastic differential equation
\begin{equation}\label{eqn:pendulum-sde}
\begin{cases}
    \d x_t &= y_t\d t\\
    \d y_t &= -(\kappa\sin x_t +\gamma y_t)\d t +\sigma \d B_t
\end{cases}
\end{equation}
on $\mathcal{M} = \mathbb{T} \times \mathbb{R}$. Since this differential equation has $\mathcal{C}^{\infty}$-vector fields with bounded derivatives of all orders, it clearly induces a stochastic flow of diffeomorphisms $(\varphi_t)_{t\geq 0}$ (see \cite[Chapter II.4]{Kunita1984StochasticDiffeomorphisms}).

While this system is the stochastic version of a classical ordinary differential equation and also typically appears in applications related to superconductors~\cite{Kadlec1977OnJunctions}, its study from a mathematically rigorous point of view is so far limited.

Analogously to the previous example, this system can also be seen as an attracting version of a Hamiltonian pendulum (when the friction $\gamma=0$). We are thus interested in the possible chaotic behaviour of the system when $\sigma>0$. Note however that the mechanisms inducive of chaos for the randomly forced pendulum are not so clear: From Theorem~\ref{thm:intro-pendulum}, chaos appears as a consequence of the large noise $\sigma$ in the $y$-direction. However, this allows trajectories to spend more time away from the $x$-axis, where most of the dynamics seem to be happening from the deterministic point of view (see Figure~\ref{fig:pendulum-a}) and does not a priori push trajectories into more expanding regions of the system. The large noise may though act as a catalyst for a stretch-and-fold mechanism when the system crosses the $x$-axis (see Figure~\ref{fig:pendulum-b}).
\subsection{Ergodicity of the projective process}



In this section, we verify the ergodicity of the projective process. We achieve this by following Section~\ref{sec:proj-ergodicity} and first proving the ergodicity of the base process $(\varphi_t)_{t\geq 0}$ (Steps 1--3) and then showing that the projective process $(\xi_t)_{t\geq 0} = (\varphi_t, s_t)_{t\geq 0}$ satisfies the weak H\"ormander condition~\ref{(WH)}. As usual~\cite{CotiZelati2021ASystem, Hairer2021ConvergenceProcesses}, the former is shown via the construction of a judicious Lyapunov function, proving the process is Feller via the parabolic H\"ormander condition~\ref{(PH)} and showing the controllability of $(\varphi_t)_{t\geq 0}$ on $\mathcal{M} = \mathbb{T}\times\mathbb{R}$.

\step{1}[Choice of a Lyapunov function]
Recall that here the generator of the process $(\varphi_t)_{t\geq 0}$ is given by

$$\mathcal{L} = y \pdiff{}{x} -(\kappa\sin x + \gamma y)\pdiff{}{y} +\frac{\sigma^2}{2}\pdiff{^2}{y^2} = y \pdiff{}{x} -\kappa\sin x\pdiff{}{y} +L_y,$$
where
\begin{equation}L_y = -\gamma y\pdiff{}{y}+\frac{\sigma^2}{2}\pdiff{^2}{y^2}\label{eqn:L_y-def}\end{equation}
is the generator of an Ornstein--Uhlenbeck process with stationary density $e^{-\gamma y^2/\sigma^2}$~\cite{Pavliotis2014StochasticApplications}. Thus a natural choice for a Lyapunov function is $W(x,y) = e^{\gamma y^2/(2\sigma^2)}$; we have

\begin{align*}
    \mathcal{L}W(x,y) &= \frac{\gamma}{2\sigma^2}(-\gamma y^2+\sigma^2-2\kappa y\sin x)W(x,y)\\
    &= -c W(x,y) +\left[\frac{\gamma}{2\sigma^2}(-\gamma y^2+\sigma^2-2\kappa y\sin x)+c\right]e^{\gamma y^2/(2\sigma^2)}\\
    &\leq -c W(x,y) + d\Ind{C},
\end{align*}
where $c>0$ and
\begin{align}
    C &:=\left\{y\in \mathbb{R} \,\Big|\,\frac{\gamma}{2\sigma^2}(-\gamma y^2+\sigma^2+2\kappa |y|)+c>0\right\} \nonumber\\
    d &:=\sup_{y\in C}\left[\frac{\gamma}{2\sigma^2}(-\gamma y^2+\sigma^2+2\kappa|y|)+c\right]e^{\gamma y^2/(2\sigma^2)}. \label{eq:defd}
\end{align}

\step{2}[Parabolic H\"ormander condition~\ref{(PH)} for the base process]

We have that

$$X_0(x,y,\theta) = \vect{y}{-\kappa\sin x -\gamma y},\qquad X_2(x,y,\theta) = \vect{0}{\sigma},$$
and thus,
$$\left[X_2, X_0\right](x,y,\theta) = \vect{\sigma}{-\gamma\sigma},$$
which shows that $(\varphi_t)_{t\geq 0}$ fulfils the parabolic H\"ormander condition. 

\step{3}[Controllability of the base process] First note that from~\eqref{eqn:pendulum-sde}, it is clear that $y$ is completely controllable. Now, fixing $x_0\in \mathbb{T}$, note that for all $(\bar{x},\bar{y}) \in \mathbb{T}\times \mathbb{R}\backslash\{(0,0), (\pi,0)\}$, there exists $y_0\in \mathbb{R}$ such that $(\bar{x},\bar{y})\in O^+(x,_0,y_0)$ where $O^+(x,_0,y_0)$ is the forward orbit starting at $(x_0,y_0)$ of the deterministic system (represented in Figure~\ref{fig:pendulum-a}). This shows the controllability and thus ergodicity of the base process $(\varphi_t)_{t\geq 0}$ on $\mathcal{M} = \mathbb{T}\times \mathbb{R}$.


\step{4}[Weak H\"ormander condition for the projective process]
It now remains to be shown that the projective process $(\xi_t)_{t\geq 0} = (\varphi_t, s_t)_{t\geq 0}$ satisfies the weak H\"{o}rmander condition~\ref{(WH)} to apply Proposition~\ref{prop-san-martin}. Writing $s_t = (\cos(\theta_t/2), \sin(\theta_t/2))$, we find that on $\mathbf{P}\mathcal{M}\simeq \mathbb{T}\times \mathbb{R}\times \mathbb{T}$

$$\tilde{X}_0(x,y,\theta) = \begin{pmatrix}
y\\ -\gamma  y-\kappa  \sin x\\-\gamma  \sin \theta +\cos \theta -\kappa\cos x (\cos\theta +1)-1\end{pmatrix}, \qquad \tilde{X}_2(x,y,\theta)  = \begin{pmatrix}
0\\\sigma \\ 0\end{pmatrix}.$$

Again, we have that

$$\left[\tilde{X}_2, \tilde{X}_0\right](x,y,\theta) = \begin{pmatrix}
\sigma\\-\gamma\sigma \\ 0\end{pmatrix}.$$
It thus remains to show that among the remaining vector fields of the Lie algebra, some are non-zero in the direction $e_3$. Now

$$\left[\left[\tilde{X}_2, \tilde{X}_0\right],\tilde{X}_0\right](x,y,\theta) = \begin{pmatrix}
-\gamma  \sigma\\\sigma  \left(\gamma ^2-\kappa  \cos x\right) \\ \kappa  \sigma\sin{x}(\cos \theta +1) \end{pmatrix}.$$
so that the remaining points to check are reduced to the set $\{(x,y,\theta)\in \mathbb{T}\times\mathbb{R}\times\mathbb{T} \mid\theta = \pi\:\mathrm{or}\: x = \pi/2\pm\pi/2\}$.
For $\theta = \pi$, we have that
$$\tilde{X}_0(x,y,\pi) = \begin{pmatrix}
y\\ -\gamma  y-\kappa  \sin x\\-2\end{pmatrix}.$$
Finally, for $x = \pi/2 \pm \pi/2$, we have that

$$\left[\left[\left[\left[\tilde{X}_2, \tilde{X}_0\right],\tilde{X}_0\right],\tilde{X}_0\right],\tilde{X}_2\right](\pi/2 \pm \pi/2,y,\theta) =\mp \begin{pmatrix}
0\\0 \\  \kappa\sigma^2(1+\cos\theta)\end{pmatrix},$$
which finishes the proof of the weak H\"ormander condition~\ref{(WH)} and thus of the ergodicity for the projective process by Proposition~\ref{prop-san-martin}.





\subsection{The adjoint method for the pendulum}\label{sec:adjoint-pendulum}
In this section, we outline the proof of Theorem~\ref{thm:intro-pendulum}. Recall that we must numerically find approximate solutions $\bar{u}$ and $\bar{\lambda}$ to the Poisson problem
\begin{equation}\label{eqn:pendulum-poisson}
    \tilde{\mathcal{L}}u = Q - \lambda,
\end{equation}
where
\begin{equation}\label{eqn:pendulum-L}
\tilde{\mathcal{L}} = y\pdiff{}{x} - (1 + \gamma  \sin \theta -\cos \theta +\kappa (\cos \theta +1) \cos x)\pdiff{}{\theta}- \kappa\sin{x}\pdiff{}{y}+L_y
\end{equation}
and

$$Q(x,y,\theta) = \frac{1}{2} (\gamma  \cos \theta +\sin \theta -\gamma -\kappa  \sin \theta  \cos{x}).$$

We follow the procedure of Section~\ref{sec:cap} and need to choose a judicious Hilbert basis.
Note that (as it is well-known from the study Ornstein--Uhlenbeck process)

$$L_yH_m(\sqrt{\gamma}y/\sigma) = - \gamma m H_m(\sqrt{\gamma}y/\sigma),$$
where the $H_m$'s are Hermite polynomials (see~\cite[p. 775]{Abramowitz1970HandbookSeries}) and $L_y$ is as in~\eqref{eqn:L_y-def}. We thus naturally use the Hilbert basis 
$$\{\cos kx, \sin kx\}\otimes \{h_m(\sqrt{\gamma}y/\sigma)\}\otimes \{\cos n\theta, \sin n\theta\}, $$
where we choose the normalisation $h_m = H_m/\sqrt{2^m m!}$. With respect to this basis, we compute a numerical solution $\bar{u}$ to the Poisson problem~\eqref{eqn:pendulum-poisson} available in the file \texttt{pendulum/ubar} at~\cite{Huggzz/Enclosure-of-Lyapunov-exponents} for the parameters $\kappa = 2/3, \gamma = 1/4$ and $\sigma = 4$.

In the file \texttt{pendulum/compute\_sups.jl} at~\cite{Huggzz/Enclosure-of-Lyapunov-exponents}, we give enclosures for

$$\sup_{y\in \mathbb{R}}\left|\frac{h_m(\sqrt{\gamma}y/\sigma)}{W(x,y)}\right| = \sup_{z\in\mathbb{R}}|h_m(z)e^{-z^2/2}|,$$
for all $m=0, \ldots, 551$, where $W(x,y) = \tilde{W}(x,y) = e^{\gamma y^2/(2\sigma^2)}$. This is achieved by enclosing all the roots of $\partial_z(h_m(z)e^{-z^2/2})$ by combining the Fundamental Theorem of Algebra and the Intermediate Value Theorem. Finally, note that since

$$\mathcal{L} W \leq -cW +d,$$
with any $c>0$ and $d$ as in~\eqref{eq:defd}, we obtain the bound $\mu(W)\leq d/c$ by \cite[Theorem 4.3.(ii)]{Meyn1993StabilityProcesses}. Since $\mu$ is the marginal of $\tilde{\mu}$ on $\mathcal{M} = \mathbb{T} \times \R$, and we also have that

    $$\tilde{\mu}(W) = \int_{\mathbf{P}\mathcal{M}}W\d\tilde{\mu} = \int_{\mathcal{M}}W\d \mu\leq \frac{d}{c},$$
which is computed rigorously for a specific $c$ in the file \texttt{pendulum/Lyapunov\_bound.jl} at~\cite{Huggzz/Enclosure-of-Lyapunov-exponents}. The proof of Theorem~\ref{thm:intro-pendulum} is performed in the file \texttt{pendulum/proof.ipynb}, applying estimate~\eqref{eq:computable_estimate} of Proposition~\ref{prop:method}.

\begin{rmk}
    For this problem, the most challenging part was without a doubt the computation of a good approximate solution $\bar{u}$ to the Poisson problem~\eqref{eqn:pendulum-poisson}, which was done via the least-square problem~\eqref{eq:leastsquare}. This difficulty stems from the very non-elliptic nature of the operator~\eqref{eqn:pendulum-L}. Even accounting for the symmetry in the problem, we used $48,312,000$ basis functions just to confirm the sign of the Lyapunov exponent $\lambda$ (We used $105,568,462$ for the result of Theorem~\ref{thm:intro-pendulum}). A crucial feature of our computation was to first solve the problem on a small number of basis functions and iteratively increase the number of basis functions while taking the result of the previous step as an initial guess for \texttt{LSQR}~\cite{Paige1982LSQR:Squares}. Some additional improvements to speed up our calculation would be to parallelise our matrix-vector multiplication (with several CPUs, using for instance MKL or using a GPU). 
\end{rmk}

\section{The Hopf normal form}\label{sec:hopf}


We now come back to the system studied in Theorem~\ref{thm:intro-Hopf}, namely the Hopf normal form with additive noise

\begin{equation}\label{eqn:Hopf-sys}
    \d\vect{x_t}{y_t} = \left[\vect{\alpha & -\beta}{\beta &\alpha}\vect{x_t}{y_t} - \vect{a & b}{-b & a}\vect{x_t}{y_t}(x_t^2+y_t^2)\right]\d t + \sigma \d \vect{B^1_t}{B^2_t}. 
\end{equation}

This system was first introduced in~\cite{DeVille2011StabilitySystem} in the context of random dynamical systems as a simple perturbation of the deterministic Hopf normal form. See also~\cite{Dijkstra2008,Tantet2020, Wieczorek2009} for applications to ocean and laser dynamics. This system has been conjectured (see also~\cite{Doan2018HopfNoise}) to exhibit the \emph{shear-induced chaos} scenario proposed in~\cite{Lin2008Shear-inducedChaos}, for which the perturbation of a limit cycle induces a stretch-and-fold mechanism leading to chaos. This conjecture can be quantitatively translated into proving a transition from a negative to positive Lyapunov exponent $\lambda_b$ as $b$ increases. This problem was already attacked with computer-assisted techniques in~\cite{Breden2023Computer-AssistedSystems}, but this work only managed to establish such a transition for the \emph{conditioned} Lyapunov exponent, which in particular meant restricting the problem to a bounded domain. In contrast, the method presented in this paper allows to deal with the full problem on $\R^2$, while being much simpler than the rather involved computer-assisted techniques used in~\cite{Breden2023Computer-AssistedSystems}. Recently,      Baxendale~\cite{Baxendale2024LyapunovNoise}, and Chemnitz and Engel~\cite{Chemnitz2023PositiveNoise} did prove that the Lyapunov exponent $\lambda_b$ associated to~\eqref{eqn:Hopf-sys} becomes positive in the asymptotic regime $b\to \infty$. In this work, we tackle the opposite (and complementary) problem for $b$ small, and prove that $\lambda_b>0$ for an explicit parameter range of finite and explicit values of $b$, while also enclosing the critical value of $b$ at which the transition from negative to positive takes place. Note that Baxendale~\cite{Baxendale2024LyapunovNoise}
also made use of the adjoint method in his proof, though in a quite different fashion.

\subsection{Derivation of the Furstenberg--Khasminskii formula}\label{subsec:Hopf-FK}

As in~\cite{Baxendale2024LyapunovNoise, Chemnitz2023PositiveNoise, DeVille2011StabilitySystem, Doan2018HopfNoise}, we prepare the problem by reducing the formula for $\lambda$ to an integral over the ergodic measure $\mu$ of a lower-dimensional and uniformly elliptic diffusion.

In polar coordinates $x_t = r_t \cos\phi_t, y_t = r_t \sin\phi_t$, we have
\begin{numcases}{}
    \d r_t\hspace{-6pt}&$=\left(\alpha r_t -a r_t^3 +\dfrac{\sigma^2}{2r_t}\right)\d t +\sigma \d B^r_t$\nonumber\\
    \d \phi_t\hspace{-6pt}&$= (\beta + br^2_t)\d t +\dfrac{\sigma}{r_t}\d B^{\phi}_t$\label{eqn:Hop-phi-eq}
\end{numcases}
where $B^r_t, B^{\phi}_t$ are independent Brownian motions
$$\d\vect{B^r_t}{B^{\phi}_t} = \vect{\cos\phi_t & \sin\phi_t}{-\sin\phi_t &\cos\phi_t}\d \vect{B^1_t}{B^2_t}.$$
Taking $(\varrho_t, s_t)_{t\geq 0}$ as in~\eqref{eqn:rho-s-def} and similarly as in the previous derivations, posing $s_t = (\cos\theta_t, \sin\theta_t)$ we get
\begin{numcases}{}
    \d \varrho_t \hspace{-6pt}&$= \left(\alpha - 2a r_t^2 +\sqrt{a^2+b^2}r^2_t\sin(2\theta_t -\chi_0 - 2\phi_t)\right)\varrho_t\d t$ \nonumber\\
    \d \theta_t \hspace{-6pt}&$= \left(\beta 
+ 2b r_t^2 +\sqrt{a^2+b^2}r^2_t\cos(2\theta_t -\chi_0 - 2\phi_t)\right)\d t$ \label{eqn:Hop-theta-eq}
\end{numcases}
where $\chi_0 = \arccos(b/\sqrt{a^2+b^2})$. Now, letting $\psi_t = \theta_t - \chi_0/2 - \phi_t$ and substracting~\eqref{eqn:Hop-theta-eq} by~\eqref{eqn:Hop-phi-eq}, we obtain 
\begin{numcases}{}
    \d r_t \hspace{-6pt}&$=\left(\alpha r_t -a r_t^3 +\dfrac{\sigma^2}{2r_t}\right)\d t +\sigma \d B^r_t$\label{eqn:Hopf-r-eq}\\
    \d \psi_t \hspace{-6pt}&$= \left( b r_t^2 +\sqrt{a^2+b^2}r_t^2\cos(2\psi_t)\right)\d t-\dfrac{\sigma}{r_t}\d B^{\phi}_t$\label{eqn:Hopf-psi-eq}\\
    \d \varrho_t \hspace{-6pt}&$= \left(\alpha - 2a r_t^2 +\sqrt{a^2+b^2}r^2_t\sin(2\psi_t)\right)\varrho_t\d t\label{eqn:Hopf-rho-eq}$
\end{numcases}
Thus,
$$\lambda = \int_{\R_+\times \mathbb{T}}Q(r, \psi) \mu(\d r, \d \psi),$$
where $\mu$ is the unique stationary distribution of $(\xi_t)_{t\geq 0} = (r_t, \psi_t)_{t\geq 0}$ given by Eq. (\ref{eqn:Hopf-r-eq}--\ref{eqn:Hopf-psi-eq}) and
 $$Q(r, \psi) = \alpha - 2ar^2+\sqrt{a^2+b^2}r^2\sin2\psi.$$

Note that this is a slightly different formulation than in~\cite{Baxendale2024LyapunovNoise,Breden2023Computer-AssistedSystems,Chemnitz2023PositiveNoise,DeVille2011StabilitySystem,Doan2018HopfNoise} which factors in an additional symmetry. Our formulation turns out to be more practical for numerical purposes as it allows us to recover the polar coordinate Laplacian while still exploiting this symmetry. Indeed, we get that the generator of $(\xi_t)_{t\geq 0} = (r_t, \psi_t)_{t\geq 0}$ is

$$\mathcal{L} = (\alpha r - ar^3)\pdiff{}{r} +\left( b r^2 +\sqrt{a^2+b^2}r^2\cos2\psi\right)\pdiff{}{\psi} +\frac{\sigma^2}{2}\Delta$$
where $\Delta$ is the polar coordinate Laplacian

$$\Delta = \frac{1}{r}\pdiff{}{r}\left(r\pdiff{}{r}\right) +\frac{1}{r^2}\pdiff{^2}{\psi^2}.$$

Moreover, note that since the generator $\mathcal{L}$ of $(r_t, \psi_t)_{t\geq 0}$ is uniformly elliptic and e.g.~has $e^{ar^4/(4\sigma^2)}$ as a Lyapunov function, the process $(r_t, \psi_t)_{t\geq 0}$ indeed has a unique stationary measure $\mu$ on $\mathbb{R}^2$.

\subsection{The adjoint method for a single parameter}

Before proving Theorem~\ref{thm:intro-Hopf} in the next subsection, we first explain how to rigorously enclose the Lyapunov exponent of~\eqref{eqn:Hopf-sys} for a fixed value of $b$. 

We first need to numerically solve the Poisson problem
\begin{equation}\label{eqn:Hopf-poisson}
    \mathcal{L} u = Q -\lambda,
\end{equation}
where $\mathcal{L}$ and $Q$ are as in Section~\ref{subsec:Hopf-FK}. This equation is ``polynomial'' with respect to these new polar coordinates $(r_t, \psi_t)$. This allows us to choose a suitable basis in polar coordinates, rendering our problem sparse. Here, we use the following basis in the Hermite--Laguerre family (or a scaling of it)
$$H^m_n(r, \psi) := r^{|m|}L^{(|m|)}_n(r^2)\begin{cases}\sin{m\psi}\qquad m>0\\ \cos{m\psi}\qquad m\leq 0
\end{cases}$$
where $L^{(m)}_n$ denotes the generalised Laguerre polynomial of degree $n$ and parameter $m$ (see~\cite[p. 775]{Abramowitz1970HandbookSeries}). We choose the normalisation convention of numerical analysis and denote $h^m_n = \sqrt{n!/(m+n)!}H^m_n$ for all $m\in \mathbb{Z}, n\in\mathbb{N}$. Though this basis is well-known to physicists and has been studied mathematically in~\cite{Ito1952ComplexIntegral,Chen2014OnOperators}, it has not been used much for numerical purposes (see~\cite{Breden2024} for a recent example in a computer-assisted context).

\begin{rmk}
    Note that our choice of basis does not perfectly match the ``physics'' of our problem. Indeed, our basis is orthogonal with respect to the Gaussian measure $e^{-r^2}r\d r \d \psi$. One could expect better results with a basis orthogonal with respect to the Freud weight $e^{-(ar^4/2-\alpha r^2)/\sigma^2}r\d r \d \psi$ which is the marginal of $\mu$ onto $r$. However, the construction of such a basis is beyond the scope of this work.
    Nevertheless, our choice of basis does not prevent the application of our method as the only rigorous computation needed is essentially the application of the operator $\mathcal{L}$ to $\bar{u}$. This is in contrast with computer-assisted proofs for the solutions of PDEs which would typically necessitate some \textit{a priori} estimates on the inverse of some part of the operator $\mathcal{L}$ with respect to a suitable basis.
\end{rmk}

For $\alpha = a = 4, \sigma = \sqrt{2}$ and $b = 21.5381$, in the file \texttt{Hopf/proof} at~\cite{Huggzz/Enclosure-of-Lyapunov-exponents}, with respect to this basis, we compute an ansatz $\bar{u}$ to~\eqref{eqn:Hopf-poisson} represented in Figure~\ref{fig:Hopf_ubar} and obtained the following result.

\begin{thm}
    For the Hopf normal form with additive noise~\eqref{eqn:Hopf-sys}, with the parameters $\alpha = a = 4, \sigma = \sqrt{2}$ and $b = 21.5381$, the Lyapunov exponent $\lambda$ is positive.
\end{thm}


\begin{figure}[h]
    \centering
    \includegraphics[width=0.5\textwidth]{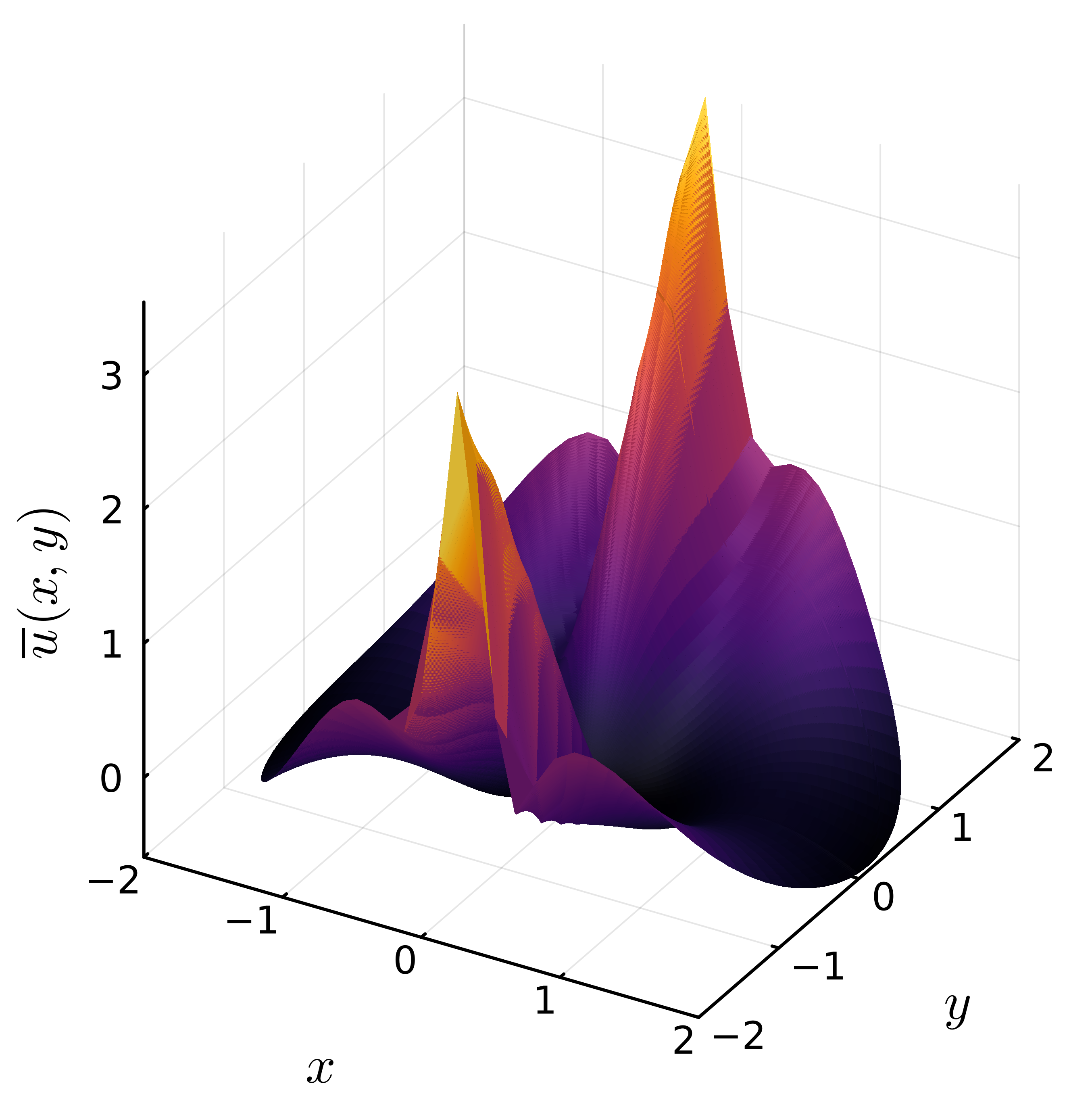}
    \caption{Plot of the numerical solution $\bar{u}$ to Eq.~\eqref{eqn:Hopf-poisson} with $a = \alpha = 4, \sigma = \sqrt{2}$ and $b =21.5381$.\label{fig:Hopf_ubar}}
\end{figure}

Finally, note that the choice of $W = \tilde{W}$ is quite straightforward as from~\eqref{eqn:Hopf-r-eq}, the marginal of $\mu$ on the radial direction is simply

$$\mu(\d r , \mathbb{T}) = \frac{re^{-(ar^4/2-\alpha r^2)/\sigma^2}}{Z}\d r, \qquad Z = \int_0^{\infty}re^{-(ar^4/2-\alpha r^2)/\sigma^2}\d r.$$
Therefore, choosing for instance $W$ of the form

$$W(r) = e^{c r^4/(2\sigma^2)}, \qquad 0<c<a,$$
we can compute the integral

$$\mu(W) = \frac{1}{Z}\int_0^{\infty}re^{((c-a) r^4/2 + \alpha r^2)/\sigma^2}\d r = \frac{\sqrt{a}e^{\alpha^2 c/2a\sigma^2(a-c)}(1+\mathrm{erf}(\alpha/(\sigma\sqrt{2(a-c)})))}{\sqrt{a-c}(1+\mathrm{erf}(\alpha/(\sigma\sqrt{2a})))},\qquad \alpha>0,$$
where the function $\mathrm{erf}$ can be rigorously enclosed using
\texttt{Arb}~\cite{Johansson2017Arb:Arithmetic} which implements the evaluation of special functions in ball arithmetic. Here, we chose $c = a/2$.
As in Section~\ref{sec:adjoint-pendulum}, we computed the suprema of $|H^m_n/W|$ in the file \texttt{Hopf/compute\_sups.jl} at~\cite{Huggzz/Enclosure-of-Lyapunov-exponents}.



\subsection{Continuation with respect to the shear \texorpdfstring{$b$}{b}}\label{subsec:hopf-cont}

We now show how our method can be combined with continuation techniques recently developed in~\cite{Breden2023AExpansions}, which we implement in the file \texttt{Hopf/continuation} at~\cite{Huggzz/Enclosure-of-Lyapunov-exponents}. We now want to enclose the parameter-dependent Lyapunov exponent

$$\lambda_b = \int_{\mathbb{R}_+\times \mathbb{T}}Q_b \d\mu_b = \int  _{\mathbb{R}_+\times \mathbb{T}}\left(\alpha - 2ar^2+\sqrt{a^2+b^2}r^2\sin2\psi\right)\mu_b(\d r, \d \psi),$$
for all $b$ in a parameter range $[b_-,b_+]$. We first define an approximate solution $\bar{u}_b$ for the Poisson problem $\mathcal{L}_b u_b = Q_b - \lambda_b$ (where these are all now parameter-dependent) for all $b\in[b_-, b_+]$ of the form

$$\bar{u}_b(r, \psi) = \sum_{k=0}^K\sum_{m, n} c_{k, m, n} T_k(b)h^m_n(r, \psi),$$
where $T_k$ denotes the Cheybyshev polynomial of the first kind of degree $k$ on $[b_-, b_+]$. This is done by computing approximate solutions for $K+1$ values of $b$ in $[b_-, b_+]$ (taken at the Chebyshev nodes) and interpolating them.

We then need to rigorously apply $\mathcal{L}_b$ to $\bar{u}_b$, for $b\in[b_-, b_+]$. This could be done \emph{exactly}, i.e.~$\mathcal{L}_b \bar{u}_b$ would have a finite expansion in the (Hermite--Laguerre)--Chebyshev basis, except for the fact that $Q_b$ is not polynomial in $b$. We therefore first build a degree $K$ polynomial $p$ (in the Chebyshev basis) approximating $\sqrt{a^2+b^2}$ and compute an error bound $\varepsilon$ such that
$$\sup_{b\in [b_-,b_+]}\left|p(b)-\sqrt{a^2+b^2}\right|<\varepsilon.$$
We then denote $\hat{Q}_b := \alpha - 2ar^2+p(b)r^2\sin2\psi,$
which is now polynomial in $b$ and should approximate $Q_b$ well in $b$, and
$$\hat{\mathcal{L}}_b := (\alpha r - ar^3)\partial_r +\left( b r^2 +p(b)r^2\cos 2\psi\right)\partial_{\psi} +\frac{\sigma^2}{2}\Delta.$$
We now have that $\hat{\mathcal{L}}\bar{u}_b - \hat{Q}_b$ is polynomial of degree $2K$ in $b$ and given that it is written in terms of Chebyshev coefficients, it is straightforward to give an upper bound on
$$\sup_{b\in[b_-, b_+]}\|(\hat{\mathcal{L}}_b\bar{u}_b - \hat{Q}_b + \bar{\lambda}_b)/W\|_{\infty},$$
where $\bar{\lambda}_b = P_0(\hat{Q}_b - \hat{\mathcal{L}}_b\bar{u}_b)$.

We now incorporate the errors coming from approximating $Q_b$ by $\hat{Q}_b$

\begin{align*}
    |\bar{\lambda}_b-\lambda_b|&\leq \left|\bar{\lambda}_b - \int_{\R_+\times\mathbb{T}}Q_b\d \mu_b\right|\\
    &\leq \left|\bar{\lambda}_b-\int_{\R_+\times\mathbb{T}}\hat{Q}_b\d \mu_b\right|+\left|\int_{\R_+\times\mathbb{T}}(Q_b-\hat{Q}_b)\d \mu_b\right|\\
    &\leq \left|\int_{\R_+\times\mathbb{T}}(\bar{\lambda}_b -\hat{Q}_b+\hat{\mathcal{L}}\bar{u}_b)\d \mu_b\right|+\left|\int_{\R_+\times\mathbb{T}}(\hat{\mathcal{L}}\bar{u}_b)\d \mu_b\right|+\left|\int_{\R_+\times\mathbb{T}}(Q_b-\hat{Q}_b)\d \mu_b\right|.
\end{align*}
We then have,

$$\left|\int_{\R_+\times\mathbb{T}}(\bar{\lambda} -\hat{Q}_b+\hat{\mathcal{L}}\bar{u}_b)\d \mu_b\right|\leq \mu(W)\sup_{b}\|(\hat{\mathcal{L}}_b\bar{u}_b - \hat{Q}_b + \bar{\lambda}_b)/W\|_{\infty} =: \delta_1,$$
which can be bounded explicitly. Similarly, using the fact that $\mathcal{L}\bar{u}_b$ is $\mu_b$-mean zero

\begin{align*}
    \left|\int_{\R_+\times\mathbb{T}}(\hat{\mathcal{L}}\bar{u}_b)\d \mu_b\right| &= \left|\int_{\R_+\times\mathbb{T}}(\hat{\mathcal{L}}\bar{u}_b-\mathcal{L}\bar{u}_b)\d \mu_b\right|\\
    &\leq \left|\int_{\R_+\times\mathbb{T}}(\sqrt{a^2+b^2}-p(b))r^2\sin(2\psi)\partial_{\psi}\bar{u}_b\d \mu_b\right|\\
    &\leq \varepsilon\mu(W)\sup_{b\in [b_-,b_+]}\|(r^2\sin(2\psi)\partial_{\psi}\bar{u}_b)/W\|_{\infty}=:\delta_2,
\end{align*}
which also can be bounded explicitly,
and finally
$$\left|\int_{\R_+\times\mathbb{T}}(Q_b-\hat{Q}_b)\d \mu_b\right|\leq \varepsilon\mu(W) \|r^2\sin(2\psi)/W\|_{\infty}=:\delta_3,$$
so that, for all $b\in [b_-, b_+]$,
$$|\lambda_b-\bar{\lambda}_b|\leq \delta_1+\delta_2+\delta_3.$$



\begin{rmk}
    The two important parameters of the system are $\alpha$ and $b$ as the sign of the Lyapunov exponent is invariant under a proper rescaling with respect to the other parameters (see~\cite[Proposition 3.1]{Baxendale2024LyapunovNoise} for instance). Our method could of course be naturally extended to compute a full rigorous two-parameter bifurcation diagram for the Lyapunov exponent. While this would be for a bounded parameter region, one could hope to combine it with the complementary results of~\cite{Baxendale2024LyapunovNoise,Chemnitz2023PositiveNoise} (with a more precise treatment to make this analysis more quantitative) to obtain a diagram for $(\alpha, b)\in \mathbb{R}^2$.
\end{rmk}



\section{Proof of chaos for the stochastic Duffing--van der Pol system}
\label{sec:Duffing}

In this Section, we explain how to apply the adjoint method to system~\eqref{eqn:bax_vfs}, leading to the proofs of Theorem~\ref{thm:intro-Duffing} and Theorem~\ref{thm:Duffing-average-sign}.
As shown in~\cite[p. 263]{Baxendale2004StochasticEquation}, the sign of the Lyapunov exponent of the system~\eqref{eqn:bax_vfs} only depends on the quantities $\beta\nu^2/\sigma^2$ and $b\nu/a$, and we can assume without loss of generality $\nu/a =  \nu/\sigma = 1$. Henceforth, we thus consider the system~\eqref{eqn:gen-sde} on $\mathcal{M} = \R^2\backslash\{0\}$ with the vector fields

\begin{equation}
\begin{aligned}X_0(x, y) = \left(\frac{\beta}{2}-\frac{b}{8}(x^2+y^2)\right)\vect{x}{y}-\frac{3}{8}(x^2+y^2)\vect{y}{-x}&,
\\X_1(x,y) = \frac{1}{2\sqrt{2}}\vect{x}{-y},\qquad X_2(x,y) = \frac{1}{2\sqrt{2}}\vect{y}{x},\qquad X_3(x,y) =& \frac{1}{2}\vect{y}{-x},
\end{aligned}
\end{equation}
depending only on the parameters $b$ and $\beta$.

Note that when this system is considered on full space $\mathbb{R}^2$, it displays a multiplicity of stationary distributions with two ergodic components $\{0\}$ and $\R^2\backslash\{0\}$. However, as will become clear, this does not prevent us from applying the adjoint method. While the Lyapunov exponent associated to the system linearised at $0$ can be found explicitly, finding bounds on the Lyapunov exponent for the system on $\R^2\backslash\{0\}$ is more challenging. Indeed, the degeneracy of the diffusion around $0$ creates an additional difficulty which we remove in a rather interesting fashion.

To the best of our knowledge, the chaotic behaviour of the Duffing--van der Pol system~\eqref{eqn:Duffing} was previously only shown in~\cite{Baxendale2002LyapunovSystems} for $b = 0$ and small $\varepsilon>0$ (and thus for small $b>0$ by a continuity argument).

\subsection{Derivation of the Furstenberg-Khasminskii formula}

Following~\cite{Baxendale2004StochasticEquation}, writing the system in polar coordinates, i.e.~$(x_t,y_t) = (r_t\cos\phi_t,\sin\phi_t)$ and passing to It\^o integration gives
\begin{numcases}{}
    \d r_t \hspace{-6pt}&$=\left(\dfrac{\beta r_t}{2} - \dfrac{br^3_t}{8} +\dfrac{3r_t}{16}\right)\d t +\dfrac{1}{2\sqrt{2}}r_t\left(\cos(2\phi_t) \d B^1_t+\sin(2\phi_t)\d B^2_t\right)$,\nonumber\\
    \d \phi_t \hspace{-6pt}&$ = \dfrac{r^2_t}{8}\d t +\dfrac{1}{2\sqrt{2}}\left(-\sin(2\phi_t)\d B^1_t +\cos(2\phi_t)\d B^2_t\right)-\dfrac{1}{2}\d B^3_t$.\label{eqn:bax-phi-eq}
\end{numcases}
%
Now, taking $(\varrho_t, s_t)_{t\geq 0}$ as in~\eqref{eqn:rho-s-def} and posing $s_t = (\cos \theta_t, \sin\theta_t)$, one obtains
\begin{numcases}{}
    \d \log\varrho_t \hspace{-6pt}&$= \left(\dfrac{\beta}{2}+\dfrac{1}{8}-\dfrac{b r^2_t}{8}-\dfrac{br^2}{4}\cos^2(\theta_t-\phi_t)+\dfrac{3r^2_t}{8}\sin(2\theta_t-2\phi_t)\right)\d t$\nonumber\\
    &\quad$+\dfrac{1}{2\sqrt{2}}\left(\cos(2\theta_t) \d B^1_t+\sin(2\theta_t)\d B^2_t\right)$, \nonumber\\
    \d \theta_t \hspace{-6pt}&$= r^2_t\left(\dfrac{1}{8}+\dfrac{b}{8}\sin(2\theta_t - 2\phi_t)+\dfrac{3}{4}\cos^2(\theta_t-\phi_t)\right)\d t$\nonumber\\
    &\quad$+\dfrac{1}{2\sqrt{2}}\left(-\sin(2\theta_t)\d B^1_t +\cos(2\theta_t)\d B^2_t\right) -\dfrac{1}{2}\d B^3_t$. \label{eqn:bax-theta-eq}
\end{numcases}
%
Defining independent Brownian motions $B^r_t, B^{\phi}_t$ as
$$\d\vect{B^r_t}{B^{\phi}_t} = \vect{\cos2\phi_t & \sin 2\phi_t}{-\sin2\phi_t &\cos2\phi_t}\d \vect{B^1_t}{B^2_t},$$
and $\psi_t = 2(\theta_t -\phi_t)$, by subtracting~\eqref{eqn:bax-theta-eq} by~\eqref{eqn:bax-phi-eq}, one obtains
\begin{numcases}{}
    \d r_t \hspace{-6pt}&$=\left(\dfrac{\beta r_t}{2}-\dfrac{br^3_t}{8}+\dfrac{3r_t}{16}\right)\d t+\dfrac{1}{2\sqrt{2}}r_t\d B^r_t$\label{eqn:bax-r-eq}\\
    \d \psi_t \hspace{-6pt}&$= r^2_t\left(\dfrac{b}{4}\sin\psi_t+\dfrac{3}{4}\left(1+\cos\psi\right)\right)\d t +\dfrac{1}{\sqrt{2}}
    \left(-\sin\psi_t\d B^r_t+(\cos\psi_t-1)\d B^{\phi}_t\right)$,\label{eqn:bax-psi-eq}\\
    \d \log\varrho_t \hspace{-6pt}&$= \left(\dfrac{\beta}{2}+\dfrac{1}{8}-\dfrac{b r^2_t}{8}\left(2+\cos\psi_t\right)+\dfrac{3r^2_t}{8}\sin\psi_t\right)\d t+\dfrac{1}{2\sqrt{2}}\left(\cos\psi_t\d B^r_t+\sin\psi_t\d B^{\phi}_t\right)$.\label{eqn:bax-rho-eq}
\end{numcases}
The existence and uniqueness of a stationary measure $\mu_{\beta}$ on $(0,\infty) \times \mathbb{T}$ for the process $(r_t,\psi_t)_{t\geq 0}$ for $\beta>-1/4$ is shown in~\cite{Baxendale2004StochasticEquation}. Furthermore, note that from~\eqref{eqn:bax-r-eq}, the process $(r_t)_{t\geq 0}$ is now decoupled so that the marginal of $\mu_{\beta}$ onto $\R^*_+:=(0,\infty)$ is given by
\begin{equation}\mu_{\beta}(\d r, \mathbb{T}) = \frac{1}{Z_{\beta}}r^{1+8\beta}\exp\left(-br^2\right), \qquad Z_{\beta} = \frac{1}{2}b^{-(1+4\beta)}\Gamma(1+4\beta).\label{eqn:mu-marg}
\end{equation}

\subsection{An adaptation of the adjoint method}
\label{sec:adj_Duffing}

We now explain how to adapt the ideas from Section~\ref{sec:cap_framework} to this specific problem. The generator of the process $(r_t,\psi_t)_{t\geq 0}$ is given by
$$\mathcal{L}_{\beta} = -\frac{r}{16} \left(2 b r^2-8 \beta -3\right) \partial_r+\frac{r^2}{16} \partial_{rr}+\frac{r^2}{4} (b \sin \psi+3 \cos \psi +3)\partial_{\psi}-\frac{r}{4} \sin \psi \partial_{r\psi}+\frac{1}{2} (1-\cos \psi )\partial_{\psi\psi},$$
and we want to approximately solve the Poisson problem
\begin{equation}\label{eqn:bax-poisson}
    \mathcal{L}_{\beta}u = Q_{\beta} - \mathcal{\lambda}_{\beta}, \qquad \lambda_{\beta} :=\int_{\R^*_+\times\T}Q_{\beta}\d\mu_{\beta},
\end{equation}
where, from~\eqref{eqn:bax-rho-eq}, the infinitesimal growth rate along the projective flow is
$$Q_{\beta}(r,\psi) = \frac{\beta}{2}+\frac{1}{8}-\frac{b}{8}r^2(2+\cos\psi)+\frac{3}{8}r^2\sin\psi.$$
We consider the problem for a fixed $\beta$, the continuation procedure allowing to get a control that is uniform in $\beta$ being the same as in Section~\ref{subsec:hopf-cont} (and even simpler in practice, since here $\mathcal{L}_{\beta}$ and $Q_{\beta}$ are already polynomials in $\beta$).

A natural basis to consider in order to find an approximate solution $\bar{u}$ could be
$$\{\cos m\psi,\sin m\psi\}_{m\in\N}\otimes \{L_n(br^2)\}_{n\in\N}$$
built from the Laguerre polynomials $\{L_n\}_{n\in\N}=\{L^{(0)}_n\}_{n\in\N}$. However, one rapidly realises that such a polynomial basis is not quite suitable. 


Indeed, let us first look at a simpler Poisson problem, chosen somewhat arbitrarily, namely
$$\mathcal{L}_{\beta}v = \frac{b}{8}r^2 - \mathcal{I}, \qquad \text{where } \mathcal{I} = \int_{\mathbb{R}^*_+\times\T}\frac{b}{8}r^2\mu_{\beta}(\d r, \d\psi) = \frac{\beta}{2}+\frac{1}{8} \mbox{ by~\eqref{eqn:mu-marg}}.$$
Then, one can compute that the solution to this equation is given (up to an additive constant) by $v(r,\psi) = -\log r$. This function is singular around $r=0$, and we expect the solution $u$ to Eq.~\eqref{eqn:bax-poisson} to have a similar behaviour near the origin, therefore it is hopeless to try to approximate $u$ with a purely polynomial basis. However, it turns out that one can simply add this single extra $\log$ function to our basis in order to resolve the Poisson problem~\eqref{eqn:bax-poisson} numerically up to reasonable accuracy. That is, we now consider the space
$$\mathcal{V} = \{\log r\}\oplus\mathcal{H}_{M,N}, \qquad \mathcal{H}_{M,N}=\mathrm{Span}\left\{\{\cos m\psi,\sin m\psi\}_{m=0}^M\otimes\{L_n(br^2)\}_{n=0}^N\right\}.$$
Note that, even if $\mathcal{V}$ contains the singular function $\log$, we have $\mathcal{L}(\mathcal{V}) \subset \mathcal{H}_{M+1,N+1}$. Therefore, the addition of the log allows us to obtain a more accurate approximate solution $\bar{u}$ but does not hamper the a posteriori error estimates. Indeed, once a suitable approximate solution $\bar{u}\in \mathcal{V}$ is found, we only need to control $\mathcal{L}\bar{u}$ when using the adjoint method, which we detail below for this specific example.

\begin{figure}[h]
    \centering
    \includegraphics[width=0.5\textwidth]{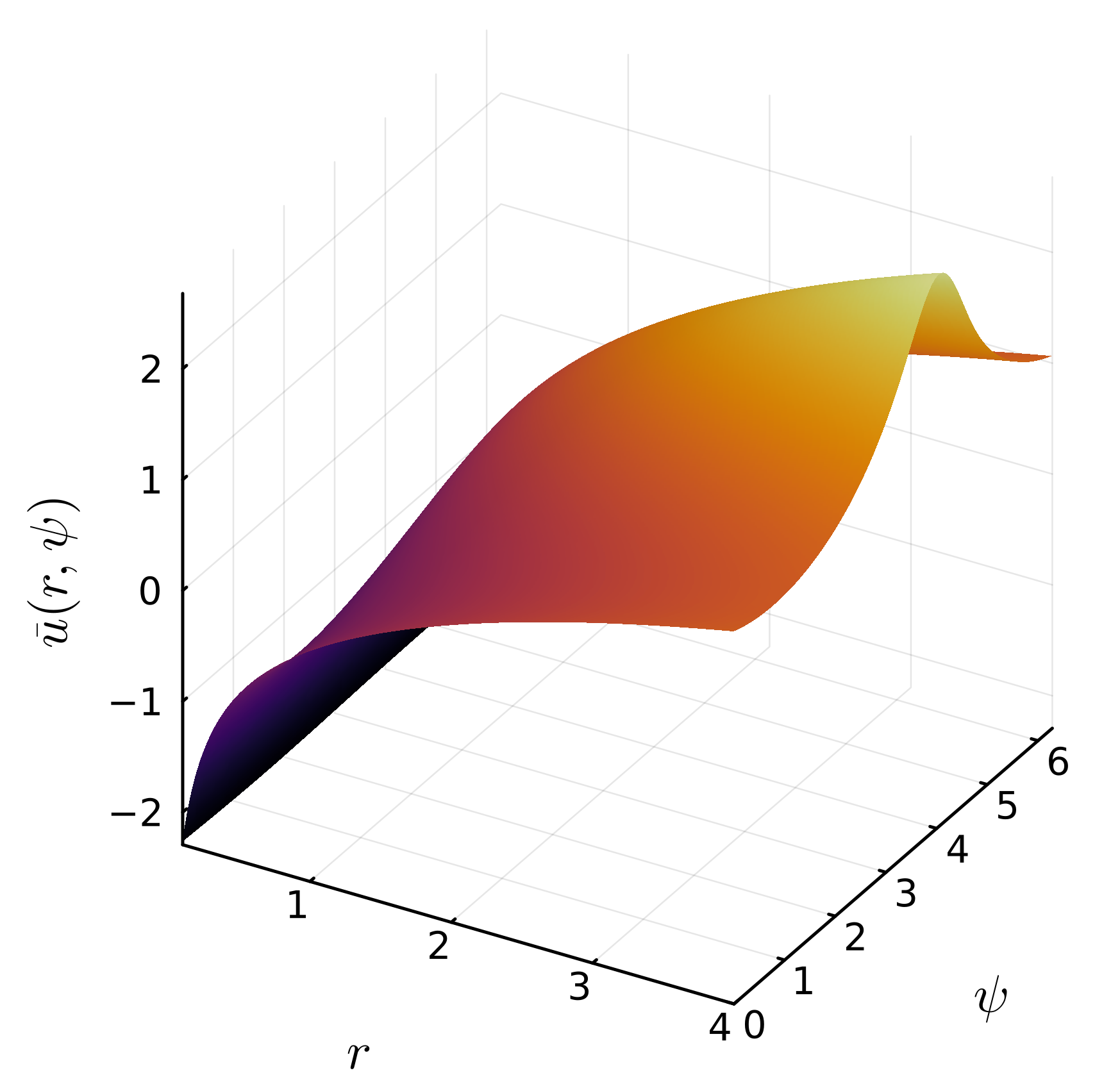}
    \caption{Plot of the numerical solution $\bar{u}$ to Eq.~\eqref{eqn:bax-poisson} with $b = 1.24$ and $\beta=-0.159$. \label{fig:Bax_ubar}}
\end{figure}

Here again $Q_\beta-\bar{Q}_\beta$ will not be bounded in $L^\infty$, so we consider $W = \exp(br^2/2)$, for which we have $\|L_n/W\|_{\infty} = 1$ for all $n\in \N$, along with $\mu(W) = 2^{1+4\beta}$. We then take $\tilde{W} = W\log^2r$, since for all $f
\in \mathcal{V}$, $|f(r,\psi)|/\tilde{W}(r)\rightarrow 0$ as $|f(r,\psi)|\to \infty$ and
$$\mathcal{L}\tilde{W} \leq c\tilde{W}, \qquad c = \sup_{r\in\R^*_+}\left|\frac{\log r \left(b r^2 \log r \left(-b r^2+8 \beta +4\right)+16 \beta +4\right)+2}{16 \log ^2r}\right|<\infty.$$
In particular, this applies if we take $f(r,\psi) = \log r$ such that, by~\cite[Proposition B.1]{Baxendale2024LyapunovNoise}
$$\int_{\R^*_+\times \mathbb{T}}(\mathcal{L}\log(r,\psi))\mu(\d r,\d\psi) = 0.$$
This guarantees that we can still apply the adjoint method rigorously as described in Section~\ref{sec:cap_framework}, i.e.~writing
$$(Q_\beta - \bar{Q}_\beta)(r,\psi) = g_0(br^2) +\sum_{m = 1}^{M+1}(g_m(br^2)\cos{m\psi} +g_{-m}(br^2)\sin{m\psi}),\qquad g_m =\sum_{n=0}^{N+1}\epsilon_{mn}L_n,$$
we have the following bound on the Lyapunov exponent $\lambda_{\beta}$
\begin{equation}\label{eqn:l1-est}
    |\lambda_{\beta}-\bar{\lambda}_{\beta}|\leq \mu_{\beta}(W)\|Q_\beta-\bar{Q}_{\beta}\|_{\infty} \leq 2^{1+4\beta}\sum_{|m|\leq M+1}\sum_{n=0}^{N+1}|\epsilon_{mn}|.
\end{equation}
Note that while the Lyapunov exponent $\lambda_{\beta}$ is not well defined at $\beta = -1/4$, as the system loses ergodicity, the quantity $\|Q_{\beta} - \bar{Q}_{\beta}\|_{\infty}$ is; which allows us to perform the continuation down to $\beta =-1/4$. 

\subsection{Improved error bounds}
\label{sec:ell2bounds}

The error bound~\eqref{eqn:l1-est} was obtained following the general strategy exposed in Section~\ref{sec:cap_framework}, and provides a control on $|\lambda-\bar{\lambda}|$ in terms of the $\ell^1$-norm of the vector $\epsilon=\left(\epsilon_{mn}\right)$ describing $Q_\beta-\bar{Q}_\beta$. However, we show here that better estimates can be obtained, using the same $\epsilon$. This is particularly relevant for this specific example, as $\lambda$ itself is already quite small in magnitude (see Figure~\ref{fig:Duff-a}), therefore rigorously determining its sign requires relatively tight error bounds. 

For $\beta>-1/8$, instead of~\eqref{eqn:l1-est}, we may use the following estimate:
\begin{align*}
    |\lambda_{\beta}-\bar{\lambda}_{\beta}|&=\left|\int_{\R^*_+\times\T}(Q_{\beta}-\bar{Q}_{\beta})(r,\psi)\mu_{\beta}(\d r, \d\psi)\right|,
    \qquad \mbox{from Eq.~\eqref{eqn:lambda-Q},}\\
    &\leq\int_{\R^*_+\times\T}\left|Q_{\beta}-\bar{Q}_{\beta}\right|(r,\psi)\mu_{\beta}(\d r, \d\psi)\\
    &\leq \sum_{|m|\leq M+1}\int_{\R^*_+\times\T}|g_m(br^2)|\mu_{\beta}(\d r,\d\psi)\qquad \mbox{using $|\cos{m\psi}|,|\sin{m\psi}|\leq 1$}\\
    &\leq \frac{1}{Z_{\beta}}\sum_{|m|\leq M+1}\int_{0}^{\infty}|g_m(br^2)|r^{1+8\beta}e^{-br^2}\d r\\
    &\leq \frac{b^{-1-4\beta}}{2Z_{\beta}}\sum_{|m|\leq M+1}\int_{0}^{\infty}|g_m(z)|z^{4\beta}e^{-z}\d z\qquad \mbox{with $z = br^2$}\\
    &\leq \frac{1}{\Gamma(1+4\beta)}\sum_{|m|\leq M+1}\left(\int_{0}^{\infty}|g_m(z)|^2e^{-z}\d z\right)^{1/2}\left(\int_{0}^{\infty}z^{8\beta}e^{-z}\d z\right)^{1/2}\qquad \mbox{by Cauchy--Schwarz}\\
    &\leq \frac{\Gamma(1+8\beta)^{1/2}}{\Gamma(1+4\beta)}\sum_{|m|\leq M+1}\left(\int_{0}^{\infty}|g_m(z)|^2e^{-z}\d z\right)^{1/2}.
\end{align*}
Thus, by orthonormality of the Laguerre polynomials $\{L_n\}_{n\in\N}$ with respect to $e^{-z}\d z$, we obtain
\begin{equation}\label{eqn:l2-est}
     \left|\lambda_{\beta}-\bar{\lambda}_{\beta}\right|\leq\frac{\Gamma(1+8\beta)^{1/2}}{\Gamma(1+4\beta)}\sum_{|m|\leq M+1}\left(\sum_{n=0}^{N+1}\epsilon_{mn}^2\right)^{1/2}.
\end{equation}
Comparing~\eqref{eqn:l1-est} and~\eqref{eqn:l2-est}, we trivially have that
\begin{align*}
    \left(\sum_{n=0}^{N+1}\epsilon_{mn}^2\right)^{1/2} \leq \sum_{n=0}^{N+1} \vert\epsilon_{mn}\vert,
\end{align*}
but the $\ell^2$-norm can actually be much smaller than the $\ell^1$-norm, especially when the coefficients $\epsilon_{mn}$ do not decay rapidly with $n$ (see Figure~\ref{fig:eps}). For instance, in the extreme case where $\epsilon_{mn}=\epsilon_{m}$ does not depend on $n$, one has
\begin{align*}
    \left(\sum_{n=0}^{N+1}\epsilon_{mn}^2\right)^{1/2} \approx \frac{1}{\sqrt{N}} \sum_{n=0}^{N+1} \vert\epsilon_{mn}\vert.
\end{align*}
Therefore, we expect~\eqref{eqn:l2-est} to give a sharper estimate than~\eqref{eqn:l1-est}, at least when the constants $2^{1+4\beta}$ and $\Gamma(1+8\beta)^{1/2}/\Gamma(1+4\beta)$ have the same order of magnitude, which is the case as long as $\beta>-1/8$ is not too close to $-1/8$ (see Figure~\ref{fig:constants}).

\begin{figure}[h]
\captionsetup[subfigure]{justification=centering}
\begin{subfigure}{0.5\textwidth}
\includegraphics[height=0.65\linewidth]{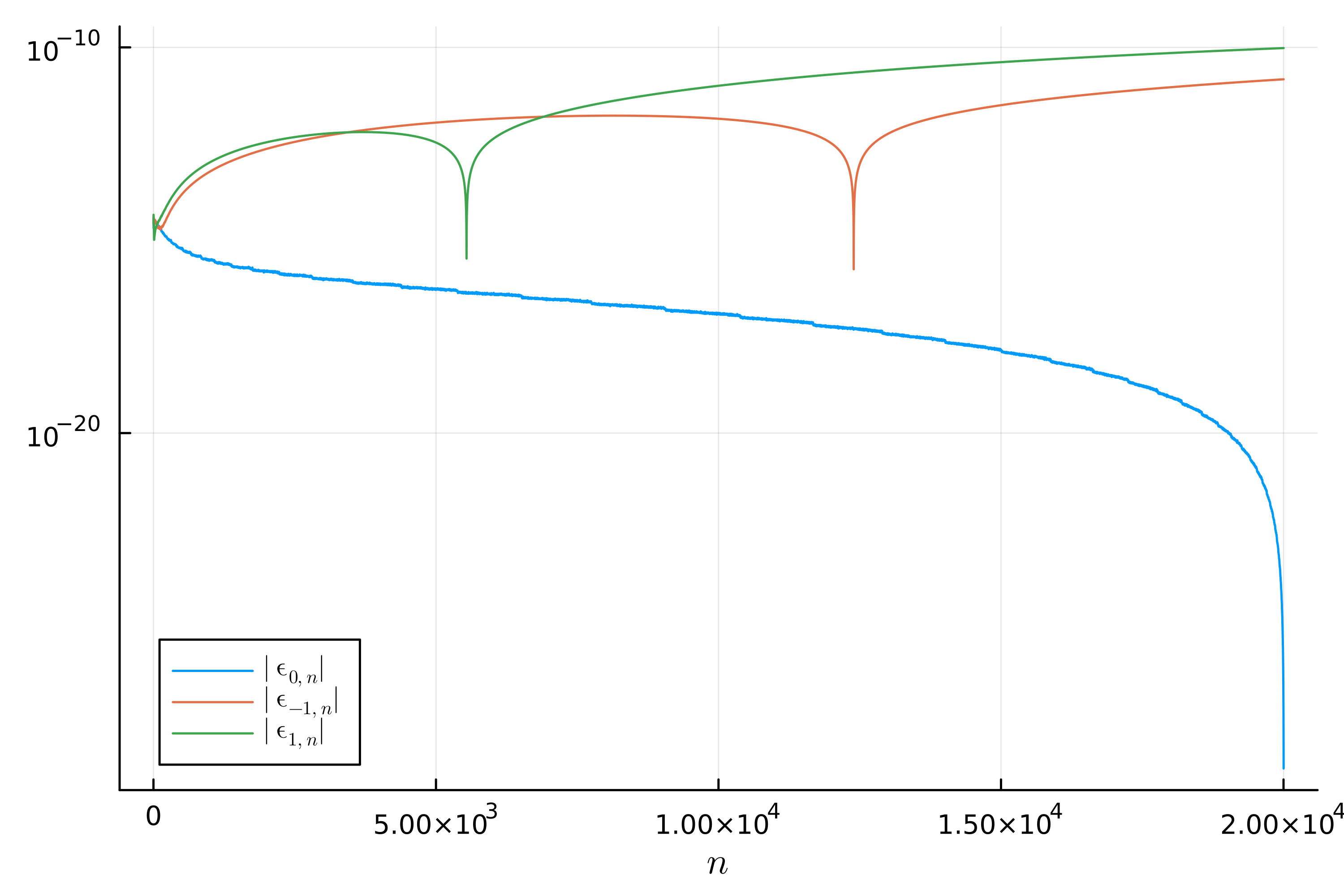} 
\caption{Plot of $\epsilon_{m,n}$ for $m = -1,0,1$ at $\beta = 0$\label{fig:eps}.}
\end{subfigure}
\begin{subfigure}{0.5\textwidth}
\includegraphics[height=0.65\linewidth]{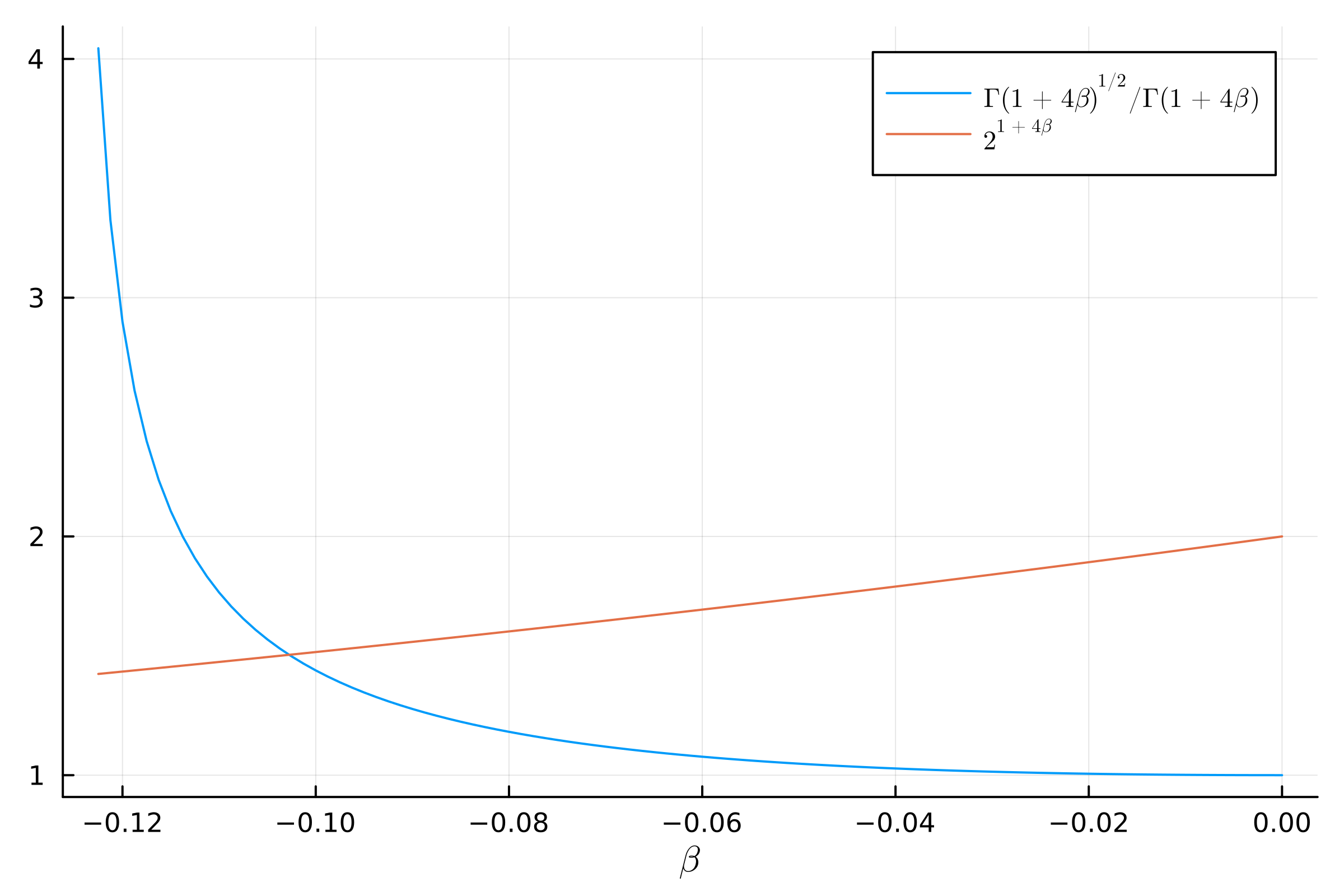} 
\caption{The two different constants\label{fig:constants}.}
\end{subfigure}

\caption{Some of the terms involved in the error bounds~\eqref{eqn:l1-est} and~\eqref{eqn:l2-est}.}
\end{figure}

The above $\ell^2$ bound is limited to $\beta>-1/8$, but a similar idea can in fact be used to go past this barrier, up to using a different basis. For any $\alpha>-1$, we consider $\left\{l^{(\alpha)}_n\right\}_{n\in\N} = \left\{L^{(\alpha)}_n/\sqrt{\Gamma(n+\alpha+1)/n!}\right\}_{n\in\N}$, which forms an orthonormal with respect to $z^{\alpha}e^{-z}\d z$. We can use this basis in place of $\left\{l^{(0)}_n\right\}_{n\in\N} = \left\{L_n\right\}_{n\in\N}$ in $\mathcal{H}_{M,N}$, first in order to numerically find the approximate solution $\bar{u}$, and then to express $Q_\beta-\bar{Q}_\beta$:
$$(Q_\beta - \bar{Q}_\beta)(r,\psi) = g_0(br^2) +\sum_{m = 1}^{M+1}(g_m(br^2)\cos{m\psi} +g_{-m}(br^2)\sin{m\psi}),\qquad g_m =\sum_{n=0}^{N+1}\epsilon_{mn}l^{(\alpha)}_n.$$
For any $\alpha>-1$, the above $\ell^2$ bound now generalises for all $\beta>(\alpha-1)/8$, and yields
\begin{align}
    |\lambda_{\beta}-\bar{\lambda}_{\beta}|&\leq \frac{b^{-1-4\beta}}{2Z_{\beta}}\sum_{|m|\leq M+1}\int_{0}^{\infty}|g_m(z)|z^{4\beta}e^{-z}\d z\qquad \mbox{with $z = br^2$}\nonumber\\
    &\leq \frac{1}{\Gamma(1+4\beta)}\sum_{|m|\leq M+1}\left\{\int_{0}^{\infty}|g_m(z)|^2z^{\alpha}e^{-z}\d z\right\}^{1/2}\left\{\int_{0}^{\infty}z^{8\beta-\alpha}e^{-z}\d z\right\}^{1/2}\qquad \mbox{by Cauchy--Schwarz}\nonumber\\
    &\leq \frac{\Gamma(1+8\beta-\alpha)^{1/2}}{\Gamma(1+4\beta)}\sum_{|m|\leq M+1}\left\{\int_{0}^{\infty}|g_m(z)|^2z^{\alpha}e^{-z}\d z\right\}^{1/2}\nonumber\\
    &\leq \frac{\Gamma(1+8\beta-\alpha)^{1/2}}{\Gamma(1+4\beta)}\sum_{|m|\leq M+1}\left\{\sum_{n=0}^{N+1}\epsilon_{mn}^2\right\}^{1/2}.\label{eqn:l2-est-alpha}
\end{align}
A natural choice would be to take $\alpha = 4\beta$ but this would make the continuation with respect to $\beta$ more intricate, as it would make the basis functions depend on $\beta$. Instead, we split the interval $[-0.25,0]$ into three subintervals $[\underline{\beta},\bar{\beta}]$, and use a fixed value of $\alpha$ on each subinterval.
The output is summarised in the first three rows of Table~\ref{tab:bounds}, and yields Theorem~\ref{thm:intro-Duffing}. In order to obtain Theorem~\ref{thm:Duffing-average-sign}, we refine our results by focusing close to the values of $\beta$ for which $\lambda_\beta$ vanishes, as reported in the two last rows of Table~\ref{tab:bounds}.



\begin{table}[H]
\centering
\begin{tabular}{|c|c|c|c|c|}
    \hline
    $[\underline{\beta},\bar{\beta}]$ & Estimate & $\alpha$ & $\delta> |\lambda_{\beta}-\bar{\lambda}_{\beta}|$ & File at~\cite{Huggzz/Enclosure-of-Lyapunov-exponents}\\
    \hline
    $[-0.25, -0.18]$ & \eqref{eqn:l1-est} & $0$ & $1.3\times10^{-6}$ & \texttt{Duffing/proof\_1}\\
    \hline
    $[-0.18, -0.11]$ & \eqref{eqn:l2-est-alpha} & $-1/2$ & $9.4\times10^{-7}$ &\texttt{Duffing/proof\_2}\\
    \hline
    $[-0.11, 0]$ & \eqref{eqn:l2-est} & $0$ & $2.2\times10^{-6}$&\texttt{Duffing/proof\_3}\\
    \hline
    $[-0.25, -0.249]$ & \eqref{eqn:l1-est} & $0$ & $1.6\times10^{-7}$ & \texttt{Duffing/proof\_1bis}\\
    \hline
    $[-0.16, -0.159]$ & \eqref{eqn:l2-est-alpha} & $-1/2$ & $1.4\times10^{-7}$&\texttt{Duffing/proof\_2bis}\\
    \hline
\end{tabular}
\captionsetup{justification=centering}
\caption{Error bounds resulting from the continuation with respect to the parameter $\beta$\\obtained with $M = 200$ and $N = 20,000$ which Theorems~\ref{thm:intro-Duffing} and~\ref{thm:Duffing-average-sign} rely upon. \label{tab:bounds}}
\end{table}

\section{Outlook}
\label{sec:outlook}



\subsection{Application to more general random dynamical systems}
It turns out that the adjoint method approach can be applied to other Markov processes such as jump diffusions and discrete-time Markov chains~\cite{Glynn2008BoundingProcesses}. For instance, consider the ``double-tent map'' with additive noise on $[0,1)$

$$x_{n+1} = T(x_n) +\xi_n \qquad \mod 1,$$
where the $\xi_n$'s are independent identically distributed uniform random variables on $[-\varepsilon, \varepsilon]$ and

\begin{equation}\label{eqn:Giu_T-def}
T(x) = \begin{cases}
    x/3 \qquad &\mbox{for $0\leq x<1/6$,}\\
    1/9 - x/3 \qquad &\mbox{for $1/6\leq x<1/3$,}\\
    3x-1 \qquad &\mbox{for $1/3\leq x<2/3$,}\\
    3 - 3x\qquad &\mbox{for $2/3\leq x<1$.}\\
\end{cases}
\end{equation}

We chose this example as it fulfils the assumptions of~\cite{Lamb2025NonHorseshoe}. This system clearly displays a negative Lyapunov exponent $\lambda_{\varepsilon}$ for a small noise $\varepsilon$ but one can show~\cite{Lamb2025HorseshoesMaps} that there exists ${\varepsilon_0}<1/2$ such that for all $\varepsilon > \varepsilon_0$, $\lambda_{\varepsilon}>0$. The techniques of~\cite{Lamb2025HorseshoesMaps} actually provide an upper bound for $\varepsilon_0$. Our technique should however give a sharper estimate for $\varepsilon_0$. In this case, the Lyapunov exponent can be written as 

$$\lambda_{\varepsilon} = \int_0^1 \log{|T'(x)|}\mu(\d x),$$
where $\mu$ is the unique stationary distribution of the Markov process $(x_n)_{n\in \mathbb{N}}$. Now consider the Markov semi-group $(\mathcal{P}^n)_{n\in \mathbb{N}}$ associated to $(x_n)_{n\in \mathbb{N}}$

$$\mathcal{P}u(x) = \mathbb{E}\left[u(x_1)\mid x_0 = x\right],$$
for all $u\in [0,1) \to \mathbb{R}$ bounded-measurable. We have that for all such $u$

$$\int_0^1 (\mathcal{P}u)\d \mu = \int_0^1 u \d (\mathcal{P}^*\mu) = \int_0^1 u \d \mu,$$
and thus,

$$\int_0^1 (\mathcal{P}u- u)\d \mu = 0.$$
We therefore seek to solve the Poisson problem

\begin{equation}\label{eqn:Giu_Poisson}
(\mathcal{P}- \mathrm{id})u = \log|T'| - \lambda_{\varepsilon}.
\end{equation}
We thus consider for some large $N\in \N$, a piecewise constant numerical guess $\bar{u}$ for the above Poisson problem

$$\bar{u}(x) = \sum_{n = 1 }^N \bar{u}_i \,\Ind{[(i-1)/N, i/N)}(x).$$

It turns out that directly solving this system led to numerical instabilities, as $(\mathcal{P}-\mathrm{id})$ is non-invertible and dense, and its kernel is not immediately straightforward to mod out. We thus instead used the formula~\cite{Glynn1996AEquation}

$$u(x) = -\sum_{n = 0}^{\infty}\mathbb{E}\left[\log|T'(x_n)|-\lambda_{\varepsilon}\mid x_0 = x\right] = -\sum_{n=0}^{\infty}(\mathcal{P}^n \log|T'|-\lambda_{\varepsilon})(x).$$
By iterating $\mathcal{P}$ enough times, we find $\bar{u}$ represented in Figure~\ref{fig:Giu_ubar} for $\varepsilon = 1/4$.
\begin{figure}[h]
\begin{subfigure}{0.5\textwidth}
\includegraphics[width=0.9\linewidth]{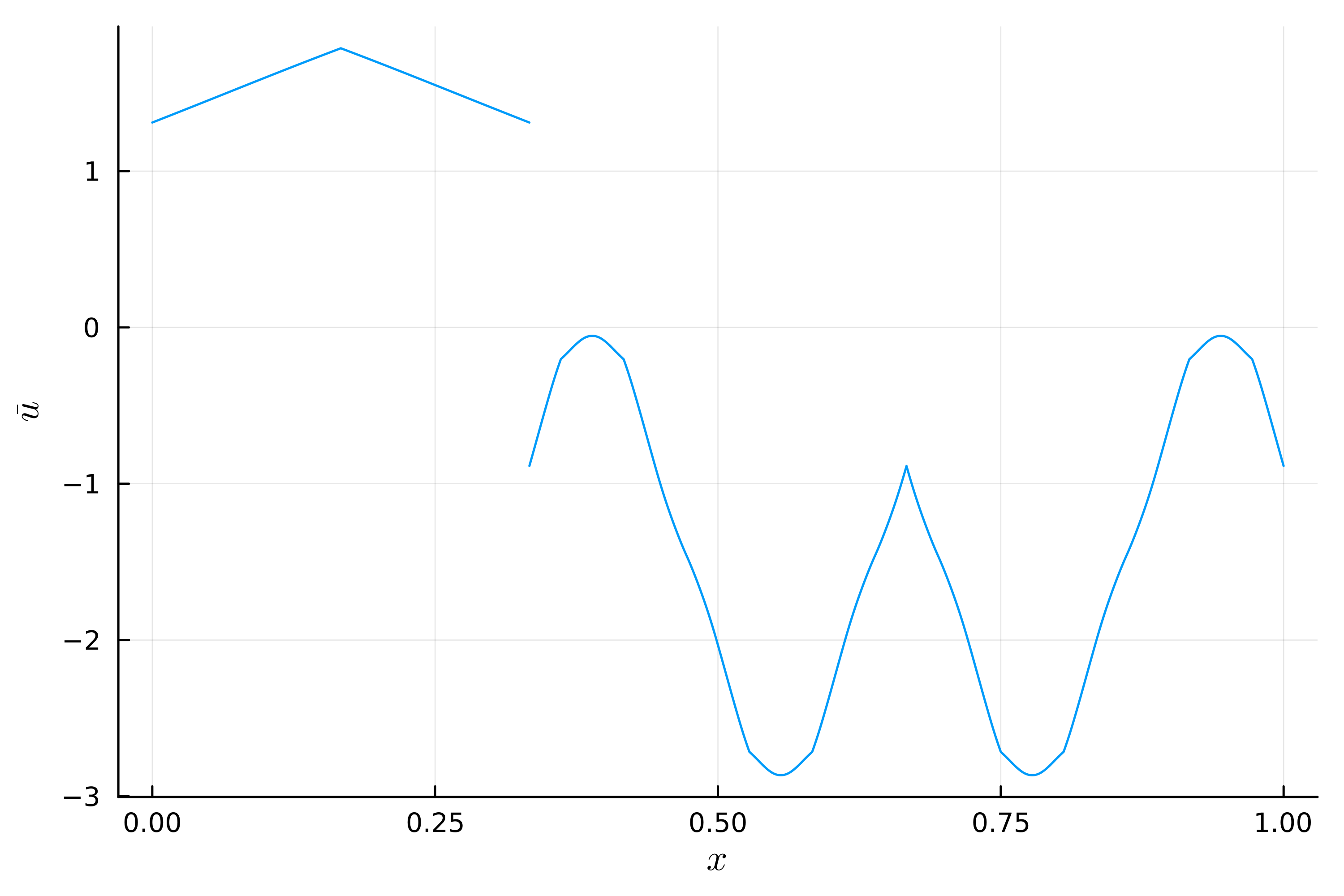}
\captionsetup{justification=centering}
\caption{Plot of a numerical solution $\bar{u}$ to the\\Poisson problem~\eqref{eqn:Giu_Poisson}.}
\label{fig:Giu_ubar}
\end{subfigure}
\begin{subfigure}{0.5\textwidth}
\includegraphics[width=0.9\linewidth]{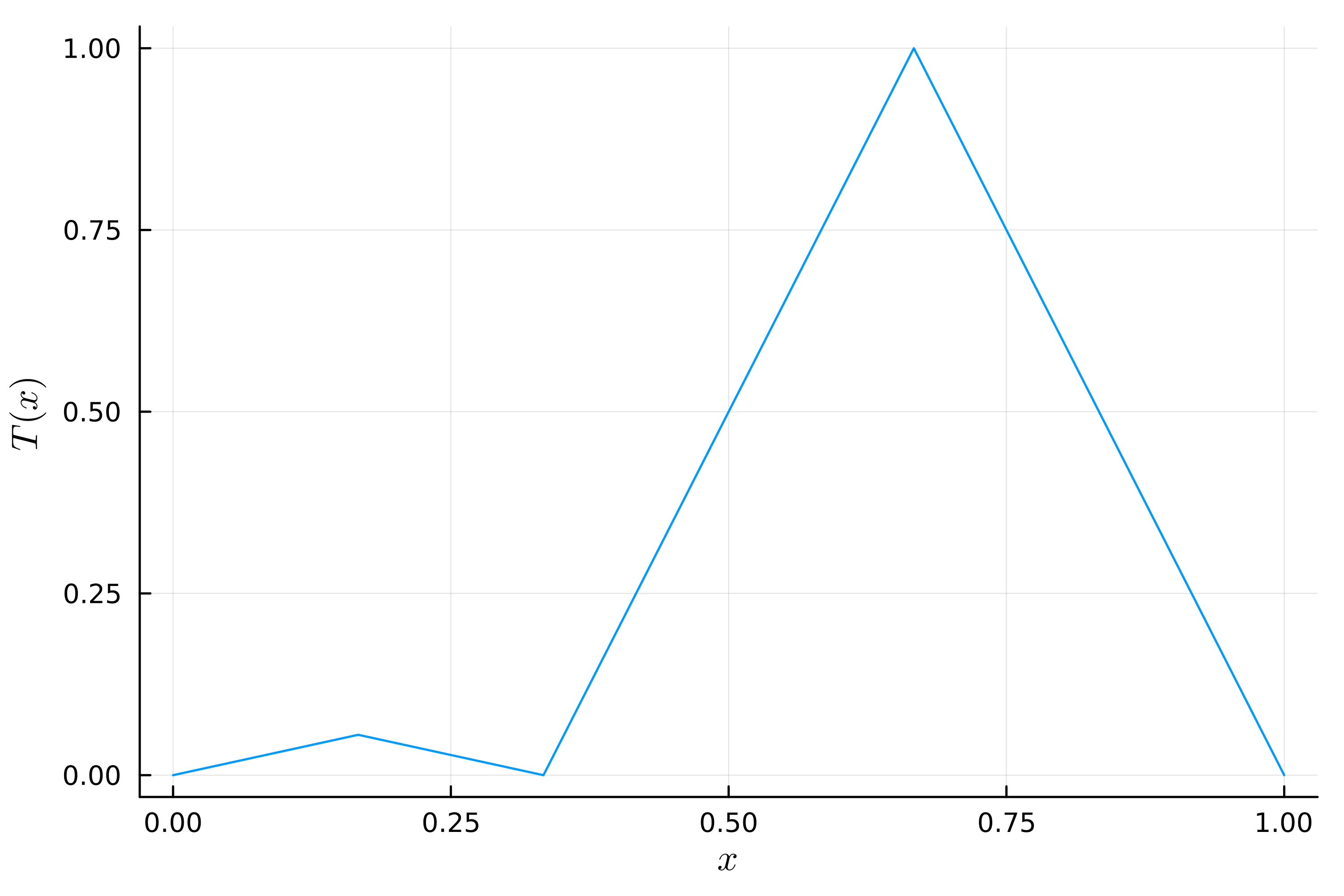} 
\caption{Plot of $T:[0,1)\rightarrow [0,1)$ as\\defined in~\eqref{eqn:Giu_T-def}.}
\label{fig:Giui_T}
\end{subfigure}
\caption{Plots of the functions $\bar{u}$ and $T$.}
\end{figure}

Thus, for the parameter $\varepsilon = 1/4$, we find that

$$\lambda_{\varepsilon} = 0.095547\pm 1.4\times 10^{-4}>0.$$

Furthermore, the results of~\cite{Lamb2025NonHorseshoe} imply that for this parameter, the random dynamical system displays a random Young tower and a random horseshoe.

Note that this approach is quite simplistic and one would expect much better results from a spectral method~\cite{Wormell2019SpectralDynamics}, by splitting the domain into several pieces and using Chebyshev interpolation on each subinterval for instance.

While this system can also be treated with existing methods (such as the Ulam method~\cite{Chihara2022ExistenceMaps, Galatolo2020ExistenceProof})
and this naive approach simply in this form cannot at this stage handle critical points of $T$, we believe that it provides a more realistic framework to treat higher-dimensional systems such as the Hénon map with bounded noise (see Remark~\ref{rmk:mat-vec}). 

\subsection{Towards the central limit theorem}  
If one is interested in the study of the fluctuations of the finite-time Lyapunov exponents
$$\lambda_t(\omega, x, v) = \frac{1}{t}\log\|D\varphi_t(\omega, x)v\|,$$
then one can for instance use the central limit theorem~\cite[Section 7]{Arnold1986LyapunovSystems}
$$\sqrt{t}(\lambda - \lambda_t)\xrightarrow[t\to\infty]{\mathrm{law}}\mathcal{N}(0,\sigma^2_{\lambda}),$$
where
$$\sigma^2_{\lambda} = -2\int(Q-\lambda)u\d\tilde{\mu},$$
and $\mathcal{L}u = Q - \lambda$ (see also~\cite{ Baxendale1988LargeDiffeomorphisms, Baxendale2025InPreparation, Bhattacharya1982OnProcesses}). Note that we then have
$$\left|\sigma^2_{\lambda} - 2\int(\lambda - Q)\bar{u}\d\tilde{\mu}\right|\leq 2\left|\int(Q - \lambda)(u-\bar{u})\d\tilde{\mu}\right|,$$
where the integral of the left-hand side can in principle be enclosed with the techniques of the present paper (provided we first have obtained bounds on $\lambda$). Now, for the right-hand side (assuming without loss of generality that $\tilde{\mu}(u) = \tilde{\mu}(\bar{u}) = 0$), for a suitable Lyapunov function $W$, one can bound the operator $\mathcal{L}^{-1}$~\cite[Theorem 2.3]{Glynn1996AEquation} (see also~\cite{Pardoux2005On3}): there exists a constant $C>0$ (depending only on $\mathcal{L}$ and $W$) such that

$$\left\|\frac{(u-\bar{u})}{W}\right\|_{\infty}\leq C\left\|\frac{(Q-\bar{Q})}{W}\right\|_{\infty},$$
where $\|(Q-\bar{Q})/W\|_{\infty}$ is already known. In the case of diffusions, an estimate for $C$ seems difficult to obtain (even with computer-assisted techniques) as its existence usually relies on a non-completely constructive argument (typically in the proof of the so-called coupling property). 
An exception is in the case of uniformly elliptic diffusions (such as in Section~\ref{sec:hopf}): one can in principle use~\cite{Bogachev2018TheDiffusions} for the construction of such a bound which is however intricate and beyond the scope of this work. In the case of a uniformly elliptic diffusion on a bounded domain and in particular on the torus $\mathbb{T}^{2d-1}$, one can in principle enclose $\bar{u}$ with respect to the $L^{\infty}$-norm~\cite{Nakao2019NumericalEquations}. We additionally mention~\cite{Glynn2025ComputableChains,Herve2025ComputableKernels} which gives computable bounds in the case of discrete-time Markov chains, and~\cite{BlessingNeamtu2025DetectingFunctions} where a computer-assisted enclosure for $\sigma_\lambda^2$ is obtained for a very specific SDE example.


\section*{Acknowledgments}
HC thanks Dennis Chemnitz, Luiz San Martin, Giuseppe Tenaglia and Luca Ziviani for valuable discussions. HC is particularly grateful to Michele Coti Zelati for providing some feedback on an initial version of this work. MB and HC thank Xue-Mei Li for her hospitality and helpful discussions.

MB and HC thank the \textit{Centre de recherches mathématiques}, Montréal for its hospitality during the thematic program \textit{Computational Dynamics: Analysis, Topology and Data} where this work was completed. HC is grateful to G-Research and the Doris Chen mobility fund for making this visit possible. HC was supported by a scholarship from the Imperial College London EPSRC DTP in Mathematical Sciences (EP/W523872/1). HC and JL were additionally supported by the EPSRC Centre for Doctoral Training in Mathematics of Random Systems: Analysis, Modelling and Simulation (EP/S023925/1).
MB, HC and MR thank the CNRS--Imperial \textit{Abraham de Moivre} IRL for supporting this research.
MB was supported by the ANR project CAPPS: ANR-23-CE40-0004-01.

\printbibliography[heading=bibintoc]

\end{document}